\documentclass{article}

\usepackage{amsmath}
\usepackage{amssymb} 
\usepackage{amsfonts} 
\usepackage{bm}
\usepackage{graphicx}
\usepackage{booktabs}
\usepackage[margin=1in]{geometry}
\usepackage{natbib}
\usepackage{amsthm}
\usepackage[colorlinks]{hyperref}
\hypersetup{citecolor=blue}
\usepackage{authblk} 
\usepackage{bbm}
\usepackage{xcolor}

\usepackage[capitalize]{cleveref}
\crefname{assum}{assumption}{assumptions}
\Crefname{assumption}{Assumption}{Assumptions}

\newtheorem{definition}{Definition}
\newtheorem{theorem}{Theorem}
\newtheorem{lemma}{Lemma}
\newtheorem{proposition}{Proposition}
\newtheorem{corollary}{Corollary}

\newtheorem{example}{Example}
\newtheorem{bound}{Bound}
\newtheorem{problem}{Problem}
\newtheorem*{theorem*}{Theorem}
\newtheorem*{lemma*}{Lemma}
\newtheorem*{proposition*}{Proposition}
\newtheorem*{corollary*}{Corollary}
\newtheorem*{example*}{Example}
\newtheorem{bound*}{Bound}
\newtheorem{problem*}{Problem}
\newtheorem{assumption}{Assumption}

\newcommand{\mynull}{\operatorname{null}}
\newcommand{\range}{\operatorname{range}}
\newcommand{\diag}{\operatorname{diag}}
\newcommand{\Res}{\operatorname{Res}}
\newcommand{\sign}{\operatorname{sign}}
\newcommand{\supp}{\operatorname{supp}}

\newcommand{\dist}{\operatorname{dist}}

\newcommand{\uargmin}{\operatorname*{argmin}}

\title{Explicit Bounds on Hoffman Constants for\\ Sparse Polyhedral Regularizers}

\author[1]{Faris Chaudhry}
\author[1]{Anthea Monod}
\author[2]{Keisuke Yano}
\affil[1]{Imperial College London, UK\\ \texttt{\{fc522, a.monod\}@imperial.ac.uk}}
\affil[2]{Institute of Statistical Mathematics, Japan\\ \texttt{yano@ism.ac.jp}}

\date{}

\begin{document}

\maketitle

\begin{abstract}
Hoffman constants quantify the distance to the solution set of a polyhedral system in terms of its residual. They control error bounds, geometric conditioning, the convergence rate of first-order iterative algorithms, and the stability of regression estimators to a perturbation in the target. Despite their importance, explicit values of Hoffman constants are rarely available beyond simple settings due to their computational complexity. In this paper, we derive explicit bounds on the Hoffman constant for a class of structured polyhedral systems arising in sparse regularization. We show that the constants admit a decomposition into spectral, geometric, and interaction components, which together determine their scaling behavior. This framework explains the gap between worst-case dimension-dependent bounds and the favorable conditioning observed empirically in structured problems. As a primary application, we analyze $\ell_1$-regularized least squares, demonstrating that the effective constants depend on low-dimensional active subspaces rather than the ambient dimension. We extend this analysis to the elastic net and adaptive $\ell_1$-regularized formulations, demonstrating how the underlying optimization geometry is altered by the $\ell_2$ penalty and reweighting scheme. Furthermore, we relate our results to classical statistical assumptions such as the restricted eigenvalue and isometry conditions. We show that the restricted eigenvalue condition bounds only the primal curvature of the active set, leaving the dual interaction geometry unconstrained and is thus insufficient alone for fast theoretical convergence rates. We further show standard $\ell_\infty$ coherence conditions do not yield dimension-independent Euclidean bounds to close this gap, which motivates the analysis of Hoffman bounds in non-Euclidean norms. Our results show that Hoffman constants, typically treated as implicit quantities, admit explicit and interpretable forms in structured regression settings. For practitioners, this geometric framework provides actionable tools for post-hoc feature pruning and a priori optimal hyperparameter selection.
\end{abstract}

\section{Introduction}

Hoffman constants are a classical and central quantity in the variational analysis of polyhedral systems~\citep{hoffman1952approximate}. Fundamentally, they quantify how far an arbitrary point lies from the solution set of a linear system in terms of the magnitude of its constraint violation. In modern data contexts, they provide global error bounds for polyhedral feasibility problems and underlie a wide range of results concerning metric regularity, stability of solution mappings~\citep{robinson1981polyhedral}, sensitivity analysis, and linear convergence guarantees for first-order methods~\citep{leventhal2008randomized,bolte2017errorbounds, drusvyatskiy2018errorbounds,chaudhry2026trajectory}. Given these connections, the Hoffman constant of a system is a quantitative measure of geometric conditioning.

Despite their importance, Hoffman constants are known to be extremely difficult to compute explicitly. For purely equality-constrained systems, they reduce to familiar spectral quantities such as reciprocals of singular values~\citep{leventhal2008randomized}. However, once inequalities are introduced, the geometry becomes substantially more complex. Mixed systems depend on which constraints are active, the corresponding normals interact, and how affine and polyhedral components intersect. In general, no simple closed-form expression is available, and most existing results are either existential, worst-case, or expressed through combinatorial characterizations that are difficult and can be NP-hard to evaluate in practice~\citep{foucart2013compressivesensing, freund1999some, pena2000understanding}. As a result, Hoffman constants are often invoked abstractly, rather than used as concrete computable and interpretable quantities. This gap is particularly relevant in structured high-dimensional optimization problems, where empirical behavior is often far better than pessimistic worst-case theory would suggest. A canonical example is the least absolute shrinkage and selection operator (LASSO)~\citep{tibshirani1996lasso} and related sparse regularization methods. These problems combine a smooth quadratic loss with polyhedral regularization, producing Karush--Kuhn--Tucker (KKT) systems with highly structured linear equality and inequality constraints. Such systems are ubiquitous in statistics, signal processing, and machine learning, yet their geometric conditioning remains poorly understood at an explicit level.

In this paper, we show that for a class of structured polyhedral systems used in sparse regularization, Hoffman bounds can in fact be derived, decomposed, and interpreted explicitly. We focus on systems whose constraints admit a natural separation into affine, sign, and box components, with the LASSO KKT system as a central motivating example. Our analysis reveals that the associated Hoffman constants decompose into three distinct mechanisms: (i) an \emph{affine contribution} arising from the conditioning of the design matrix and hyperparameters; (ii) a \emph{polyhedral contribution} induced by the constraint set; and (iii) an \emph{interaction contribution} capturing the geometry between the affine space and polyhedral constraint set. Our perspective yields transparent bounds, clarifies dimension-dependent scaling, and explains how favorable local structure can coexist with poor ambient worst-case conditioning. 

For practitioners, our bounds can provide quantitative tools to (i) check for poor optimization conditioning which causes certain optimization algorithms to stall, (ii) prune inactive features to improve the optimization space as a post-hoc procedure similar to existing tilting methods~\citep{cho2012tiltingfeaturescreening}, (iii) calibrate KKT-based optimization stopping criteria by translating them to the distance to the optimal solution set, and (iv) assess the stability of sparse regression solutions under perturbation to the target variable, such as those caused by cross-validation.

In the setting of $\ell_1$-regularized least squares, our framework shows that effective conditioning is governed not simply by ambient dimension, but by active-set curvature and active-inactive feature interactions. This perspective reveals a fundamental divide between statistical and algorithmic guarantees: while curvature in the active feature space controls the statistical estimation error in the \textit{primal} problem, the interaction of active and inactive features dictates a form of a \textit{dual} optimization residual that governs the geometric conditioning of the KKT system. Because classical high-dimensional statistics is primarily focused on statistical estimation error and feature selection, its standard conditions can leave the optimization geometry largely unconstrained. For example, the restricted eigenvalue property bounds only the active curvature, ignoring the dual optimization geometry; meanwhile, incoherence conditions control dual interactions in the $\ell_\infty$ norm but fail to constrain the Euclidean convergence rate. By formally bridging this gap, we demonstrate how the restricted conditions of high-dimensional statistics can be repurposed to derive uniform, dimension-independent Hoffman bounds for optimization.

The remainder of this paper is organized as follows. \cref{sec:background} introduces common background such as sparse regularizers, the global and restricted Hoffman constant, and standard high-dimensional statistical conditions. It also provides the literature review. \cref{sec:explicit-hc} develops explicit Hoffman bounds for general structured polyhedral systems through a decomposition argument. \cref{sec:application-lasso} specializes this general theory for the $\ell_1$-regularized least squares system (i.e., the LASSO). \cref{sec:restricted-statistical-conditions} connects these optimization geometries to high-dimensional statistical theory. We prove a hierarchy of geometric stability: the restricted nullspace property guarantees invertibility of the KKT system; the restricted eigenvalue condition bounds the primal active curvature but leaves the active-inactive interaction unconstrained; and we show that no standard $\ell_\infty$ incoherence condition---including the mutual and entrywise coherence properties---can close this gap in the Euclidean norm, as any translation from $\ell_\infty$ or sparse spectral control to the $\ell_2$ operator norm accumulates an unavoidable ambient dimensional penalty. This establishes a fundamental obstruction of Euclidean geometry, and motivates non-Euclidean geometries. \cref{sec:hoffman-other-norms} examines how these geometric bounds scale in non-Euclidean ($\ell_1$ and $\ell_\infty$) norms along with a discussion of fundamental sample complexity limits. \cref{sec:application-other} demonstrates the framework's broader utility by analyzing how algorithmic interventions, such as the elastic net and adaptive LASSO, alter the underlying optimization geometry; in doing so, we argue that the prototypical decomposition argument can be used for a wide variety of sparse regularization problems. Here, optimizing the variational Hoffman bound can be used for optimal hyperparameter selection. Finally, \cref{sec:discussion} concludes with broader implications and future directions.

\section{Background and Related Work}
\label{sec:background}

\subsection{Sparse Polyhedral Regularizers}
\label{subsec:sparse-regularizers-background}

In high-dimensional estimation problems where the number of features $d$ vastly exceeds the sample size of observations $n$ (that is, $d \gg n$), classical unpenalized models like ordinary least-squares (OLS) are ill-posed. To obtain meaningful estimators, it is standard practice to assume that the true underlying signal is sparse. The most direct approach to estimating such a signal is to penalize the $\ell_0$ pseudonorm, which counts the number of nonzero coefficients.
\begin{problem}[The $\ell_0$-regularized least squares estimator; \citet{foucart2013compressivesensing,bertsimas2016best}]
    \begin{equation}
    \hat{\beta} = \uargmin_{\beta \in \mathbb{R}^d} \frac{1}{2n} \|A\beta - y\|_2^2 + \lambda_n \|\beta\|_0,
\end{equation}
where $A \in \mathbb{R}^{n \times d}$ is the design matrix, $y \in \mathbb{R}^n$ is the target vector, and $\lambda_n \in \mathbb{R}_+$ is the regularization parameter governing the penalty strength. 
\end{problem}
However, optimizing the $\ell_0$ penalty is a combinatorial search problem that is known to be NP-hard~\citep{foucart2013compressivesensing}.

To achieve a computationally tractable objective, the $\ell_0$ penalty is replaced by its convex relaxation, the $\ell_1$ norm. This yields a Lagrangian form of the $\ell_1$-regularized least squares problem, better known as the LASSO.
\begin{problem}[The LASSO estimator; \citet{tibshirani1996lasso}]
    \label{problem:lasso}
    \begin{equation}
    \hat{\beta} = \uargmin_{\beta \in \mathbb{R}^d} \frac{1}{2n} \|A\beta - y\|_2^2 + \lambda_n \|\beta\|_1.
\end{equation}
\end{problem}
The polyhedral geometry of the $\ell_1$ penalty naturally promotes feature selection and sparsity.

Let $[d] := \{1, \dots, d\}$. The intrinsic sparsity of some vector of coefficients $\beta \in \mathbb{R}^d$ is captured by the \emph{active set} (or support):
\begin{equation}
    \mathcal{A}(\beta)= \{j \in [d] : \beta_j \neq 0\} = \supp(\beta),
\end{equation}
with a sparsity level given by $s(\beta) := |\mathcal{A}(\beta)|$. A vector $\beta$ is called $s$-sparse if it has at most $s$ nonzero entries, or formally $\|\beta\|_0 \le s$. The identified active set of the optimal estimator $\hat{\beta}$ is intrinsically low-dimensional, generally satisfying $\hat{s} :=s(\hat{\beta}) \ll n \ll d$. The corresponding subspace of vectors supported on $\mathcal{A}$ is denoted by
\begin{equation}
    \mathcal{S}_{\mathcal{A}} := \{\beta \in \mathbb{R}^d : \beta_j = 0 \text{ for all } j \notin \mathcal{A}\}.
\end{equation}
Note that there are $2^d$ total possible active sets in the ambient space, but only $\sum_{k=0}^s \binom{d}{k}$ possible $s$-sparse active sets, which greatly reduces the combinatorial search space when $s \ll d$.

In the context of statistics, it is important to distinguish between the optimal LASSO estimator $\hat{\beta}$ and the true underlying signal $\beta^*$ which is assumed to generate the target; namely, $y = A \beta^* + \epsilon$ and $\epsilon$ is some noising distribution. Their active sets are denoted by $\hat{\mathcal{A}}$ and $\mathcal{A}^*$ respectively. The difference between these two vectors, $\hat{\beta} - \beta^*$, is referred to as the statistical error. In \cref{subsec:sparse-recovery-background}, we provide background on the conditions used to guarantee recovery of $\mathcal{A}^*$ and good estimation of $\beta^*$ from the high-dimensional statistical literature.

To compute the estimator $\hat{\beta}$ in practice, a common algorithmic choice is the proximal gradient method~\citep{parikh2014proximal}. For LASSO, the algorithm is known as the iterative shrinkage-thresholding algorithm (ISTA)~\citep{daubechies2004ista}, and its accelerated variant fast ISTA (FISTA)~\citep{beck2009fista}. Most recently, the iterative dynamics of these solvers have inspired algorithm unfolding techniques, leading to neural network architectures like the learned ISTA (LISTA)~\citep{gregor2010lista, chen2021lista}. These models attempt to amortize an approximation of the sparse estimator map by learning algorithm parameters directly.

\subsection{The Global Hoffman Constant}
\label{subsec:hoffman-background}

The geometric conditioning of polyhedral sets is classically quantified by the Hoffman constant~\citep{hoffman1952approximate}. It is a foundational tool that bounds the distance from any point to a polyhedral feasible set by a constant multiple of the residual of the inequalities defining that set. In the broader context of variational analysis, the Hoffman bound serves as the exact polyhedral instantiation of a global error bound or metric subregularity condition and always exists~\citep{robinson1981polyhedral, rockafellar1998variational}. Equivalently, it establishes the global upper Lipschitz continuity of the corresponding polyhedral set-valued inverse mapping with respect to additive perturbations.

\begin{definition}[Hoffman constant; \citet{hoffman1952approximate}]
\label{def:hoffman}
For a nonempty polyhedral set $\mathcal{X} = \{x \in \mathbb{R}^d : Gx = h, \; Mx \le r\}$ and given norms $\|\cdot\|_V$ on the domain and $\|\cdot\|_W$ on the residual space, the \emph{Hoffman constant} $H^{V, W}$ is the smallest constant such that
\begin{equation}
    \dist_V(x, \mathcal{X})
    :=\min_{x'\in\mathcal{X}}\|x-x'\|_V
    \le H^{V, W} 
    \underbrace{
    \left\|
    \begin{pmatrix}
    Gx - h \\
    {[Mx - r]}_+
    \end{pmatrix}
    \right\|_W
    }_{:= \Res_W(x, \mathcal{X})}
    \quad \text{for all}\; x \in \mathbb{R}^d,
\end{equation}
where $[\cdot]_+ = \max(0, \cdot)$ denotes the element-wise positive part, and $\Res_W(x, \mathcal{X})$ defines the scalar residual of the algebraic system defining $\mathcal{X}$ evaluated under the norm $\|\cdot\|_W$. 
\end{definition}
The most common scenario for analyzing the Hoffman constant is the Euclidean setting when both $V$ and $W$ are the $\ell_2$ norm with the corresponding Hoffman constant $H^{2,2}$. In this case, the global Hoffman constant for a simple system of linear equalities $Mx=r$ admits a direct spectral characterization. It is determined by the reciprocal of the smallest positive singular value $\sigma^+_{\min}(M)$ of the defining matrix~\citep{leventhal2008randomized}:
\begin{equation}
    H^{2, 2} = \frac{1}{\sigma_{\min}^{+}(M)} = \|M^{\dagger}\|_{2},
\end{equation}
where $M^{\dagger}$ is the pseudoinverse.

Hoffman constants for polyhedral systems play a significant role in optimization, numerical analysis, and stability analysis:
\begin{itemize}
    \item \textbf{Optimization.} For polyhedral systems, a small Hoffman constant means decreasing constraint violations translate efficiently into progress toward feasibility. Thus, the convergence rate of first-order methods is governed by the Hoffman constant. In particular, for a wide class of proximal gradient and randomized methods, the error contraction factor is shown to be $1 -\mathcal{O}(1/(H^2 L))$~\citep{ leventhal2008randomized} after an error bound or Polyak--{\L}ojasiewicz type condition is established~\citep{karimi2016pl,drusvyatskiy2018errorbounds,chaudhry2026trajectory}.

    It also follows from \cref{def:hoffman} that the Hoffman bound provides a way to translate residuals, which are cheap to calculate, into a reliable guarantee of distance to an optimal solution. This can be used to justify stopping criteria for iterative algorithms.
    \item \textbf{Numerical analysis.} Hoffman-type constants can also be viewed as a type of condition measure in numerical linear algebra~\citep{chaudhry2026trajectory}. In particular, they can be viewed as generalizations of classical matrix condition numbers and are closely connected to other condition measures~\citep{pena2020equivalence} for linear systems such as Renegar's condition number~\citep{renegar2001mathematical} and the Stewart--Todd condition measure~\citep{stewart1989scaled,todd1990dantzig}, which quantify sensitivity of feasibility and optimality to perturbations. 
    \item \textbf{Stability.} Finally, Hoffman constants provide worst-case sensitivity bounds under perturbation of the problem data. Drawing on the fundamental results of~\citet{robinson1981polyhedral} for polyhedral multifunctions, the Hoffman constant serves as an upper Lipschitz modulus of the inverse KKT mapping. For regularized problems like LASSO, this implies that, under local uniqueness of the optimal solution, the distance between optimal solutions $\hat{\beta}(y)$ and $\hat{\beta}(y')$ corresponding to perturbed target vectors $y$ and $y'$ is linearly bounded by the Hoffman constant of the KKT system, taking a form like
    \begin{equation}
        \label{eq:stability-under-perturbation}
        \|\hat{\beta}(y) - \hat{\beta}(y')\|_V \le H^{V, W} \left\| \frac{1}{n} A^\top (y - y') \right\|_W.
    \end{equation}
    Hence, the Hoffman constant can be used to control leave-one-out changes in approximate cross-validation analyses~\citep{auddy2023approximatecv}.
    \item \textbf{Duality theory.} Recent work has uncovered theoretical connections between the conditioning of primal and dual feasibility problems. Under suitable Slater conditions, duality inequalities and identities can be established between the Hoffman constants of paired primal-dual polyhedral systems~\citep{pena2026hoffmanduality}. This perspective emphasizes how Hoffman constants are a geometric measure across dual optimization spaces.
\end{itemize}

While these global bounds have many uses, evaluating them explicitly remains a challenge. A few recent works are particularly relevant in attempting to characterize these geometries. \citet{pena2021hoffman} characterize Hoffman constants relative to reference polyhedra, although this general characterization does not expose how problem-specific structure, such as the design matrix and active support, determines their magnitude. Most closely related to our stability discussion is the work of~\citet{berk2023lasso} which derives Lipschitz bounds for the LASSO solution map under perturbations of the response and regularization parameter. Such output sensitivity analysis is highly relevant to domains like differentially-private LASSO, where bounding the sensitivity of the estimator is required to calibrate noise mechanisms. As noted above via the Robinson continuity property~\citep{robinson1981polyhedral}, a Hoffman bound for the KKT system can also yield a Lipschitz-type sensitivity bound. Nevertheless, the two quantities characterize fundamentally different mappings: the Lipschitz constant measures parameter-to-primal-solution stability, whereas the Hoffman constant measures residual-to-primal-dual-solution stability, thereby capturing the interactions between active and inactive features.

However, relying on the global Hoffman constant introduces a flaw in the high-dimensional sparse regime ($d \gg n$). While the underlying active linear subspaces may be exceptionally well-conditioned, composite regularizers like the LASSO yield mixed polyhedral KKT systems. To bound the global distance to optimality, the worst-case geometry is forced to account for dense ambient constraints across all inactive dimensions. Consequently, the global Hoffman constant degrades pathologically with the ambient dimension, typically scaling as $H \sim \mathcal{O}(\sqrt{d})$ based on the singular values of random design matrices~\citep{chaudhry2026trajectory, rudelson2010singularvalues}.

To circumvent these loose or sometimes infeasible global bounds, the high-dimensional statistics literature relies on localized, structural properties of the design matrix, which we formalize next.

\subsection{Statistical Conditions in Sparse Recovery}
\label{subsec:sparse-recovery-background}

In high-dimensional statistics, successful recovery of the true sparse signal $\beta^*$ requires the design matrix to exhibit favorable geometric properties on sparse subspaces~\citep{wainwright2019highdimensional, vershynin2018highdimensional, foucart2013compressivesensing}. We define the normalized Gram matrix as $G = \frac{1}{n} A^\top A$, and denote its submatrices for a given active set $\mathcal{A}$ by $G_{\mathcal{A} \mathcal{A}}$, $G_{\mathcal{A}^c \mathcal{A}}$, and so forth.

Because we minimize a regularized objective rather than an unconstrained one, the relevant error vectors are not arbitrary directions in $\mathbb{R}^d$. To characterize this geometry, we first define the \emph{restricted cone} for any generic active set $\mathcal{A} \subseteq [d]$ and constant $\alpha \in \mathbb{R}_+$ as the set of directions $\Delta \in \mathbb{R}^d$ where the $\ell_1$ mass off the support is bounded by a multiple of the mass on the support~\citep[Definition 7.21]{wainwright2019highdimensional}:
\begin{equation}
    \mathcal{C}_{\alpha}(\mathcal{A}) = \left\{ \Delta \in \mathbb{R}^d : \|\Delta_{\mathcal{A}^c}\|_1 \le \alpha \|\Delta_{\mathcal{A}}\|_1 \right\}.
    \label{eq: restricted cone}
\end{equation}
Observe that $\mathcal{S}_{\mathcal{A}} \subseteq \mathcal{C}_{\alpha}(\mathcal{A})$ for any $\mathcal{A}$ and $\alpha \in \mathbb{R}_+$ trivially as the inactive component is 0 by construction.

Standard statistical optimality arguments demonstrate that, for a suitably chosen $\lambda_n$, the \emph{statistical error} $\hat{\beta} - \beta^*$ is forced into the restricted cone $\mathcal{C}_{3}(\mathcal{A}^*)$ defined by the true active set~\citep{wainwright2019highdimensional}. Similarly, the \emph{optimization error} $\beta - \hat{\beta}$ for an arbitrary iterate $\beta$ is also funneled into these restricted directions along the algorithmic trajectory~\citep{agarwal2012fast}.

To guarantee that the objective possesses sufficient curvature within the restricted cone, the standard requirement is the restricted eigenvalue (RE) condition~\citep{bickel2009simultaneous, raskutti2010restricted}:
\begin{definition}[Restricted eigenvalue condition; \citet{wainwright2019highdimensional}]
The matrix $A$ satisfies the RE condition over a support $\mathcal{A}$ with parameter $\alpha \in \mathbb{R}_+$ if there exists a constant $\gamma > 0$ such that
\begin{equation}
    \frac{1}{n}\|A\Delta\|_2^2 \ge \gamma \|\Delta\|_2^2, \quad \forall \Delta \in \mathcal{C}_{\alpha}(\mathcal{A}) \setminus \{0\}.
\end{equation}
\end{definition}
Geometrically, when evaluated at the true support $\mathcal{A}^*$, the RE condition bounds the \emph{primal} statistical error by guaranteeing that the smallest eigenvalue of the active Gram submatrix, $\lambda_{\min}(G_{\mathcal{A}^* \mathcal{A}^*})$, is positive to permit curvature over sparse directions. It ensures that, if the prediction error $\frac{1}{n} \|A \Delta\|_2^2$ is small, the Euclidean distance to the true parameter $\|\Delta\|_2^2$ must also be proportionally small, bounded by $1/\gamma$. Results for the RE condition holding with high probability generally give $n \gtrsim s \log d$ as a tight bound~\citep[Theorem 7.16]{wainwright2019highdimensional}.

However, while primal curvature is necessary for parameter estimation, it does not guarantee that the estimator selects the correct features. To guarantee exact recovery of the true support (i.e., $\hat{\mathcal{A}} = \mathcal{A}^*$), the design matrix must control the interaction between active and inactive features.

This requirement is captured by statistical incoherence conditions. The irrepresentability condition (IRC)~\citep{zhao2006model} ensures that the projection of the inactive features onto the active features, weighted by the true signs, $\sign({\beta^*})$, is bounded below 1 in the $\ell_\infty$ norm. A stronger, sign-agnostic variant is the mutual incoherence condition (MIC)~\citep{wainwright2019highdimensional}, which bounds the induced $\ell_\infty$ operator norm directly:
\begin{equation}
    \| G_{(\mathcal{A}^*)^c \mathcal{A}^*} G_{\mathcal{A}^* \mathcal{A}^*}^{-1} \|_{\infty} \le 1 - \eta,
\end{equation}
for some margin $\eta \in (0, 1]$.

Exact support recovery also requires appropriate regularization tuning. First, to prevent false positives (inactive variables entering the model), $\lambda_n$ must be chosen sufficiently large to suppress the projected ambient noise on the inactive set. Second, to prevent false negatives (true signals being shrunk to zero), the LASSO requires a \emph{beta-min condition}. This assumes that the weakest true signal exceeds the maximum $\ell_\infty$ estimation error induced by both the noise $\epsilon$ and the shrinkage bias of $\lambda_n$~\citep[Theorem 7.21]{wainwright2019highdimensional}:
\begin{equation}
    \min_{j \in \mathcal{A}^*} |\beta_j^*| > B(\lambda_n, \epsilon) := \left\| G_{\mathcal{A}^*\mathcal{A}^*}^{-1} \frac{1}{n} A_{\mathcal{A}^*}^\top \epsilon \right\|_\infty + \lambda_n \|G_{\mathcal{A}^*\mathcal{A}^*}^{-1}\|_\infty.
\end{equation}
When the MIC holds and $\lambda_n$ is bounded between the noise floor and the beta-min signal ceiling, the LASSO is statistically consistent and identifies the correct active manifold in finite time~\citep{vaiter2015model}.

When one wishes to uniformly bound both the primal curvature and the incoherence simultaneously, the literature frequently relies on the restricted isometry property (RIP)~\citep{candes2008rip}.
\begin{definition}[Restricted isometry property; \citet{vershynin2018highdimensional}]
A matrix $A$ satisfies the RIP of order $s$ with constant $\delta_s \in (0, 1)$ if for all $s$-sparse vectors $v \in \mathbb{R}^d$:
\begin{equation}
    (1 - \delta_s)\|v\|_2^2 \le \frac{1}{n}\|Av\|_2^2 \le (1 + \delta_s) \|v\|_2^2.
\end{equation}
\end{definition}
The RIP is a powerful spectral condition: the matrix $A$ acts as a near-isometry on sufficiently sparse directions. That is, the eigenvalues of any sparse active Gram submatrix $G_{\mathcal{A}\mathcal{A}}$ lie in the interval $[1-\delta_s, 1+\delta_s]$, smoothly controlling the condition number of the active system. 

Certain random matrices like Gaussian and sub-Gaussian designs satisfy the RIP~\citep{baraniuk2008rip, kasiviswanathan2019rip, blanchard2011rip, lu2019rip}. Beyond unstructured ensembles, some highly structured matrices such as partial circulant or Toeplitz designs, which arise in signal processing, satisfy the RIP~\citep{rauhut2009circulant, rahut2012circulant}. Finally, constructing deterministic matrices with the RIP is an open problem~\citep{bandeira2013rip}.

In providing explicit bounds for Hoffman constants, we will see how and when these restricted statistical conditions provide bounds on the geometric components of the Hoffman-type bounds. For example, it can be shown that the RE alone only controls the primal notion of error (statistical error; curvature of the active set) and not the dual notion of error (interaction between the active and inactive set). It is precisely this lack of focus on the dual notion of error which arises in the geometric analysis, and is the crux of the divide between the optimization and statistical literature.

\subsection{The Restricted Hoffman Constant}
\label{subsec:restricted-hoffman-background}

To reconcile the gap between pessimistic global optimization bounds and the structured, local nature of trajectories observed empirically,~\citet{chaudhry2026trajectory} introduced a restricted geometric framework. The effective conditioning of algorithms like the ISTA is governed by a \emph{restricted Hoffman constant}, evaluated exclusively on the subsets of the domain visited by the trajectory of the algorithm.

\begin{definition}[Restricted Hoffman constant; \citet{chaudhry2026trajectory}]
\label{def:restricted-hoffman}
Let $\mathcal{X} = \{x \in \mathbb{R}^d : Gx = h, \; Mx \le r\}$ be a nonempty polyhedral set. Given a subset $\mathcal{K} \subseteq \mathbb{R}^d$ and norms $\|\cdot\|_V$ and $\|\cdot\|_W$, the $\mathcal{K}$-restricted Hoffman constant, denoted $H^{V, W}_{\mathcal{K}}$, is the smallest constant such that
\begin{equation}
    \dist_V(x, \mathcal{X}) \le H_{\mathcal{K}}^{V, W} \Res_W(x, \mathcal{X})
    \quad \text{for all}\; x \in \mathcal{K}.
\end{equation}
\end{definition}

The utility of \cref{def:restricted-hoffman} relies on identifying subsets $\mathcal{K}$ that accurately reflect the trajectory of the optimization algorithm. In the context of sparse regularization, there is substantial evidence that optimization paths remain confined to highly structured regions of the domain. For instance, homotopy and sequential active-set methods such as least-angle regression (LARS) construct the solution path by explicitly building up the active set step-by-step, meaning their iterates never traverse the dense ambient space~\citep{osborne2000homotopy, efron2004lars}.

For unconstrained first-order methods like the ISTA, the iterates also exhibit trajectory containment. \citet{agarwal2012fast} established that, up to a statistical tolerance, the proximal gradient iterates for the LASSO enter and remain within the restricted cone ($\alpha=3$) plus some affine component; \citet{chaudhry2026trajectory} observed empirically that the LASSO trajectory is contained within the restricted cone with $\alpha=1$. 

Separately, the optimization perspective established the phenomenon of \emph{manifold identification}. For a broad class of nonsmooth problems, including those with polyhedral regularizers such as the $\ell_1$ norm, first-order methods identify an \emph{active manifold} or support in finite time, after which the iterates evolve within a lower-dimensional structure~\citep{hare2004activeset, lee2012manifoldidentification}. This transition induces a qualitative change in the geometry of the problem: once the active set is identified, the optimization effectively proceeds with an optimal active set $\mathcal{{\hat{A}}}$ of the estimator. Under suitable nondegeneracy conditions, such as strict complementarity, this identification occurs after finitely many iterations~\citep{sun2019manifoldidentification, nutini2019activeset}; conversely, if the problem is degenerate, guaranteeing manifold identification in finite time is not guaranteed~\citep{burke1988activeset, lee2012manifoldidentification}.

Motivated by these trajectory properties, we can formalize alternative notions of the Hoffman constant by selecting canonical restricted subsets $\mathcal{K}$. To do so, let $\mathcal{P}([d])$ denote the power set of all possible active sets, and let $\mathcal{P}_s([d]) := \{\mathcal{A} \subseteq [d] : |\mathcal{A}| \le s\}$ denote the collection of all $s$-sparse active sets. We now present four regimes of interest:
\begin{enumerate}
    \item \textbf{The global regime ($\mathcal{K} = \mathbb{R}^d$).} 
    The global Hoffman constant is the maximum metric distortion over all possible active subspaces:
    \begin{equation}
        H = H_{\mathbb{R}^d} = \max_{\mathcal{A} \in \mathcal{P}([d])} H_{\mathcal{S}_{\mathcal{A}}}.
    \end{equation}
    This matches the formulation given in~\cref{def:hoffman}. As previously discussed, this worst-case metric accounts for dense, ambient inequalities and scales poorly with the ambient dimension $d$.

    \item \textbf{The $s$-sparse regime ($\mathcal{K} = \bigcup_{\mathcal{A} \in \mathcal{P}_s([d])} \mathcal{S}_{\mathcal{A}}$).} 
    By restricting $\mathcal{K}$ to the union of all $s$-sparse subspaces, we obtain the $s$-sparse Hoffman constant, which evaluates worst-case local stability across all sufficiently sparse faces:
    \begin{equation}
        H_{\text{$s$-sparse}} = H_{\bigcup_{\mathcal{A} \in \mathcal{P}_s([d])}} = \max_{\mathcal{A} \in \mathcal{P}_s([d])} H_{\mathcal{S}_{\mathcal{A}}}.
    \end{equation}

    \item \textbf{The cone-restricted regime ($\mathcal{K} = \beta^* + \mathcal{C}_{\alpha}(\mathcal{A}^*)$).} 
    By setting $\mathcal{K}$ to the restricted cone, the resulting constant $H_{\beta^* + \mathcal{C}_{\alpha}(\mathcal{A}^*)}$ bounds the optimization geometry exactly where trajectories of algorithms like the ISTA are contained. This yields a highly path-aware metric that bypasses the need for bounds over all possible sparse sets, some of which may be ill-conditioned.

    \item \textbf{The local active-set regime ($\mathcal{K} = \mathcal{S}_{\hat{\mathcal{A}}}$).} 
    Once finite manifold identification occurs, the iterates are restricted to the optimal active subspace. The restricted constant $H_{\mathcal{S}_{\hat{\mathcal{A}}}}$ characterizes the geometry of this specific, fixed polyhedral face, ultimately governing the terminal linear convergence rate of the algorithm.
\end{enumerate}
While establishing that optimization trajectories visit these restricted sets is a promising first step, evaluating the restricted constants over these regimes remains a significant challenge. Existing literature treats these constants implicitly. We now turn to our primary contribution: deriving explicit, interpretable, computable bounds for these constants to bridge the statistical and geometric properties of sparse polyhedral regularizers.

\section{Explicit Hoffman Bounds for Structured Systems}
\label{sec:explicit-hc}

In this section, we develop explicit Hoffman bounds for a class of structured polyhedral systems whose constraints admit a natural separation into affine and polyhedral (e.g., sign and box) components. Our central objective is to transform the Hoffman constant from an abstract regularity parameter into a computable quantity whose dependence on the underlying system can be read directly from its structure.

For this section, the norm $\|\cdot\|_V$ on the domain is arbitrary. Each constituent constraint system is equipped with a residual norm $\|\cdot\|_{W_i}$. For a system obtained by vertically concatenating several constituent systems, we equip the product residual space with a block-compatible $\ell_p$ norm $\|\cdot\|_W$, meaning that
\begin{equation}
    \|u\|_W = \|(u_1, \dots, u_N)\|_W = \left(\sum_{i=1}^N \|u_i\|^p_{W_i} \right)^{1/p},
\end{equation}
with the usual maximum convention when $p = \infty$. We use $H_{\mathcal{K}} := H^{V, W}_{\mathcal{K}}$ for notational brevity when the norms are clear.

Throughout, we consider a structured nonempty polyhedral solution set $\Gamma \subseteq \mathbb{R}^d$ that is expressed as the intersection of two polyhedral sets:
\begin{equation}
    \label{eq:general-sol-of-polyhedral-regression}
    \Gamma := \Gamma_{\text{aff}} \cap \Gamma_{\text{con}}\ne \emptyset,
\end{equation}
where $\Gamma_{\text{aff}} \subseteq \mathbb{R}^d$ is an affine subspace determined by a linear system
\begin{equation}
    \Gamma_{\text{aff}} := \{x \in \mathbb{R}^d \mid M x = b\},\quad M\in \mathbb{R}^{n\times d},\quad b\in\mathbb{R}^{n},
\end{equation}
and $\Gamma_{\text{con}} \subseteq \mathbb{R}^d$ is a polyhedron encoding \textit{separable} coordinate-wise constraints. That is, it is the Cartesian product of independent one-dimensional intervals
\begin{equation}
    \Gamma_{\text{con}} = \Gamma_{\text{con}}^{(1)} \times \dots \times \Gamma_{\text{con}}^{(d)},
\end{equation}
where each $\Gamma_{\text{con}}^{(i)} \subseteq \mathbb{R}$ is a closed interval (possibly infinite), and thus a polyhedral set. This includes coordinate-wise sign and box constraints. The motivating example is the KKT system of the LASSO which will be the focus of \cref{sec:application-lasso}; however, our abstract analysis applies more broadly whenever such a decomposition is available.

Because both $\Gamma_{\text{aff}}$ and $\Gamma_{\text{con}}$ are polyhedral with a nonempty intersection it follows from the theory of error bounds for set intersections~\citep[Corollary 5.26, Corollary 3, and Proposition 2.2 respectively]{bauschke1996projection, bauschke1999convexintersection, lindstrom2022errorbounds} that for any arbitrary restricted subset $\mathcal{K} \subseteq \mathbb{R}^d$, the following bound holds for all $x \in \mathcal{K}$:
\begin{equation}
    \label{eq:restricted-bounded-regularity}
    \dist_V(x, \Gamma) \le \tau_{\mathcal{K}}(\Gamma_{\text{aff}}, \Gamma_{\text{con}}) \max(\dist_V(x, \Gamma_{\text{aff}}), \dist_V(x, \Gamma_{\text{con}}))
\end{equation}
where $\tau_{\mathcal{K}} \ge 1$ is the restricted \textit{linear regularity constant} for any finite collection of sets. This constant measures the geometric angle of intersection between sets over $\mathcal{K}$.~\footnote{For two linear subspaces in Euclidean space, linear regularity is closely related to their Friedrichs angle~\citep{deutsch1995angle, bauschke1993projectionconvergence}, with the precise constant depending on how the individual distances are aggregated. For the maximum-distance formulation used here, the two coordinate axes have global linear regularity constant $\tau=\sqrt{2}$: the distance from $(1,1)$ to their intersection is $\sqrt{2}$, whereas its distance to either axis is $1$. For two distinct lines, the constant grows without bound as their acute angle tends to zero. More generally, a finite collection of polyhedral sets with nonempty intersection admits a finite linear regularity constant~\citep{robinson1981polyhedral, zhou2017errorbounds}.}

Using this result, a similar bound can be proven for the Hoffman constants of intersected systems.
\begin{proposition}[Hoffman constant of two intersected systems]
    \label{prop:hoffman-intersection}
    Let $C_1, C_2$ be polyhedral sets such that $C := C_1 \cap C_2$ is nonempty. Equip the residual spaces of the two constituent algebraic systems with norms $\|\cdot\|_{W_1}$ and $\|\cdot\|_{W_2}$ respectively, and equip the residual space of the vertically concatenated algebraic system defining $C$ with an $\ell_p$ product norm $\|\cdot\|_W$. Then
    \begin{equation}
         H_{\mathcal K}^{V,W}(C) \le \tau_{\mathcal K}(C_1,C_2) \max\left[H_{\mathcal K}^{V,W_1}(C_1), H_{\mathcal K}^{V,W_2}(C_2) \right].
    \end{equation}
\end{proposition}
\begin{proof}
    By the definition of the individual restricted Hoffman constants, the distance from any point $x \in \mathcal{K}$ to each constituent set is bounded by its respective scalar residual:
    \begin{equation}
        \dist_V(x,C_1) \le H_{\mathcal K}^{V,W_1}(C_1) \Res_{W_1}(x,C_1) \quad \text{and} \quad \dist_V(x,C_2) \le H_{\mathcal K}^{V,W_2}(C_2) \Res_{W_2}(x,C_2).
    \end{equation}
    Thus, using the intersection bound as in \eqref{eq:restricted-bounded-regularity} and substituting the Hoffman constant upper bounds yields:
    \begin{align}
        \dist_V(x,C)
        &\le
        \tau_{\mathcal K}(C_1,C_2)
        \max\left[
            \dist_V(x,C_1),
            \dist_V(x,C_2)
        \right] \nonumber \\
        &\le
        \tau_{\mathcal K}(C_1,C_2)
        \max\left[
            H_{\mathcal K}^{V,W_1}(C_1)\Res_{W_1}(x,C_1),
            H_{\mathcal K}^{V,W_2}(C_2)\Res_{W_2}(x,C_2)
        \right] \nonumber \\
        &\le
        \tau_{\mathcal K}(C_1,C_2)
        \max\left[
            H_{\mathcal K}^{V,W_1}(C_1),
            H_{\mathcal K}^{V,W_2}(C_2)
        \right]
        \max\left[
            \Res_{W_1}(x,C_1),
            \Res_{W_2}(x,C_2)
        \right],
    \end{align}
    where the final line uses the property that $\max(A a, B b) \le \max(A, B) \max(a, b)$ for all nonnegative operands.

    Because the algebraic representation of $C$ is obtained by vertically concatenating the systems defining $C_1$ and $C_2$, its composite
    residual satisfies, for $1\le p<\infty$,
    \begin{equation}
        \Res_W(x,C)
        =
        \left(
            \Res_{W_1}(x,C_1)^p
            +
            \Res_{W_2}(x,C_2)^p
        \right)^{1/p}.
    \end{equation}
    Consequently,
    \begin{equation}
        \max\left[
            \Res_{W_1}(x,C_1),
            \Res_{W_2}(x,C_2)
        \right]
        \le
        \Res_W(x,C).
    \end{equation}
    When $p=\infty$, this relationship holds with equality. Therefore,
    \begin{equation}
        \dist_V(x,C)
        \le
        \tau_{\mathcal K}(C_1,C_2)
        \max\left[
            H_{\mathcal K}^{V,W_1}(C_1),
            H_{\mathcal K}^{V,W_2}(C_2)
        \right]
        \Res_W(x,C).
    \end{equation}
    Since this holds for every $x\in\mathcal K$, the definition of the
    restricted Hoffman constant yields the claimed bound.
\end{proof}
A few remarks should be made about this result. Firstly, the sets $C_1$ and $C_2$ are more general than the sets $\Gamma_{\text{aff}}$ and $\Gamma_{\text{con}}$: this result holds without the added structure of polyhedral sparse regression problems. Secondly, the result extends to intersections of finitely many systems equipped with an arbitrary product residual norm, as recorded below for completeness.

\begin{lemma}[Hoffman constant of finitely many intersected systems]
\label{lem:hoffman-constants-finitely-many}
    Let $C_1, \dots, C_N$ be a finite collection of polyhedral sets and suppose that $C := \bigcap_{i=1}^N C_i$ is nonempty. Let $\|\cdot\|_V$ be the norm on the ambient space, and let $\|\cdot\|_{W_i}$ be the residual norm associated with $C_i$. Suppose each constituent system has restricted Hoffman constant $H_{\mathcal K}^{V,W_i}(C_i)$.

    Equip the product residual space $W_1 \oplus \cdots \oplus W_N$ with a norm $\|\cdot\|_W$, and define
    \begin{equation}
    \kappa_W := \\sup_{\substack{(u_1,\dots,u_N)\neq 0\\u_i\in W_i}} \frac{\max_{1\leq i\leq N}\|u_i\|_{W_i}}{\|(u_1,\dots,u_N)\|_W}.
    \end{equation}
    Then the restricted Hoffman constant of the intersection $C$ is bounded by:
    \begin{equation}
        H^{V, W}_{\mathcal{K}}(C) \le \kappa_W  \tau_{\mathcal{K}}(C_1, \dots, C_N) \max_{1 \le i \le N} H^{V, W_i}_{\mathcal{K}}(C_i),
    \end{equation}
    where $\tau_{\mathcal{K}}(C_1, \dots, C_N)$ is the restricted linear regularity constant for the collection.
\end{lemma}
\begin{proof}
By the linear regularity condition for finite collections of polyhedral sets~\citep[Corollary 5.26]{bauschke1996projection}, for every $x\in\mathcal K$, restricted linear regularity gives
\begin{equation}
    \dist_V(x,C)
    \leq
    \tau_{\mathcal K}(C_1,\dots,C_N)
    \max_{1\leq i\leq N}\dist_V(x,C_i).
\end{equation}
Applying the individual restricted Hoffman bounds yields
\begin{align}
    \dist_V(x,C)
    &\leq
    \tau_{\mathcal K}(C_1,\dots,C_N)
    \max_{1\leq i\leq N}
    \left[
        H_{\mathcal K}^{V,W_i}(C_i)
        \Res_{W_i}(x,C_i)
    \right] \nonumber \\
    &\leq
    \tau_{\mathcal K}(C_1,\dots,C_N)
    \left(
        \max_{1\leq i\leq N}
        H_{\mathcal K}^{V,W_i}(C_i)
    \right)
    \max_{1\leq i\leq N}
    \Res_{W_i}(x,C_i).
\end{align}
By the definition of $\kappa_W$,
\begin{equation}
    \max_{1\leq i\leq N}\Res_{W_i}(x,C_i)
    \leq
    \kappa_W
    \Res_W(x,C),
\end{equation}
where the combined residual is formed by concatenating the constituent residuals. Consequently,
\begin{equation}
    \dist_V(x,C)
    \leq
    \kappa_W
    \tau_{\mathcal K}(C_1,\dots,C_N)
    \left(
        \max_{1\leq i\leq N}
        H_{\mathcal K}^{V,W_i}(C_i)
    \right)
    \Res_W(x,C),
\end{equation}
which proves the result. 

In particular, if the combined residual norm is the $\ell_p$ product norm
\begin{equation}
    \|(u_1,\dots,u_N)\|_W
    :=
    \left(
        \sum_{i=1}^N \|u_i\|_{W_i}^p
    \right)^{1/p},
    \qquad 1\leq p<\infty,
\end{equation}
or the usual maximum convention when $p = \infty$, then $\kappa_W=1$.
\end{proof}

Applying \cref{prop:hoffman-intersection} to the solution set \eqref{eq:general-sol-of-polyhedral-regression} of polyhedral regression isolates three geometric mechanisms dictating the stability of the system:
\begin{enumerate}
    \item[(i)] an \emph{affine component}, governed by the Hoffman constant $H_{\mathcal{K}}(\Gamma_{\text{aff}})$ of the linear system. Consequently, this depends on the conditioning (e.g., spectral properties) of $M$;
    \item[(ii)] a \emph{polyhedral component}, governed by the Hoffman constant $H_{\mathcal{K}}(\Gamma_{\text{con}})$ of the polyhedral constraint system. Consequently, this depends on the geometry of the corresponding polyhedron;
    \item[(iii)] an \emph{interaction component}, $\tau_{\mathcal{K}}(\Gamma_{\text{aff}}, \Gamma_{\text{con}})$, measuring how cleanly the affine subspace and the polyhedral constraints intersect.
\end{enumerate}
We now discuss methods for controlling each component individually.

\subsection{The Polyhedral Contribution}

The set $\Gamma_{\text{con}}$ encodes the permissible bounds for the primal and dual variables (i.e., subgradients) imposed by the regularization scheme. A common feature for many of these structured systems is that $\Gamma_{\text{con}}$ is \emph{separable}, taking the form of a Cartesian product of independent one-dimensional intervals: $\Gamma_{\text{con}} = \Gamma_{\text{con}}^{(1)} \times \dots \times \Gamma_{\text{con}}^{(d)}$, where each $\Gamma_{\text{con}}^{(i)} \subseteq \mathbb{R}$ is polyhedral.

Given that $\tau_{\mathcal{K}}(C_{1},C_{2})$ can be calculated easily when the sets $C_{1},C_{2}$ are orthogonal, it is possible to use the result in \cref{lem:hoffman-constants-finitely-many} to give an upper bound on $H_{\mathcal{K}}(\Gamma_{\text{con}})$ if bounds for each $H_{\mathcal{K}}(\Gamma_{\text{con}}^{(i)})$ can be obtained. However, it turns out that there is a more direct approach to find $H_{\mathcal{K}}(\Gamma_{\text{con}})$ exactly when a general absolute norm is used~\citep[Definition 5.4.18 and Theorem 5.4.19]{horn2012matrix}.

\begin{proposition}[Hoffman constant of separable polyhedral constraints]
    \label{prop:hoffman-separable-polyhedral}
    Let the norm $\|\cdot\|_V = \|\cdot\|_W$ be an absolute norm on $\mathbb{R}^d$ (i.e., $\|x\|_V = f(|x_1|, \dots, |x_d|)$ for some nondecreasing function $f$). Then, for any subset $\mathcal{K} \subseteq \mathbb{R}^d$ such that $\Gamma_{\text{con}} \subset \mathcal{K}$, the restricted Hoffman constant of $\Gamma_{\text{con}}$ is
    \begin{equation}
        H_{\mathcal{K}}(\Gamma_{\text{con}}) = 1.
    \end{equation}
\end{proposition}
\begin{proof}
    By the definition of the Cartesian product and the fact that $f$ is nondecreasing in every argument (monotonicity), the distance from any point $x$ to the set $\Gamma_{\text{con}}$ is achieved by projecting each coordinate independently. Thus, the distance evaluates to the norm of the vector of individual orthogonal distances:
    \begin{equation}
        \dist_V(x, \Gamma_{\text{con}}) = f\left(\dist_V(x_1, \Gamma_{\text{con}}^{(1)}), \dots, \dist_V(x_d, \Gamma_{\text{con}}^{(d)})\right).
    \end{equation}
    Similarly, for a system of independent one-dimensional constraints $x_i \in \Gamma_{\text{con}}^{(i)}$, the scalar residual for each coordinate is the distance to the interval: $\Res_W(x_i, \Gamma_{\text{con}}^{(i)}) = \dist_V(x_i, \Gamma_{\text{con}}^{(i)})$.

    The $\mathcal{K}$-restricted Hoffman inequality requires $\dist_V(x, \Gamma_{\text{con}}) \le H_{\mathcal{K}}(\Gamma_{\text{con}}) \Res_W(x, \Gamma_{\text{con}})$ for all $x \in \mathcal{K}$. Because the joint residual is evaluated under the same absolute norm $\|\cdot\|_W = \|\cdot\|_V$, we have
    \begin{equation}
        \Res_W(z, \Gamma_{\text{con}}) = f\left(\Res_W(z_1, \Gamma_{\text{con}}^{(1)}), \dots, \Res_W(z_m, \Gamma_{\text{con}}^{(d)})\right).
    \end{equation}
    Substituting the coordinate-wise equality of distances and residuals, we obtain
    \begin{equation}
        \label{eq:separable-dist-eq-con}
        \dist_V(z, \Gamma_{\text{con}}) = \Res_W(z, \Gamma_{\text{con}}) \quad \text{for all } x \in \mathcal{K}.
    \end{equation}
    Substituting this equality into the definition of the restricted Hoffman bound yields
    \begin{equation}
        \dist_V(x, \Gamma_{\text{con}}) \le H_{\mathcal{K}}(\Gamma_{\text{con}}) \dist_V(x, \Gamma_{\text{con}}) \quad \text{for all } x \in \mathcal{K}.
    \end{equation}
    For any point $x \in \mathcal{K}$ outside the feasible set where $\dist_V(x, \Gamma_{\text{con}}) > 0$, which exists as $\Gamma_{\text{con}} \subset \mathcal{K}$, dividing both sides by the distance provides the bound $1 \le H_{\mathcal{K}}(\Gamma_{\text{con}})$. Since the Hoffman inequality holds exactly for $H_{\mathcal{K}}(\Gamma_{\text{con}})\ = 1$, and the constant is defined as the smallest such value, it follows that $H_{\mathcal{K}}(\Gamma_{\text{con}}) = 1$.
\end{proof}

Notably, \cref{prop:hoffman-separable-polyhedral} applies to weighted $\ell_p$ norms. A consequence of this result is that the geometry of separable box and sign constraints never intrinsically worsens the conditioning of the optimization problem. Furthermore, the compatibility with weighted norms demonstrates that feature-scaling interventions, such as the adaptive LASSO~\citep{zou2006adaptivelasso}, do not alter the Hoffman constant of the polyhedral component, as the weights scale the domain distance and the residual identically.

\subsection{The Geometric Interaction Contribution}

Evaluating the bounded regularity constant $\tau_{\mathcal{K}}(\Gamma_{\text{aff}}, \Gamma_{\text{con}})$ for arbitrary polyhedral intersections is generally nontrivial. Geometrically, it captures the maximum sharpness or metric distortion created where the affine subspace slices through the boundary of the polyhedral constraints, which can be made arbitrarily pathological.

However, in the context of regularized regression, standard nondegeneracy conditions (such as signed general position~\citep{tibshirani2012lassouniqueness}) force the active linear subsystem to possess a unique solution that satisfies the constraints. For example, the LASSO solution is unique almost surely for a random design matrix $A$ from a continuous distribution~\citep{tibshirani2012lassouniqueness}. In these cases, the intersection of the restricted sets reduces to a unique point and the geometric interaction penalty vanishes.
\begin{proposition}[Linear regularity for linear systems with unique solutions]
    \label{prop:bounded-regularity}
    Suppose the linear system $Mx=b$ has a unique solution $\bar{x}$ and thus $\Gamma_{\text{aff}} = \{\bar{x}\}$. If $\bar{x} \in \Gamma_{\text{con}}$ such that the intersection is the nonempty singleton $\Gamma := \Gamma_{\text{aff}} \cap \Gamma_{\text{con}} = \{\bar{x}\}$, then for any norm $\|\cdot\|_V$ and any restricted domain $\mathcal{K}$ with $\mathcal K \backslash \{\bar{x}\} \neq \emptyset$, the linear regularity constant $\tau_{\mathcal{K}}(\Gamma_{\text{aff}}, \Gamma_{\text{con}}) = 1$.
\end{proposition}
\begin{proof}
    Because $\Gamma_{\text{aff}} = \{\bar{x}\}$ is a singleton and $\bar{x} \in \Gamma_{\text{con}}$ by assumption, it follows immediately that the intersection is $\Gamma = \{\bar{x}\}$. Therefore, for any arbitrary point $x \in \mathcal{K}$, the distance to the intersection is equal to the distance to the affine set:
    \begin{equation}
        \dist_V(x, \Gamma) = \|x - \bar{x}\|_V = \dist_V(x, \Gamma_{\text{aff}}).
    \end{equation}
    By the properties of the maximum operator, this equality implies:
    \begin{equation}
        \dist_V(x, \Gamma) \le \max \left( \dist_V(x, \Gamma_{\text{aff}}), \dist_V(x, \Gamma_{\text{con}}) \right).
    \end{equation}
    By the definition of the linear regularity condition \eqref{eq:restricted-bounded-regularity}, this establishes the upper bound $\tau_{\mathcal{K}} \le 1$. 
    
    Conversely, the reverse bound $\tau_{\mathcal{K}} \ge 1$ holds universally for \textit{any} intersection. Because $\Gamma \subseteq \Gamma_{\text{aff}}$ and $\Gamma \subseteq \Gamma_{\text{con}}$, the distance from any point $x$ to the intersection is naturally bounded below by the distance to either constituent set: $\dist_V(x, \Gamma_{\text{aff}}) \le \dist_V(x, \Gamma)$ and $\dist_V(x, \Gamma_{\text{con}}) \le \dist_V(x, \Gamma)$. Hence, we have
    \begin{equation}
        \max \left( \dist_V(x, \Gamma_{\text{aff}}), \dist_V(x, \Gamma_{\text{con}}) \right) \le \dist_V(x, \Gamma),
    \end{equation}
    and conclude $\tau_{\mathcal{K}}(\Gamma_{\text{aff}}, \Gamma_{\text{con}}) = 1$. If $\mathcal{K} \subseteq \{ \bar{x} \}$, then $\tau_{\mathcal K}(\Gamma_{\text{aff}}, \Gamma_{\text{con}})$ can be 0.
\end{proof}

Under these uniqueness conditions and an absolute norm shared over the domain and residual space, the decomposition in \cref{prop:hoffman-intersection} can be greatly simplified. Substituting $H_{\mathcal{K}}(\Gamma_{\text{con}}) = 1$ (\cref{prop:hoffman-separable-polyhedral}) and $\tau_{\mathcal{K}}(\Gamma_{\text{aff}}, \Gamma_{\text{con}}) = 1$ (\cref{prop:bounded-regularity}) into our intersected system bound (\cref{prop:hoffman-intersection}) reduces the problem to dependence on the Hoffman constant of the affine space
\begin{equation}
    \label{eq:hoffman-uniqueness-sep-bound}
    H_{\mathcal{K}}(\Gamma) \le \max \left(H_{\mathcal{K}}(\Gamma_{\text{aff}}), 1\right). 
\end{equation}
For practical purposes, the bound \eqref{eq:hoffman-uniqueness-sep-bound} applies to the analysis of sufficiently sparse active sets where the $\ell_p$ norm is used in both the domain and residual space (e.g., $H^{2, 2}, H^{1, 1}$, etc.).

\subsection{The Affine Contribution}

It is left to discuss the Hoffman constant of the affine component $\Gamma_{\text{aff}}$, and consequently the conditioning of the algebraic system $Mx = b$. Evaluating the restricted constant $H_{\mathcal{K}}(\Gamma_{\text{aff}})$ amounts to measuring the worst-case metric distortion of the active linear subsystem over the specified restricted domain $\mathcal{K}$.

For a nonempty affine system $\Gamma_{\mathrm{aff}} := \{x\in\mathbb{R}^d: Mx=b\} \neq \emptyset$, the exact global Hoffman constant is characterized by the quotient-space modulus
\begin{equation}
    \label{eq:affine-operator-norm}
    H_{\mathbb{R}^d}^{V,W}(\Gamma_{\mathrm{aff}}) =
    \sup_{\substack{r\in\range M\\r\neq 0}} \frac{\inf\{\|v\|_V:Mv=r\}}{\|r\|_W}.
\end{equation}
Indeed, if $r=Mx-b$, then
\begin{equation}
    \dist_V(x,\Gamma_{\mathrm{aff}}) = \inf\{\|v\|_V:Mv=r\}.
\end{equation}
This naturally upper-bounds the Hoffman constant over any restricted domain $\mathcal{K} \subseteq \mathbb{R}^d$. Therefore:
\begin{equation}
    H_{\mathcal K}^{V,W}(\Gamma_{\mathrm{aff}}) \le H_{\mathbb{R}^d}^{V,W}(\Gamma_{\mathrm{aff}}).
\end{equation}

For Euclidean domain and residual norms, the minimum-norm preimage is $M^\dagger r$, and hence
\begin{equation}
    H_{\mathbb{R}^d}^{2,2}(\Gamma_{\mathrm{aff}}) = \|M^\dagger\|_2.
\end{equation}
More generally, if $M$ is invertible, then for arbitrary norms
$\|\cdot\|_V$ and $\|\cdot\|_W$,
\begin{equation}
    H_{\mathbb{R}^d}^{V,W}(\Gamma_{\mathrm{aff}}) = \|M^{-1}\|_{W\to V} := \sup_{u\neq 0}
    \frac{\|M^{-1}u\|_V}{\|u\|_W}.
\end{equation}
Therefore, when $M$ is invertible and the conditions of
\eqref{eq:hoffman-uniqueness-sep-bound} hold, the composite system satisfies
\begin{equation}
    H_{\mathcal K}^{V,W}(\Gamma) \le \max\left( \|M^{-1}\|_{W\to V}, 1 \right).
\end{equation}

This variational characterization allows the geometric framework to accommodate different statistical and algorithmic metrics. In
particular, when $M$ is invertible and the domain and residual spaces are equipped with the same standard $\ell_p$ norm, the affine Hoffman constant reduces to the corresponding induced matrix norm $\|M^{-1}\|_p$~\citep{horn2012matrix}:
\begin{itemize}
    \item \textbf{Euclidean/spectral geometry ($p=2$):}
    \begin{equation}
        \|M^{-1}\|_2 = \frac{1}{\sigma_{\min}(M)},
    \end{equation}
    noting that we can replace $\sigma^+_{\min}(M)$ with $\sigma_{\min}(M)$ as $M$ is invertible.

    \item \textbf{Manhattan geometry ($p=1$):}
    \begin{equation}
        \|M^{-1}\|_1
        =
        \max_j\sum_i |(M^{-1})_{ij}|,
    \end{equation}
    which is the maximum absolute column sum of $M^{-1}$.

    \item \textbf{Chebyshev geometry ($p=\infty$):}
    \begin{equation}
        \|M^{-1}\|_\infty
        =
        \max_i\sum_j |(M^{-1})_{ij}|,
    \end{equation}
    which is the maximum absolute row sum of $M^{-1}$.
\end{itemize}

In \cref{sec:application-lasso}, we explicitly construct the active KKT matrix $M_{\mathcal A}$ for the LASSO. When this is invertible, the preceding induced-norm characterization applies directly. By exploiting its block structure, we show how the affine KKT Hoffman constant captures both active-set curvature and active-inactive feature interaction. We then derive explicit spectral bounds for $\|M_{\mathcal A}^{-1}\|_2$ whose dependence on the sparsity, ambient dimension, and structural properties of the design is transparent. In \cref{sec:hoffman-other-norms}, we extend this analysis to the induced $\ell_1$ and $\ell_\infty$ norms.

\subsection{Monotonicity Properties for Hoffman Bounds}

We conclude this section by stating some monotonicity properties of Hoffman constants, which allow certain Hoffman-type bounds to be transferred across different norms. To the best of our knowledge, these results are not stated explicitly in the variational analysis literature, which is primarily concerned with the existence of Hoffman constants rather than their specific magnitudes.

To emphasize that these properties are fundamental to the variational analysis of polyhedral systems and do not rely on the specific decomposition assumed in \eqref{eq:general-sol-of-polyhedral-regression}, we state them here for an arbitrary nonempty polyhedral set $\mathcal{X}$ and restricted domain $\mathcal{K}$.

\begin{proposition}[Monotonicity of the Hoffman constant in the domain norm]
    \label{prop:hoffman-monotonicity-domain}
    Let $\mathcal{X} \subseteq \mathbb{R}^d$ be a nonempty polyhedral solution set and $\mathcal{K} \subseteq \mathbb{R}^d$ be a restricted domain. Suppose $\|\cdot\|_V$ and $\|\cdot\|_{V'}$ are norms on the domain such that for all $x \in \mathcal{K}$ and $x' \in \mathcal{X}$, $\|x-x'\|_{V'} \le \|x-x'\|_V$. Then for any residual space norm $\|\cdot\|_W$:
    \begin{equation}
        H_{\mathcal{K}}^{V', W}(\mathcal{X}) \le H_{\mathcal{K}}^{V, W}(\mathcal{X}).
    \end{equation}
\end{proposition}
\begin{proof}
    By definition of the restricted Hoffman constant $H_{\mathcal{K}}^{V, W}(\mathcal{X})$, we have for all $x \in \mathcal{K}$ that
    \begin{equation}
        \dist_V(x, \mathcal{X}) = \min_{x' \in \mathcal{X}} \|x - x'\|_V \le H_{\mathcal{K}}^{V, W}(\mathcal{X}) \Res_W(x, \mathcal{X}).
    \end{equation}
    By assumption on the domain norms, the distance in $V'$ is bounded by the distance in $V$:
    \begin{equation}
        \dist_{V'}(x, \mathcal{X}) = \min_{x' \in \mathcal{X}} \|x - x'\|_{V'} \le \min_{x' \in \mathcal{X}} \|x - x'\|_V = \dist_V(x, \mathcal{X}).
    \end{equation}
    Combining these inequalities yields
    \begin{equation}
        \dist_{V'}(x, \mathcal{X}) \le H_{\mathcal{K}}^{V, W}(\mathcal{X}) \Res_W(x, \mathcal{X}).
    \end{equation}
    Since $H_{\mathcal{K}}^{V', W}(\mathcal{X})$ is the smallest constant satisfying such an inequality, it follows that $H_{\mathcal{K}}^{V', W}(\mathcal{X}) \le H_{\mathcal{K}}^{V, W}(\mathcal{X})$.
\end{proof}

This result is particularly relevant for standard $\ell_p$ norms, where the inequality $\|\cdot\|_q \leq \|\cdot\|_p$ holds uniformly for $1 \leq p \leq q \leq \infty$, implying that moving to a higher-order norm on the domain (e.g., from $\ell_2$ to $\ell_\infty$) naturally decreases the Hoffman constant. We now prove the analogous result for the residual space.

\begin{proposition}[Monotonicity of the Hoffman constant in the residual norm]
    \label{prop:hoffman-monotonicity-residual}
    Let $\mathcal{X} \subseteq \mathbb{R}^d$ be a nonempty polyhedral solution set and $\mathcal{K} \subseteq \mathbb{R}^d$ be a restricted domain. Suppose $\|\cdot\|_W$ and $\|\cdot\|_{W'}$ are norms on the residual space such that for all $x \in \mathcal{K}$, the residual satisfies $\Res_W(x, \mathcal{X}) \le \Res_{W'}(x, \mathcal{X})$. Then for any domain norm $\|\cdot\|_V$:
    \begin{equation}
        H_{\mathcal{K}}^{V, W'}(\mathcal{X}) \le H_{\mathcal{K}}^{V, W}(\mathcal{X}).
    \end{equation}
\end{proposition}
\begin{proof}
    By definition of the restricted Hoffman constant $H_{\mathcal{K}}^{V, W}(\mathcal{X})$, we have for all $x \in \mathcal{K}$ that
    \begin{equation}
        \dist_V(x, \mathcal{X}) \le H_{\mathcal{K}}^{V, W}(\mathcal{X}) \Res_W(x, \mathcal{X}).
    \end{equation}
    Substituting the assumption $Res_W(x, \mathcal{X}) \le \Res_{W'}(x, \mathcal{X})$ into the inequality:
    \begin{equation}
        \dist_V(x, \mathcal{X}) \le H_{\mathcal{K}}^{V, W}(\mathcal{X}) \Res_{W'}(x, \mathcal{X}).
    \end{equation}
    Since $H_{\mathcal{K}}^{V, W'}(\mathcal{X})$ is the smallest constant satisfying such an inequality, it follows that $H_{\mathcal{K}}^{V, W'}(\mathcal{X}) \le H_{\mathcal{K}}^{V, W}(\mathcal{X})$.
\end{proof}

A consequence of combining \cref{prop:hoffman-monotonicity-domain} and \cref{prop:hoffman-monotonicity-residual} is that there is no universal monotonicity property in the case when the same norm is used for both the domain and residual space (e.g., the $H_{\mathcal{K}}^{p,p}$ case). 

\section{The Geometry of Sparse Regression}
\label{sec:application-lasso}

Having established a general decomposition for structured polyhedral regression systems, we now apply this framework to the LASSO (\cref{problem:lasso}). In particular, we show how the Hoffman constant can further be decomposed into active-set curvature and active-inactive feature interaction effect. An extended derivation is provided in Appendix~\ref{app:lasso-polyhedral}.

Recall the LASSO objective function:
\begin{equation}
    F(\beta) := \frac{1}{2n} \|A\beta - y\|_2^2 + \lambda_n \|\beta\|_1.
\end{equation}
Because $F(\beta)$ is a proper, closed, convex function, its global minimizers are characterized by the first-order KKT conditions~\citep{boyd2004convex}. A vector $\beta \in \mathbb{R}^d$ is optimal if and only if $0 \in \partial F(\beta)$, which requires the existence of a subgradient vector $g \in \partial \|\beta\|_1$ such that:
\begin{equation}
    \label{eq:lasso-kkt-global}
    \frac{1}{n} A^\top (A\beta - y) + \lambda_n g = 0,
\end{equation}
where the subdifferential of the $\ell_1$ norm acts coordinate-wise: $g_j = \sign(\beta_j)$ if $\beta_j \neq 0$, and $g_j \in [-1, 1]$ if $\beta_j = 0$; these are the sign and box constraints respectively. 

\subsection{The Local KKT System for a Fixed Active Set}

To apply our geometric decomposition, we reformulate the global KKT conditions \eqref{eq:lasso-kkt-global} as an intersected polyhedral system $\Gamma_{\mathcal{A}} = \Gamma_{\text{aff}, \mathcal{A}} \cap \Gamma_{\text{con}, \mathcal{A}}$ over the active variables. Let $G = \frac{1}{n} A^\top A$ be the normalized Gram matrix. By partitioning the system into active ($\mathcal{A}$) and inactive ($\mathcal{A}^c$) blocks, we obtain a square linear system over the free variables $z := (\beta_{\mathcal{A}}^\top, g_{\mathcal{A}^c}^\top)^\top \in \mathbb{R}^d$:
\begin{equation}
    \label{eq:lasso-affine-system}
    \underbrace{\begin{pmatrix}
    G_{\mathcal{A}\mathcal{A}} & 0 \\
    G_{\mathcal{A}^c\mathcal{A}} & \lambda_n I
    \end{pmatrix}}_{:= M_{\mathcal{A}}}
    \begin{pmatrix}
    \beta_{\mathcal{A}} \\
    g_{\mathcal{A}^c}
    \end{pmatrix}
    =
    \underbrace{\begin{pmatrix}
    \frac{1}{n} A_{\mathcal{A}}^\top y - \lambda_n \sign(\beta_{\mathcal{A}}) \\
    \frac{1}{n} A_{\mathcal{A}^c}^\top y
    \end{pmatrix}}_{:= b_{\mathcal{A}}}.
\end{equation}
This block lower-triangular matrix $M_{\mathcal{A}} \in \mathbb{R}^{d \times d}$ defines the affine subspace with solution set $\Gamma_{\text{aff}, \mathcal{A}}$.

Simultaneously, the definition of the $\ell_1$ subdifferential imposes strict, separable coordinate-wise constraints on the vector $z$. The active parameters must maintain their identified signs (e.g., $z_j \ge 0$ if $\sign(\beta_j) = 1$), and the inactive subgradients must not violate the dual feasibility bounds ($z_j \in [-1, 1]$ for $j \in \mathcal{A}^c$). These independent, one-dimensional intervals form exactly the separable bounding polyhedron with solution set  $\Gamma_{\text{con}}$ defined in \cref{prop:hoffman-separable-polyhedral}.

\subsection{Assumptions: General Position and Sparsity}
\label{subsec:lasso-assumptions}

To extract explicit bounds from the decomposed system, we must evaluate the geometric stability of the affine linear system $M_{\mathcal{A}} z = b_{\mathcal{A}}$. By the definition of the Hoffman constant, this is governed by the operator norm of the matrix inverse, $\|M_{\mathcal{A}}^{-1}\|_2$. Because $\det(M_{\mathcal{A}}) = \det(G_{\mathcal{A}\mathcal{A}}) \cdot \det(\lambda_n I)$, the KKT block matrix $M_{\mathcal{A}}$ is invertible if and only if $\lambda_n > 0$ and the active Gram matrix $G_{\mathcal{A}\mathcal{A}}$ is positive definite. The latter is true if and only if $\mynull A_{\mathcal{A}} = \{0\}$ for the active set under consideration.

For convenience, we will assume a stronger condition which will make this true for all candidate active sets. To guarantee the restricted linear subsystem is tractable, we rely on a version of the statistical uniqueness argument of the LASSO estimator~\citep{tibshirani2012lassouniqueness}.

\begin{assumption}[General position]
\label{assum:general-position}
    Every collection of at most $n$ columns of $A\in\mathbb R^{n\times d}$ is linearly independent.
\end{assumption}

\begin{assumption}[Sparse active set]
\label{assum:sparse-active-set}
    The active set $\mathcal{A}$ under consideration is $n$-sparse. That is, $|\mathcal{A}| \le n$.
\end{assumption}

If the columns of $A$ are drawn jointly from a continuous probability distribution, they satisfy~\cref{assum:general-position} almost surely. Under \cref{assum:general-position,assum:sparse-active-set}, every active submatrix $A_{\mathcal A}$ considered has full column rank. Consequently, $G_{\mathcal A\mathcal A}$ is
positive definite and $M_{\mathcal A}$ is invertible.

Under the additional condition that the signed columns
$\{\pm A_1,\dots,\pm A_d\}$ are in general position in the sense of~\citet{tibshirani2012lassouniqueness}, the LASSO problem admits a unique, sparse optimal solution for every $\lambda_n > 0$. This means that the optimal active set is sufficiently sparse ($|\hat{\mathcal{A}}| \le n$) and has a submatrix $A_{\hat{\mathcal{A}}}$ with full column rank. Thus, once manifold identification has occurred, the resulting terminal active-set regime is covered by our analysis.

To reiterate the exact conditions of our geometric decomposition: the bound on $H(\Gamma_{\text{con}})$ in \cref{prop:hoffman-separable-polyhedral} applies for \textit{any} active set, given an absolute norm. Furthermore, the variational formulation $H(\Gamma_{\text{aff}})$ is universally applicable using the quotient-space characterization in \eqref{eq:affine-operator-norm}. However, the collapse of the intersection penalty (namely $\tau(\Gamma_{\text{aff}}, \Gamma_{\text{con}}) = 1$) requires the solution of the affine system to be unique. Under \cref{assum:general-position,assum:sparse-active-set}, we have access to the exact inverse of $M_{\mathcal{A}}$, allowing us to abandon the pseudoinverse, which is analytically cumbersome and obscures the explicit geometric interactions. 

In the high-dimensional regime ($d \gg n$), the Gram matrix over any nonsparse set is inherently rank-deficient. We further discuss ideas for bounds on nonsparse active sets in \cref{sec:discussion}.

\subsection{Explicit Spectral Bounds}
\label{subsec:spectral-bounds}

For the remainder of this analysis, we operate under \cref{assum:general-position,assum:sparse-active-set} unless otherwise specified.

To evaluate the Hoffman constant of this affine system, we apply standard block matrix inversion to $M_{\mathcal{A}}$. For any positive regularization penalty $\lambda_n > 0$, the exact inverse is given by~\citep{lu2002matrixinverse}:
\begin{equation}
    \label{eq:inverse-formulation}
    M_{\mathcal A}^{-1}=
    \begin{pmatrix}
    G_{\mathcal A\mathcal A}^{-1} & 0\\[3pt]
    -\frac{1}{\lambda_n} G_{\mathcal A^c\mathcal A} G_{\mathcal A\mathcal A}^{-1} & \frac{1}{\lambda_n} I
    \end{pmatrix}.
\end{equation}
The magnitude of the restricted Hoffman constant is therefore entirely dictated by the induced norm of this specific block matrix. It will be a recurring theme that controlling $G_{\mathcal A\mathcal A}^{-1}$ requires bounding the active set curvature, and controlling $G_{\mathcal A^c\mathcal A} G_{\mathcal A\mathcal A}^{-1}$ requires bounding the active-inactive feature interaction.

The Euclidean Hoffman constant $H^{2,2}$ of the active set is governed by $\|M_{\mathcal{A}}^{-1}\|_2$, for which we now provide a range of deterministic upper bounds. These bounds rely on techniques from matrix perturbation theory~\citep{stewart1990matrixperturbation} and exhibit varying degrees of tightness depending on the geometric regime (e.g., high versus low active-inactive feature interaction). Because no single bound is optimal across all statistical designs, the practical geometric penalty can be minimized by taking the tightest available upper bound: $\|M_{\mathcal{A}}^{-1}\|_2 \le \min_i U_i(\mathcal{A})$ for some collection $\{U_i\}$ of upper bounds.

\subsubsection{Frobenius-Type Bound}

The most immediate bound relies on a Frobenius-type argument. Because the spectral norm is bounded by the Frobenius norm ($\|\cdot\|_2 \le \|\cdot\|_F$ uniformly), and the squared Frobenius norm of a block matrix is the sum of the squared Frobenius norms of its constituent blocks, we can bound the inverse matrix \eqref{eq:inverse-formulation} directly:
\begin{equation}
    \|M_{\mathcal A}^{-1}\|_2 \le \|M_{\mathcal A}^{-1}\|_F = \sqrt{ \|G_{\mathcal A\mathcal A}^{-1}\|_F^2 + \frac{1}{\lambda_n^2} \| G_{\mathcal A^c\mathcal A} G_{\mathcal A\mathcal A}^{-1}\|_F^2 + \frac{1}{\lambda_n^2} \|I\|_F^2 }.
\end{equation}
Since the identity matrix $I$ in the lower-right block corresponds to the inactive constraints, its dimension is $(d - s) \times (d - s)$, where $s = |\mathcal{A}|$ is the active sparsity. Its squared Frobenius norm is exactly $d - s$.
\begin{bound}[Frobenius-type spectral bound]
    The restricted $\ell_2$ Hoffman constant is bounded by:
    \begin{equation}
    H^{2, 2}_{\mathcal{S}_{\mathcal{A}}}(\Gamma_{\text{aff}, \mathcal{A}}) \le \sqrt{ \|G_{\mathcal A\mathcal A}^{-1}\|_F^2 + \frac{1}{\lambda_n^2} \|G_{\mathcal A^c\mathcal A} G_{\mathcal A\mathcal A}^{-1}\|_F^2 + \frac{d - s}{\lambda_n^2} }.
    \end{equation}
\end{bound}
While algebraically simple, this bound scales explicitly and implicitly with the ambient dimension $d$, which fails to capture the dimension-independent convergence rates often observed in sparse, high-dimensional practice.

\subsubsection{Dimension-Free Spectral Bound}

To obtain a bound that does not explicitly scale with the ambient dimension, we use the variational characterization of the spectral norm and partition the domain. This provides our main ``universal'' deterministic bound.

\begin{bound}[Dimension-free spectral upper bound]
    \label{bound:dimension-free-bound}
    The restricted $\ell_2$ Hoffman constant is bounded above by:
    \begin{equation}
    H^{2, 2}_{\mathcal{S}_{\mathcal{A}}}(\Gamma_{\text{aff}, \mathcal{A}}) \le \sqrt{ \|G_{\mathcal A\mathcal A}^{-1}\|_2^2 + \frac{1}{\lambda_n^2} \left( 1 + \|G_{\mathcal A^c\mathcal A} G_{\mathcal A\mathcal A}^{-1}\|_2^2 \right) }.
    \end{equation}
\end{bound}
\begin{proof}
    By definition, the spectral norm is given by the supremum over the unit ball
    \begin{equation}
        \|M_{\mathcal A}^{-1}\|_2 = \sup_{\|z\|_2=1} \|M_{\mathcal A}^{-1} z\|_2.
    \end{equation} 
    Let $z = (x^\top, y^\top)^\top$ be a partitioned unit vector such that $\|x\|_2^2 + \|y\|_2^2 = 1$. Multiplying by the explicit block inverse \eqref{eq:inverse-formulation}, we obtain:
    \begin{equation}
    M_{\mathcal{A}}^{-1} z = \begin{pmatrix} G_{\mathcal{A}\mathcal{A}}^{-1} x \\ -\frac{1}{\lambda_n} G_{\mathcal{A}^c\mathcal{A}} G_{\mathcal{A}\mathcal{A}}^{-1} x + \frac{1}{\lambda_n} y \end{pmatrix}.
    \end{equation}
    The squared Euclidean norm of this transformed vector is:
    \begin{equation}
    \|M_{\mathcal{A}}^{-1} z\|_2^2 = \|G_{\mathcal{A}\mathcal{A}}^{-1} x\|_2^2 + \left\| -\frac{1}{\lambda_n} G_{\mathcal{A}^c\mathcal{A}} G_{\mathcal{A}\mathcal{A}}^{-1} x + \frac{1}{\lambda_n} y \right\|_2^2.
    \end{equation}
    Applying the submultiplicativity of the operator norm to the first term, and the triangle inequality followed by the Cauchy--Schwarz inequality to the second term, we bound the expression as:
    \begin{equation}
    \|M_{\mathcal{A}}^{-1} z\|_2^2 \le \|G_{\mathcal{A}\mathcal{A}}^{-1}\|_2^2 \|x\|_2^2 + \frac{1}{\lambda_n^2} \left( \|G_{\mathcal{A}^c\mathcal{A}} G_{\mathcal{A}\mathcal{A}}^{-1}\|_2^2 + 1 \right) \left( \|x\|_2^2 + \|y\|_2^2 \right).
    \end{equation}
    Since our initial vector was chosen such that $\|x\|_2^2 + \|y\|_2^2 = 1$ (and consequently $\|x\|_2^2 \le 1$), this simplifies to:
    \begin{equation}
    \|M_{\mathcal{A}}^{-1} z\|_2^2 \le \|G_{\mathcal{A}\mathcal{A}}^{-1}\|_2^2 + \frac{1}{\lambda_n^2} \left( \|G_{\mathcal{A}^c\mathcal{A}} G_{\mathcal{A}\mathcal{A}}^{-1}\|_2^2 + 1 \right).
    \end{equation}
    Taking the square root of both sides yields the stated bound.
\end{proof}
This formulation further isolates the two geometric mechanisms that dictate the conditioning of the active system: the active feature curvature $\|G_{\mathcal{A} \mathcal{A}}^{-1}\|_2$, and the interaction between the active and inactive features $\|G_{\mathcal{A}^c \mathcal{A}} G_{\mathcal{A} \mathcal{A}}^{-1}\|_2$.

An analogous method can be used to derive a lower bound:
\begin{bound}[Dimension-free spectral lower bound]
    \label{bound:dimension-free-lower-bound}
    The restricted $\ell_2$ Hoffman constant is bounded below by:
    \begin{equation}
        H^{2, 2}_{\mathcal{S}_{\mathcal{A}}}(\Gamma_{\text{aff}, \mathcal{A}}) \ge \max \left( \|G_{\mathcal{A}\mathcal{A}}^{-1}\|_2, \; \frac{1}{\lambda_n}, \; \frac{1}{\lambda_n} \|G_{\mathcal{A}^c\mathcal{A}} G_{\mathcal{A}\mathcal{A}}^{-1}\|_2 \right).
    \end{equation}
\end{bound}
\begin{proof}
    Restricting the supremum in the spectral norm definition $\|M_{\mathcal{A}}^{-1}\|_2 = \sup_{\|z\|_2=1} \|M_{\mathcal{A}}^{-1} z\|_2$ to unit vectors of the form $z = (x^\top, 0^\top)^\top$ yields:
    \begin{equation}
        \|M_{\mathcal{A}}^{-1}\|_2^2 \ge \|G_{\mathcal{A}\mathcal{A}}^{-1} x\|_2^2 + \frac{1}{\lambda_n^2} \|G_{\mathcal{A}^c\mathcal{A}} G_{\mathcal{A}\mathcal{A}}^{-1} x\|_2^2.
    \end{equation}
    Taking the supremum over $x$ independently for each nonnegative term yields lower bounds of $\|G_{\mathcal{A}\mathcal{A}}^{-1}\|_2$ and $\frac{1}{\lambda_n} \|G_{\mathcal{A}^c\mathcal{A}} G_{\mathcal{A}\mathcal{A}}^{-1}\|_2$. Alternatively, restricting to $z = (0^\top, y^\top)^\top$ yields the remaining bound $\frac{1}{\lambda_n}$. 
\end{proof}

In order to see the tightness of these bounds, observe that $a^2 + b^2 + c^2 \le 3 \max(a, b, c)^2$ for any positive scalars, and thus the upper and lower bounds sandwich the exact restricted $\ell_2$ Hoffman constant within a factor of $\sqrt{3}$:
\begin{equation}
     H^{2, 2}_{\mathcal{S}_{\mathcal{A}}}(\Gamma_{\text{aff}, \mathcal{A}}) \le \sqrt{\|G_{\mathcal{A}\mathcal{A}}^{-1}\|_2^2 + \frac{1}{\lambda_n^2} + \frac{1}{\lambda_n^2} \|G_{\mathcal{A}^c\mathcal{A}} G_{\mathcal{A}\mathcal{A}}^{-1}\|_2^2} \le \sqrt{3} \max \left( \|G_{\mathcal{A}\mathcal{A}}^{-1}\|_2, \frac{1}{\lambda_n}, \frac{1}{\lambda_n} \|G_{\mathcal{A}^c\mathcal{A}} G_{\mathcal{A}\mathcal{A}}^{-1}\|_2 \right).
\end{equation}
Consequently, this technique is tight in the sense that the constant only grows large if the active curvature $\|G_{\mathcal{A}\mathcal{A}}^{-1}\|_2$ explodes, or if the interaction leakage $\|G_{\mathcal{A}^c\mathcal{A}} G_{\mathcal{A}\mathcal{A}}^{-1}\|_2$ overwhelms the regularization $\lambda_n$:
\begin{equation}
    \label{eq:l2-hoffman-tight-lower-upper-bound}
     \max \left( \|G_{\mathcal{A}\mathcal{A}}^{-1}\|_2, \frac{1}{\lambda_n}, \frac{1}{\lambda_n} \|G_{\mathcal{A}^c\mathcal{A}} G_{\mathcal{A}\mathcal{A}}^{-1}\|_2 \right) \le H^{2, 2}_{\mathcal{S}_{\mathcal{A}}}(\Gamma_{\text{aff}, \mathcal{A}}) \le \sqrt{3} \max \left( \|G_{\mathcal{A}\mathcal{A}}^{-1}\|_2, \frac{1}{\lambda_n}, \frac{1}{\lambda_n} \|G_{\mathcal{A}^c\mathcal{A}} G_{\mathcal{A}\mathcal{A}}^{-1}\|_2 \right).
\end{equation}
This tightness is demonstrated numerically for a range of designs and high-dimensional regimes in Appendix~\ref{app:empirical-tightness-scaling}.

\subsubsection{Weyl's Perturbation Inequality}

Finally, if the interaction term is sufficiently small, we can treat the active linear system as a perturbation to a well-conditioned diagonal matrix. We decompose $M_{\mathcal{A}}$ into a symmetric block-diagonal matrix $D$ (due to the symmetry of Gram matrices) and an off-diagonal perturbation $E$:
\begin{equation}
    M_{\mathcal{A}}=\underbrace{\begin{pmatrix}G_{\mathcal{A}\mathcal{A}}&0\\0&\lambda_n I\end{pmatrix}}_{D}+\underbrace{\begin{pmatrix}0&0\\G_{\mathcal{A}^c\mathcal{A}}&0\end{pmatrix}}_{E}.
\end{equation}
By Weyl's inequality for singular values, $\sigma_{\min}(M_{\mathcal{A}}) \ge \sigma_{\min}(D) - \|E\|_2$. Since $\sigma_{\min}(D) = \min(\sigma_{\min}(G_{\mathcal{A}\mathcal{A}}), \lambda_n)$ and $\|E\|_2 = \|G_{\mathcal{A}^c\mathcal{A}}\|_2$, we obtain the following bound.

\begin{bound}[Weyl's perturbation inequality]
    \label{bound:weyl-perturbation}
    Assume the statistical interaction is bounded by the active curvature and regularization, $\|G_{\mathcal{A}^c\mathcal{A}}\|_2 < \min(\sigma_{\min}(G_{\mathcal{A}\mathcal{A}}), \lambda_n)$. Then, the restricted $\ell_2$ Hoffman constant is bounded by:
    \begin{equation}
    \label{eq:weyl-hoffman-bound}
    H^{2, 2}_{\mathcal{S}_{\mathcal{A}}}(\Gamma_{\text{aff}, \mathcal{A}}) \le \frac{1}{ \min(\sigma_{\min}(G_{\mathcal{A}\mathcal{A}}), \lambda_n) - \|G_{\mathcal{A}^c\mathcal{A}}\|_2}.
    \end{equation}
    Furthermore, this bound trivially holds with equality when $G_{\mathcal{A}^c\mathcal{A}} = 0$.
\end{bound}
While this bound is invalid if the interaction perturbation is large, it becomes tight in regimes where statistical conditions (e.g., incoherence conditions) naturally suppress $\|G_{\mathcal{A}^c\mathcal{A}}\|_2$.

\subsection{Bounding the Active-Inactive Interaction}

Across the spectral bounds, a primary factor of geometric instability is the interaction block $\|G_{\mathcal{A}^c\mathcal{A}}\|_2$. If the features were known a priori, there would be no inactive set to consider, and the optimization geometry would depend entirely on the active curvature. However, since the true support is unknown and must be discovered, the optimization landscape is complicated by this extra interaction term.

We can construct a post-hoc diagnostic to determine which variables caused poor conditioning of the spectral norm. The following deterministic bound is constructed to play a similar function to the leverage score in the existing regression literature~\citep{hoaglin1978leverage}; we call this the active-inactive cross-leverage.

\begin{lemma}[Active-inactive cross-leverage bound]
    \label{lem:interaction-bound}
    Let $\Pi_{\mathcal{A}}=A_{\mathcal{A}}(A_{\mathcal{A}}^\top A_{\mathcal{A}})^{-1}A_{\mathcal{A}}^\top$ be the orthogonal projection matrix onto the subspace spanned by the active features, and let $a_j \in \mathbb{R}^n$ denote the $j$-th column of the design matrix $A$. For each inactive feature $j\in\mathcal{A}^c$, we define the active-inactive cross-leverage score as $h_j(\mathcal{A})=\frac{1}{n}\|\Pi_{\mathcal{A}}a_j\|_2^2$. The spectral norm of the interaction block satisfies:
    \begin{equation}
    \|G_{\mathcal{A}^c\mathcal{A}}\|_2 \le \frac{1}{\sqrt{n}} \sigma_{\max}(A_\mathcal{A}) \sqrt{\sum_{j\in\mathcal{A}^c}h_j(\mathcal{A})}. 
    \end{equation}
\end{lemma}
\begin{proof}
    The spectral norm is bounded by the Frobenius norm:
    \begin{equation}
        \|G_{\mathcal{A}^c\mathcal{A}}\|_2^2 \le \|G_{\mathcal{A}^c\mathcal{A}}\|_F^2 = \sum_{j \in \mathcal{A}^c} \|(G_{\mathcal{A}^c\mathcal{A}})_{j,:}\|_2^2.
    \end{equation}
    The $j$-th row of the interaction block is given by $(G_{\mathcal{A}^c\mathcal{A}})_{j,:} = \frac{1}{n} a_j^\top A_{\mathcal{A}}$. Its squared Euclidean norm expands to:
    \begin{equation}
        \|(G_{\mathcal{A}^c\mathcal{A}})_{j,:}\|_2^2 = \frac{1}{n^2} a_j^\top A_{\mathcal{A}} A_{\mathcal{A}}^\top a_j.
    \end{equation}
    Using the singular value decomposition $A_{\mathcal{A}} = U \Sigma V^\top$, we can express the matrix product as $A_{\mathcal{A}} A_{\mathcal{A}}^\top = U \Sigma^2 U^\top$. Since the squared singular values are bounded by the maximal singular value, $\Sigma^2 \preceq \sigma_{\max}^2(A_{\mathcal{A}}) I$, we can bound the matrix product geometrically:
    \begin{equation}
        A_{\mathcal{A}} A_{\mathcal{A}}^\top = U \Sigma^2 U^\top \preceq \sigma_{\max}^2(A_{\mathcal{A}}) U U^\top = \sigma_{\max}^2(A_{\mathcal{A}}) \Pi_{\mathcal{A}},
    \end{equation}
    where $\Pi_{\mathcal{A}} = U U^\top$ is the orthogonal projection matrix onto the column space of $A_{\mathcal{A}}$. Substituting this upper bound back into the quadratic form yields:
    \begin{equation}
        \|(G_{\mathcal{A}^c\mathcal{A}})_{j,:}\|_2^2 \le \frac{1}{n^2} \sigma_{\max}^2(A_{\mathcal{A}}) a_j^\top \Pi_{\mathcal{A}} a_j = \frac{1}{n} \sigma_{\max}^2(A_{\mathcal{A}}) \left( \frac{1}{n} a_j^\top \Pi_{\mathcal{A}} a_j \right) = \frac{1}{n} \sigma_{\max}^2(A_{\mathcal{A}}) h_j(\mathcal{A}).
    \end{equation}
    Summing this inequality over all inactive features $j \in \mathcal{A}^c$ gives the bound on the squared Frobenius norm. Taking the square root of both sides completes the proof.
\end{proof}

The cross-leverage score $h_j(\mathcal{A})$ measures alignment between the $j$-th inactive feature and the space of active features: if $h_j \to \frac{1}{n}\|a_j\|_2^2$, the inactive feature lies in the span of the active columns. The term $\frac{1}{\sqrt{n}}\sigma_{\max}(A_{\mathcal{A}}) = \sqrt{\lambda_{\max}(G_{\mathcal{A}\mathcal{A}})}$ acts as a gain factor. From~\citet{rudelson2010singularvalues}, the maximal singular value of the active submatrix scales purely with the sample size and sparsity as $\sigma_{\max}(A_{\mathcal{A}}) \sim \sqrt{n} + \sqrt{s} = \mathcal{O}(\sqrt{n})$. This means the normalized gain, $\frac{1}{\sqrt{n}}\sigma_{\max}(A_{\mathcal{A}})$, is naturally of order $\mathcal{O}(1)$. Thus, the active subspace does not inherently degrade the bound.

However, in highly correlated designs, even small active-inactive interaction is severely amplified. In tandem, the sum $\sum h_j$ explicitly captures the curse of dimensionality; because it sums over $d-s$ inactive features, it accumulates into a large $\mathcal{O}(\sqrt{d})$ spectral norm even when individual alignments are bounded.

\subsubsection{Active-Inactive Cross Leverage for Feature Pruning}

After fitting a model, we may use this bound to compute active-inactive cross-leverage scores for the unselected features. If a feature has a small score, it means it is mostly orthogonal to the active subspace and did not harm the conditioning of the regression with LASSO. Conversely, a large score indicates that an inactive variable strongly corresponds to the selected active set. Intuitively, this captures the geometric difficulty of feature selection: the regression is harmed (and optimization slowed) because the algorithm could have reasonably chosen similar, highly correlated features in place of the ones it actually selected. 

We demonstrate the practical utility of this diagnostic by performing a post-hoc pruning of inactive features based on their cross-leverage scores $h_j(\mathcal{A})$. This approach differs from existing a priori or in-path feature reduction techniques, such as safe screening~\citep{ghaoui2011safescreening} or sure screening~\citep{fan2008surescreening, cho2012tiltingfeaturescreening}, which are executed before or during the initial fit. The former certifies that discarded variables have zero coefficients in the LASSO solution, whereas sure screening aims to retain the true active set with probability tending to one under suitable assumptions. In particular,~\citet{cho2012tiltingfeaturescreening} evaluates the sample correlation between a candidate variable and all other variables to project away the influence of highly correlated features dynamically during the fitting process.

On the other hand, our pruning is performed \textit{a posteriori} only. We discard inactive features in $\hat{\mathcal{A}}^c$ with high cross-leverage scores to reduce the active-inactive interaction contribution to the affine Hoffman bound. For fixed data, preprocessing, and regularization, removing exactly inactive features preserves the original optimum; under uniqueness, the fitted coefficients and support remain unchanged. In this sense, our method is conceptually closest to a post-hoc version of the tilting procedure of~\citet{cho2012tiltingfeaturescreening}. 

We first examine the distribution of these scores on a synthetic high-dimensional design ($n=100, d=1000$) where the data generating process has a true sparsity of 15. As shown in \cref{fig:leverage-distribution}, the majority of inactive features possess low leverage in this Gaussian case. As the regularization $\lambda_n$ increases and the active set $s$ shrinks, the distribution of cross-leverage scores generally shifts toward zero, indicating that a more parsimonious model is naturally less susceptible to dual interaction leakage. Indeed, since the cross-leverage score is a way of bounding the active-inactive interaction from above, features with a low cross-leverage score contribute little individually to the interaction upper bound---although their aggregate contribution may remain substantial.

\begin{figure}[htbp]
    \centering
    \includegraphics[width=0.7\textwidth]{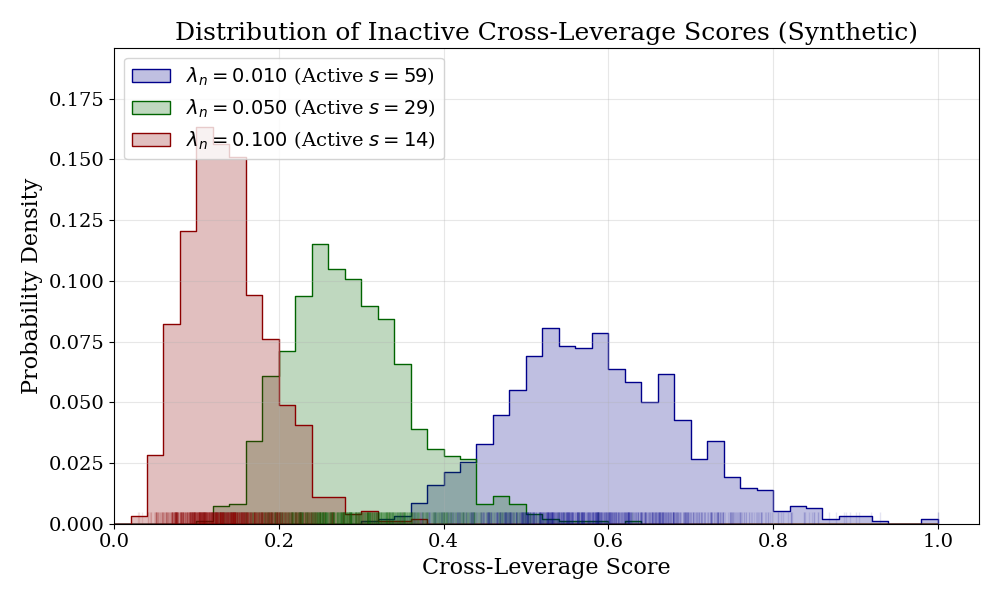}
    \caption{\textbf{Distribution of inactive cross-leverage scores for a Gaussian model.} Probability density of cross-leverage scores for unselected features across varying regularization strengths $\lambda_n$. Higher regularization yields smaller active sets, which in turn usually reduces the geometric alignment of the remaining ambient features, shifting the mass of the distribution toward the origin.}
    \label{fig:leverage-distribution}
\end{figure}

To illustrate the effect of pruning on affine conditioning and solver iterations in a real-world setting, we perform a pruning experiment on the ARCENE dataset ($d=10000, n=200$), a mass-spectrometry dataset~\citep{guyon2004neuripschallenge}. After a baseline LASSO fit, we discard inactive features exceeding a leverage threshold $h_t$ and re-solve the system. As summarized in \cref{tab:pruning-results}, discarding features with the a high cross-leverage scores ($h_t=0.95$) results in a reduction in the Hoffman constant. For $\lambda=0.1$, pruning 97\% of the ambient dimensions based on leverage allows us to reduce $H^{2,2}$ by more than half. 

In the high-regularization regime ($\lambda=1.0$), we achieve an 86\% reduction in solver iterations and drop 99\% of the ambient features without sacrificing the integrity of the originally recovered support. As a result of the smaller Hoffman constant, the stability under perturbation of the pruned solution is theoretically better, as in \eqref{eq:stability-under-perturbation}. More numerical results of this type can be found in~\citet[Section 4]{cho2012tiltingfeaturescreening}

\begin{table}[ht]
\centering
\caption{\textbf{Feature pruning on the ARCENE dataset ($d=10000, n=200$).} We compare the baseline LASSO solution to pruned models using a cross-leverage threshold $h_t=0.95$. Jaccard similarity measures support similarity with the baseline. $H^{2,2}$ is the exact affine Hoffman constant. The number of iterations of an ISTA procedure and relative reduction in iterations required to convergence compared to the baseline are reported.}
\label{tab:pruning-results}
\scriptsize
\begin{tabular}{cccccccc}
\toprule
\textbf{Regularization} ($\lambda_n$) & \textbf{Threshold} ($h_t$) & \textbf{Dropped} & \textbf{Sparsity} ($s$) & \textbf{$H^{2,2}$} & \textbf{Jaccard Sim.} & \textbf{Iters} & \textbf{Red. \%} \\ 
\midrule
0.1 & Baseline & 0 & 177 & 349.56 & 1.00 & 2178 & -- \\
    & 0.95 & 9784 & 174 & 157.24 & 0.98 & 1277 & 41.3\% \\
\midrule
1.0 & Baseline & 0 & 27 & 1.74 & 1.00 & 463 & -- \\
    & 0.95 & 9934 & 27 & 1.64 & 1.00 & 63 & 86.4\% \\
\bottomrule
\end{tabular}
\end{table}

\subsection{Examples of Structured Deterministic Designs}
\label{subsec:canonical-designs}

Although most high-dimensional statistical guarantees for the LASSO are stated for random designs, evaluating our explicit spectral bounds on a small set of deterministic, highly structured Gram matrices provides valuable insight. In general, one hopes to find designs such that the entirety of $G$ or an arbitrary active block $G_{\mathcal{A} \mathcal{A}}$ is structured. By analyzing these canonical designs, we can compute the restricted Hoffman constant in closed form to gain geometric intuition.

\subsubsection{Orthogonal Designs}

Suppose the design matrix $A_{\mathcal{A}}$ is orthogonal, yielding a fully diagonal Gram matrix $G_{\mathcal{A} \mathcal{A}} = \diag(\gamma_1, \dots, \gamma_d)$ with all $\gamma_j > 0$. Further suppose that $A_{\mathcal A^c}^{\top}A_{\mathcal A}=0$ so that the off-diagonal interaction block $G_{\mathcal{A}^c \mathcal{A}}$ vanishes entirely, and the restricted KKT matrix simplifies to:
\begin{equation}
    M_{\mathcal{A}} = \begin{pmatrix} \diag(\gamma_{\mathcal{A}}) & 0\\[4pt] 0 & \lambda_n I_{|\mathcal{A}^c|} \end{pmatrix}.
\end{equation}
This is a diagonal matrix with singular values $\gamma_j$ (for $j \in \mathcal{A}$) and $\lambda_n$. Consequently, applying \cref{bound:weyl-perturbation} yields the exact restricted Hoffman constant: 
\begin{equation}
    H^{2,2}_{\mathcal{S}_{\mathcal{A}}}(\Gamma_{\text{aff}, \mathcal{A}}) = \frac{1}{\min(\min_{j \in \mathcal{A}} \gamma_j, \lambda_n)}.
\end{equation}
This expressly demonstrates the empirical heuristic that if a selected feature possesses a small variance (i.e., $\gamma_j \approx 0$), the geometric penalty blows up, implying poor conditioning and slow convergence. Because $\gamma_j = \frac{1}{n} A^{\top}_j A_j = \frac{1}{n} \|A_j\|_2^2$ represents the sample variance of the $j$-th feature, this geometric perspective strongly motivates the use of variance filters and column normalization as a preprocessing step.

\subsubsection{Block-Diagonal Design}
In many scientific domains, high-dimensional data exhibits a modular or group-based correlation structure: in genomics, gene expression levels often form independent correlation networks corresponding to differing biological pathways; in spectroscopy, different functional groups resonate at separated frequency bands.

Consider a scenario where the features can be partitioned into $K$ disjoint groups, $\mathcal{I}_1,\dots,\mathcal{I}_K$, that are mutually orthogonal, even if intragroup features are highly correlated. Let $A=[A^{(1)},\dots,A^{(K)}]$, where $A^{(k)}$ contains the features for the $k$-th group. By orthogonality, the Gram matrix is block-diagonal: $G=\frac{1}{n}A^\top A=\operatorname{diag}(G^{(1)},\dots,G^{(K)})$. 

For any global active set $\mathcal{A}$, the local active set for the $k$-th group is defined as $\mathcal{A}_k = \mathcal{A} \cap \mathcal{I}_k$. By grouping indices corresponding to each block, the KKT matrix $M_{\mathcal{A}}$ can be permuted into a block-diagonal form:
\begin{equation}
    P M_{\mathcal{A}} P^\top = \operatorname{diag}\left(M_{\mathcal{A}_1}^{(1)}, M_{\mathcal{A}_2}^{(2)}, \dots, M_{\mathcal{A}_K}^{(K)}\right),
\end{equation}
where $P$ is a permutation matrix and each block $M_{\mathcal{A}_k}^{(k)}$ is exactly the KKT system for the subproblem on group $k$:
\begin{equation}
    M_{\mathcal{A}_k}^{(k)} = \begin{pmatrix} G^{(k)}_{\mathcal{A}_k\mathcal{A}_k} & 0 \\ G^{(k)}_{\mathcal{A}_k^c\mathcal{A}_k} & \lambda_n I \end{pmatrix}.   
\end{equation}
Because all off-diagonal interaction terms between groups are zero, the inverse of the matrix is the block-diagonal of its inverses. Since the spectral norm of a block-diagonal matrix is the maximum of the spectral norms of its blocks, the overall Hoffman constant is determined entirely by the worst-conditioned active group:
\begin{equation}
    \|M_{\mathcal{A}}^{-1}\|_2 = \left\| \operatorname{diag}\left((M_{\mathcal{A}_1}^{(1)})^{-1}, \dots, (M_{\mathcal{A}_K}^{(K)})^{-1}\right) \right\|_2 = \max_{k \in \{1, \dots, K\}} \|(M_{\mathcal{A}_k}^{(k)})^{-1}\|_2 = \max_{k \in \{1, \dots, K\}} H_{\mathcal{S}_{\mathcal{A}_k}}^{2,2}(\Gamma_{\text{aff}, \mathcal{A}_k}).
\end{equation}

\subsubsection{Population Geometry of the Spiked Model}
\label{subsubsec:spiked-model}

Our final example demonstrates why, at the population level, statistical curvature conditions do not by themselves control the active-inactive geometry of the KKT system and, more generally, the importance of explicit rather than existential convergence rate results. Consider the \textit{equicorrelated spiked covariance} model which has population covariance
\begin{equation}
    \label{eq:spiked-covariance}
    \Sigma
    :=
    (1-\rho)I_d+\rho\mathbf{1}_d\mathbf{1}_d^\top,
    \qquad
    \rho\in(0,1),
\end{equation}
as studied by~\citet[Equation 7.17]{wainwright2019highdimensional}. Equivalently, the rows $a_i^\top$ of the design matrix may be sampled
independently according to
\begin{equation}
    a_i\sim\mathcal N(0,\Sigma).
\end{equation}
The empirical Gram matrix is
\begin{equation}
    \widehat G:=\frac{1}{n}A^\top A,
\end{equation}
whereas $\Sigma=\mathbb E[\widehat G]$ is the population Gram matrix. In particular, when $d>n$, the empirical Gram matrix is singular and cannot equal the positive-definite matrix $\Sigma$.

For an active set $\mathcal A$ of cardinality
$s:=|\mathcal A|\ge 2$, the population active and interaction blocks are
\begin{equation}
    \Sigma_{\mathcal A\mathcal A}
    =
    (1-\rho)I_s+\rho\mathbf{1}_s\mathbf{1}_s^\top,
    \qquad
    \Sigma_{\mathcal A^c\mathcal A}
    =
    \rho\mathbf{1}_{d-s}\mathbf{1}_s^\top.
\end{equation}
The eigenvalues of $\Sigma_{\mathcal A\mathcal A}$ are $1+(s-1)\rho$ (with multiplicity one) corresponding to the eigenvector $\mathbf{1}_s$, and $1 - \rho$ (with multiplicity $s-1$). Consequently, the active curvature is given by the inverse of the minimum eigenvalue:
\begin{equation}
    \label{eq:spiked-active-inverse}
    \|\Sigma_{\mathcal A\mathcal A}^{-1}\|_2 = \frac{1}{1-\rho}.
\end{equation}

To compute the interaction term $\|\Sigma_{\mathcal \Sigma^c\mathcal A} \Sigma_{\mathcal A\mathcal A}^{-1}\|_2$, we note that $\mathbf{1}_s$ is an eigenvector of $\Sigma_{\mathcal{A}\mathcal{A}}^{-1}$ with eigenvalue $[1+(s-1)\rho]^{-1}$. Thus, the matrix product simplifies to:
\begin{equation}
    \Sigma_{\mathcal{A}^c \mathcal{A}} \Sigma_{\mathcal{A}\mathcal{A}}^{-1} = \left( \rho \mathbf{1}_{d-s} \mathbf{1}_s^\top \right) \Sigma_{\mathcal{A}\mathcal{A}}^{-1} = \rho \mathbf{1}_{d-s} \left( \Sigma_{\mathcal{A}\mathcal{A}}^{-1} \mathbf{1}_s \right)^\top = \frac{\rho}{1 + (s-1)\rho} \mathbf{1}_{d-s} \mathbf{1}_s^\top.
\end{equation}
This is a rank-one matrix whose spectral norm is trivially the product of the Euclidean norms of its generating vectors:
\begin{equation}
    \label{eq:spiked-population-interaction}
    \|\Sigma_{\mathcal{A}^c \mathcal{A}} \Sigma_{\mathcal{A}\mathcal{A}}^{-1}\|_2 = \frac{\rho}{1 + (s-1)\rho} \|\mathbf{1}_{d-s}\|_2 \|\mathbf{1}_s\|_2 = \frac{\rho \sqrt{s(d-s)}}{1 + (s-1)\rho}.
\end{equation}

To make the optimization interpretation precise, define the population active KKT matrix
\begin{equation}
    \overline M_{\mathcal A}
    :=
    \begin{pmatrix}
        \Sigma_{\mathcal A\mathcal A} & 0\\
        \Sigma_{\mathcal A^c\mathcal A} & \lambda_n I
    \end{pmatrix}.
\end{equation}
Substituting \eqref{eq:spiked-active-inverse} and \eqref{eq:spiked-population-interaction} into \cref{bound:dimension-free-bound} yields
\begin{equation}
    \label{eq:spiked-population-hoffman}
    H^{2,2}(\overline\Gamma_{\mathrm{aff},\mathcal A})
    \le
    \sqrt{
        \frac{1}{(1-\rho)^2}
        +
        \frac{1}{\lambda_n^2}
        \left[
            1+
            \left(
                \frac{\rho\sqrt{s(d-s)}}
                {1+(s-1)\rho}
            \right)^2
        \right]
    },
\end{equation}
where $\overline\Gamma_{\mathrm{aff},\mathcal A}$ denotes the solution set of the affine system defined by $\overline M_{\mathcal A}$.

The asymptotic behavior of the interaction term depends on the regime considered. If the sparsity level scales proportionally with the ambient dimension, e.g., $s=\alpha d$ for some fixed $\alpha \in (0,1)$, then
\begin{equation}
    \frac{\rho\sqrt{s(d-s)}}{1+(s-1)\rho} \to \sqrt{\frac{1-\alpha}{\alpha}} = \Theta(1), \qquad \text{as } d\to\infty.
\end{equation}
Thus, the interaction remains bounded when the support occupies a fixed proportion of the ambient dimension. In contrast, if the support size $s$ is held fixed while the ambient dimension $d$ increases, then the interaction term scales as
\begin{equation}
    \label{eq:spiked-population-interaction-scaling}
    \frac{\rho\sqrt{s(d-s)}}{1+(s-1)\rho} = \Theta(\sqrt{d}), \qquad \text{as } d\to\infty.
\end{equation}
The population active-inactive interaction can therefore deteriorate as correlated inactive variables are added, even though the active support and its cardinality remain unchanged.

Further, for every $v\in\mathbb R^d$,
\begin{equation}
    v^\top\Sigma v
    =
    (1-\rho)\|v\|_2^2
    +
    \rho\left(\sum_{j=1}^d v_j\right)^2
    \ge
    (1-\rho)\|v\|_2^2.
\end{equation}
Hence the population covariance has minimum eigenvalue $1-\rho>0$ and satisfies a population restricted eigenvalue bound with constant at least $1-\rho$. Furthermore, standard random-design results imply that, under suitable sample-size conditions, the empirical design drawn from this model satisfies the RE condition with high probability~\citep[Theorem 7.16 and Example 7.18]{wainwright2019highdimensional}. This does not require the empirical Gram matrix to be globally invertible.

The spiked model therefore illustrates that statistical recovery guarantees do not fully capture the population optimization complexity. Although the RE condition ensures stable estimation behavior, the optimization geometry can still deteriorate through two mechanisms: first, the loss of curvature as $\rho \to 1$, since $\|\Sigma_{\mathcal A\mathcal A}^{-1}\|_2 = (1-\rho)^{-1}$, and second, growth of active–inactive interaction when additional correlated inactive variables are introduced even while the support size $s$ and correlation $\rho$ remain fixed. It follows that i) a lower bound on population curvature alone does not provide a dimension-independent bound on the population KKT interaction and ii) even when linear convergence holds in theory, the associated theoretical constants may become arbitrarily poor in practice.

We emphasize that \eqref{eq:spiked-population-hoffman} is an explicit deterministic upper bound for the population KKT system. Transferring this formula to the empirical Hoffman constant requires additional concentration control for $\widehat G_{\mathcal A\mathcal A}$ and $\widehat G_{\mathcal A^c\mathcal A}$.

\begin{figure}[ht]
    \centering
    \includegraphics[width=0.7\textwidth]{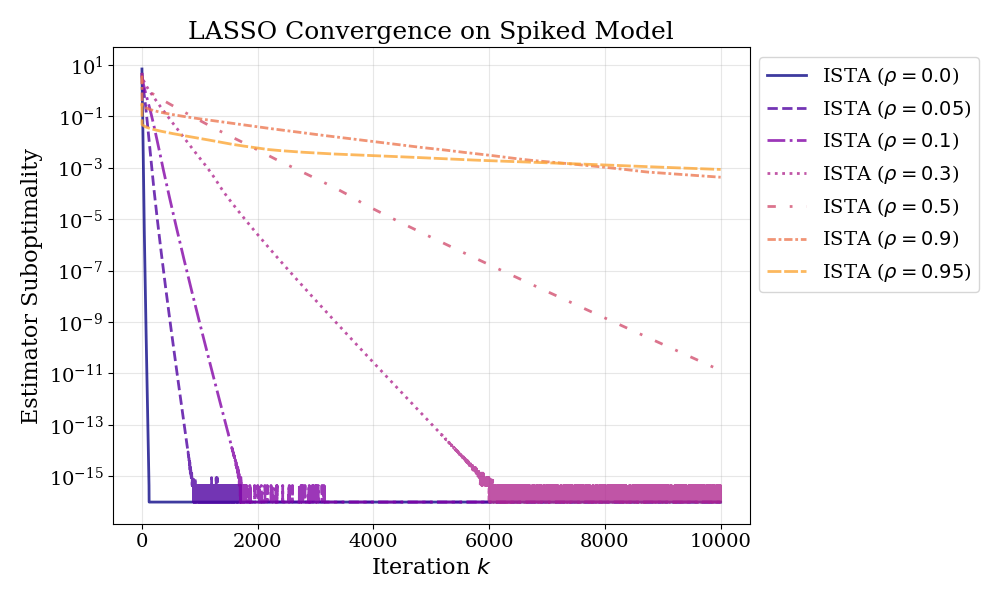}
    \caption{\textbf{LASSO convergence on the spiked model.} Estimator suboptimality for ISTA across varying correlation values $\rho \in [0, 0.95]$. The population covariance has minimum eigenvalue $1-\rho$, while the corresponding empirical designs satisfy the RE condition with high probability. Nevertheless, the observed ISTA convergence deteriorates as $\rho$ increases and can be made arbitrarily slow as $\rho \to 1$.}
    \label{fig:spiked-convergence}
\end{figure}

This phenomenon is empirically seen in \cref{fig:spiked-convergence}. As the correlation $\rho$ increases, the linear convergence rate of the ISTA solver degrades. This motivates a discussion of precise convergence rates beyond existential results of linear convergence. Further, it motivates the use of the Hoffman constant as a more faithful measure of optimization complexity than standard statistical assumptions.

\section{Connecting Optimization Geometry to Statistical Theory}
\label{sec:restricted-statistical-conditions}

The spiked model demonstrates a fundamental gap between statistics and optimization: conditions such as the RE property support stable estimation and prediction-error guarantees, but they fail to bound the active-inactive feature interaction which also affects the algorithmic behavior. 

To obtain favorable high-dimensional convergence guarantees of algorithms (e.g., proximal gradient methods) through these Hoffman-based bounds, we require control of the relevant Hoffman constant. That is, we must ensure that the restricted Hoffman constant $H^{2,2}_{\mathcal{S}_{\mathcal{A}}}$ is reasonably bounded. Looking at our explicit spectral bounds (e.g., \cref{bound:dimension-free-bound}), this requires simultaneous control over two specific matrix quantities: the active curvature $\|G_{\mathcal{A}\mathcal{A}}^{-1}\|_2$ and the active-inactive interaction $\|G_{\mathcal{A}^c\mathcal{A}}\|_2$.

In this section, we provide results to demonstrate how restricted conditions from the high-dimensional statistics affect different aspects of the Hoffman constant decomposition. Beyond the utility of these results, there is pedagogical value in seeing how these statistically-motivated conditions have a secondary geometric motivation. We prove a hierarchy of geometric stability: (i) the restricted nullspace property forces the KKT matrix to be invertible, and thus allows the existence of the explicit bounds from \cref{subsec:spectral-bounds}; (ii) the restricted eigenvalue condition bounds the primal curvature of the active set; (iii) the restricted isometry property bounds the primal curvature of any sparse active set, partially explaining the superior observed convergence rates; and (iv) that standard incoherence conditions translate poorly to the spectral setting.

\subsection{Restricted Nullspace Property Implies Invertible KKT system}

In \cref{sec:application-lasso}, we assumed the design matrix was in general position (\cref{assum:general-position}) to guarantee the active KKT matrix $M_{\mathcal{A}}$ is invertible. While general position holds almost surely for continuous random design matrices, it can fail for discrete data, categorical variables, and many structured deterministic designs.

If we relax the general position assumption, a weaker statistical condition sufficient to guarantee invertibility is the restricted nullspace property (RNP). 
\begin{definition}[Restricted nullspace property; \citet{wainwright2019highdimensional}]
    A design matrix $A \in \mathbb{R}^{n \times d}$ satisfies the RNP with respect to an active set $\mathcal{A}$ and parameter $\alpha \in \mathbb{R}_+$ if the intersection of the nullspace of $A$ and the restricted cone $\mathcal{C}_\alpha(\mathcal{A})$ defined in \eqref{eq: restricted cone} is trivial:
    \begin{equation}
        \mynull(A) \cap \mathcal{C}_\alpha(\mathcal{A}) = \{0\}.
    \end{equation}
\end{definition}
Note that this is weaker than the RE condition in the sense that if the RE condition holds with parameter $\gamma > 0$ for some active set $\mathcal{A}$ and some $\alpha \in \mathbb{R}_+$, then the RNP holds for $\mathcal{A}$ and all $\alpha' \in \mathbb{R}_+$ satisfying $\alpha' \le \alpha$.

This property acts as the structural requirement to guarantee that the local geometry is well-defined. While a finite Hoffman constant always exists for polyhedral sets (potentially utilizing the pseudoinverse), the exact invertibility of $M_{\mathcal{A}}$ is sufficient for $\Gamma_{\text{aff}, \mathcal{A}}$ to be a singleton and thus \cref{prop:bounded-regularity} to hold.

\begin{lemma}[The RNP implies an invertible KKT system]
Let $\mathcal{A}$ be a fixed active set. If the design matrix $A$ satisfies the RNP with respect to $\mathcal{A}$ for any $\alpha \in \mathbb{R}_+$, then the active KKT matrix $M_{\mathcal{A}}$ is nonsingular.
\end{lemma}
\begin{proof}
We prove the contrapositive: if $M_{\mathcal{A}}$ is singular, then the RNP must be violated.

Recall that $M_{\mathcal{A}}$ is block lower-triangular as in \eqref{eq:lasso-affine-system}. It is singular if and only if one of its diagonal blocks is singular. Since $\lambda_n I$ is positive definite (for $\lambda_n > 0$), $M_{\mathcal{A}}$ is singular if and only if $G_{\mathcal{A}\mathcal{A}} = \frac{1}{n}A_{\mathcal{A}}^\top A_{\mathcal{A}}$ is singular.

If $G_{\mathcal{A}\mathcal{A}}$ is singular, its columns are linearly dependent. Thus, there exists a nonzero vector $v_{\mathcal{A}} \in \mathbb{R}^{|\mathcal{A}|}$ such that $A_{\mathcal{A}} v_{\mathcal{A}} = 0$. Let us construct the full ambient vector $v \in \mathbb{R}^d$ by setting $v_{\mathcal{A}}$ on the active set and $v_{\mathcal{A}^c} = 0$. By construction, $Av = A_{\mathcal{A}} v_{\mathcal{A}} = 0$, so $v \in \text{null}(A)$.

We now check if $v \in \mathcal{C}_\alpha(\mathcal{A})$. The restricted cone is defined as $\{ u \in \mathbb{R}^d \mid \|u_{\mathcal{A}^c}\|_1 \le \alpha \|u_{\mathcal{A}}\|_1 \}$. For our constructed vector $v$, the inactive mass is $\|v_{\mathcal{A}^c}\|_1 = 0$ and the active mass is $\|v_{\mathcal{A}}\|_1 > 0$. The condition $0 \le \alpha \|v_{\mathcal{A}}\|_1$ holds trivially for any $\alpha \in \mathbb{R}_+$. 

Thus, $v \in \mathcal{C}_\alpha(\mathcal{A})$. Since $v$ is a nonzero vector and is in both the nullspace and the restricted cone, the RNP with respect to $\mathcal{A}$ is violated.
\end{proof}
A consequence of this result is that the RNP cannot hold if the active set is nonsparse ($|\mathcal{A}| > n$). If $|\mathcal{A}| > n$, the submatrix $A_{\mathcal{A}}$ is necessarily rank-deficient and must have a nontrivial nullspace, allowing the construction of a restricted null vector that violates the RNP. Indeed, RNP is a sufficient recovery condition~\citep[Theorem 7.8]{wainwright2019highdimensional}.

\subsection{Restricted Eigenvalue Condition Bounds Active Curvature}

We next discuss the relationship between the RE condition and the curvature in the active set. In classical high-dimensional statistics, efficient coefficient recovery theories such as \cite{bickel2009simultaneous} typically require the RE condition to hold \textit{uniformly} over all possible active sets of size $s$ so any potential sparse signal can be recovered. However, from an optimization perspective, evaluating conditioning uniformly over all $s$-sparse sets can be pathologically pessimistic. We illustrate this with a simple example.

\begin{example}[Pathological Inactive Features]
Consider a design matrix $A \in \mathbb{R}^{n \times 3}$ with columns $A_1 = \sqrt{n} e_1$ and $A_2 = A_3 = \sqrt{n} e_2$. Let the target be generated by the true sparse parameter vector $\beta^* = [\beta^*_1, 0, 0]^\top$ with $\beta^*_1 \neq 0$, such that $y = \sqrt{n} \beta^*_1 e_1 + \epsilon$. The true active set is $\mathcal{A}^* = \{1\}$, and the Gram matrix is:
\begin{equation}
    G = \frac{1}{n} A^\top A = \begin{pmatrix} 1 & 0 & 0 \\ 0 & 1 & 1 \\ 0 & 1 & 1 \end{pmatrix}.
\end{equation}
If we evaluate the RE condition uniformly over all active sets of sparsity $s=2$, we must consider the noise indices $\{2, 3\}$. The restricted cone for this subset contains the null direction $\Delta = [0, 1, -1]^\top$, for which $A\Delta = 0$. Consequently, the 2-sparse RE condition does not hold. Standard statistical theory cannot guarantee recovery because the ambient matrix cannot distinguish between the identically correlated noise features $2$ and $3$.

However, for the specific true active set $\mathcal{A}^* = \{1\}$, the active submatrix is $G_{\mathcal{A}^*\mathcal{A}^*} = 1$ and the interaction block is $G_{(\mathcal{A}^*)^c \mathcal{A}^*} = [0, 0]^\top$. The restricted Hoffman constant on the true active subspace is $H^{2,2}_{\mathcal{S}_{\mathcal{A}^*}}(\Gamma_{\text{aff}, \mathcal{A}^*}) = \frac{1}{\min(1, \lambda_n)}$. The geometric metric correctly identifies that the ambient singularity between features $2$ and $3$ is completely irrelevant to the optimization trajectory if the algorithm is initialized well, as those noise features are geometrically decoupled from the active signal.
\end{example}

This example highlights why the uniform $s$-sparse regime can be overly conservative and, more generally, why the optimization perspective may be interested in these constants at the level of active sets more so than at the level of certain sparsity thresholds. Therefore, our geometric framework evaluates the RE condition with respect to the specific active set $\mathcal{A}$. We now demonstrate how this condition explicitly bounds the eigenvalue of the active submatrix by exploiting the geometry of the restricted cone.

\begin{proposition}[The RE condition bounds active curvature]
    \label{prop:re-active-eigenvalue}
    Suppose the RE condition holds with $\gamma > 0$ with respect to the active set $\mathcal{A}$ and parameter $\alpha \in \mathbb{R}_+$. Then, the minimum eigenvalue of the active Gram submatrix satisfies $\lambda_{\min}(G_{\mathcal{A}\mathcal{A}}) \ge \gamma$.
\end{proposition}
\begin{proof}
    Let $v \in \mathbb{R}^{|\mathcal{A}|}$ be an arbitrary nonzero vector in the active subspace. We embed this vector into the full ambient dimension $\mathbb{R}^d$ by padding the inactive entries with zeros, defining $\Delta = (v, 0_{\mathcal{A}^c})$.
    
    We first verify that $\Delta$ lies within the restricted cone $\mathcal{C}_{\alpha}(\mathcal{A})$. For our constructed vector, the inactive mass is zero ($\|\Delta_{\mathcal{A}^c}\|_1 = 0$), and the active mass is positive ($\|\Delta_{\mathcal{A}}\|_1 = \|v\|_1 > 0$). The inequality condition $0 \le \alpha \|v\|_1$ holds trivially for any $\alpha \ge 0$.
    
    Thus, the entire subspace of vectors supported on $\mathcal{A}$ is contained within the restricted cone. By the definition of the RE condition, the following curvature bound holds for all vectors in the cone:
    \begin{equation}
        \frac{1}{n} \|A \Delta\|_2^2 \ge \gamma \|\Delta\|_2^2 \quad \forall \Delta \in \mathcal{C}_{\alpha}(\mathcal{A}).
    \end{equation}
    Substituting $\Delta = (v, 0_{\mathcal{A}^c})$ and recognizing that the quadratic form evaluates exactly to $\frac{1}{n}\|A\Delta\|_2^2 = v^\top G_{\mathcal{A}\mathcal{A}} v$, we obtain $v^\top G_{\mathcal{A}\mathcal{A}} v \ge \gamma \|v\|_2^2$. Since this inequality holds for any arbitrary $v \neq 0$, the Rayleigh quotient characterization guarantees that the minimum eigenvalue is bounded below by $\gamma$.
\end{proof}

\begin{bound}[Hoffman bound via the RE condition]
Assume the RE condition holds with $\gamma > 0$ with respect to the active set $\mathcal{A}$ and parameter $\alpha \in \mathbb{R}_+$. For any $\lambda_n > 0$, the restricted Hoffman constant $H^{2,2}_{\mathcal{S}_{\mathcal{A}}}(\Gamma_{\text{aff}, \mathcal{A}})$ has the following bound:
\begin{equation}
    H^{2,2}_{\mathcal{S}_{\mathcal{A}}}(\Gamma_{\text{aff}, \mathcal{A}}) \le \sqrt{ \frac{1}{\gamma^2} + \frac{1}{\lambda_n^2} \left( 1 + \frac{\|G_{\mathcal{A}^c \mathcal{A}}\|_2^2}{\gamma^2} \right) }.
\end{equation}
\end{bound}
\begin{proof}
    Recall \cref{bound:dimension-free-bound}:
    \begin{equation}
        H^{2,2}_{\mathcal{S}_{\mathcal{A}}}(\Gamma_{\text{aff}, \mathcal{A}}) \le \sqrt{ \|G_{\mathcal A\mathcal A}^{-1}\|_2^2 + \frac{1}{\lambda_n^2} \left( 1 + \|G_{\mathcal A^c\mathcal A} G_{\mathcal A\mathcal A}^{-1}\|_2^2 \right) }.
    \end{equation}
    By \cref{prop:re-active-eigenvalue}, the RE condition guarantees that $\lambda_{\min}(G_{\mathcal{A}\mathcal{A}}) \ge \gamma$. Because the active Gram submatrix $G_{\mathcal{A}\mathcal{A}}$ is symmetric positive definite, the spectral norm of its inverse is bounded by $\|G_{\mathcal{A}\mathcal{A}}^{-1}\|_2 \le \frac{1}{\gamma}$.
    
    Applying matrix submultiplicativity to the interaction term yields:
    \begin{equation}
        \|G_{\mathcal A^c\mathcal A}G_{\mathcal A\mathcal A}^{-1}\|_2 \le \|G_{\mathcal A^c\mathcal A}\|_2 \|G_{\mathcal A\mathcal A}^{-1}\|_2 \le \frac{\|G_{\mathcal A^c\mathcal A}\|_2}{\gamma}.
    \end{equation}
    Substituting these bounded components back into the spectral bound yields the stated result.
\end{proof}

\subsection{Restricted Isometry Property Bounds Active Curvature and Bilinear Interaction Uniformly}

While the local RE condition guarantees that the active submatrix is invertible and provides a curvature bound $\gamma$ on a particular active set, evaluating this constant can be analytically difficult as it depends on the specific geometry of the restricted cone. Furthermore, it only provides a lower bound on the curvature. In optimization, evaluating an upper bound on the eigenvalues is equally useful, as it determines the (restricted) Lipschitz constant of the active gradient and thus bounds the maximum allowable step-size in algorithms like ISTA~\citep{chaudhry2026trajectory}.

It has been observed empirically that when the design matrix $A$ satisfies the RIP, convergence tends to be much faster. We will now show that the RIP provides the stronger spectral structure required to bound the active curvature uniformly over all sufficiently sparse sets, naturally yielding both the lower bound for geometric stability and the upper bound for algorithmic step-sizes.

Recall that the RIP of order $s$ with constant $\delta_s \in (0, 1)$ guarantees that the eigenvalues of any active Gram submatrix $G_{\mathcal{A}\mathcal{A}}$ (where $|\mathcal{A}| \le s$) lie in the interval $[1-\delta_s, 1+\delta_s]$. The upper bound of this interval limits the Lipschitz constant of the active gradient ($L \le 1+\delta_s$), guaranteeing a safe, bounded optimization step-size. Simultaneously, the lower bound provides a uniform, dimension-independent limit on the active Hoffman curvature.

\begin{proposition}[RIP bounds active curvature]
    \label{prop:rip-active-curvature}
    Suppose the design matrix $A$ satisfies the RIP of order $s$ with constant $\delta_s < 1$. Then for any active set $\mathcal{A}$ with $|\mathcal{A}| \le s$, the spectral norm of the active inverse is bounded by:
    \begin{equation}
        \|G_{\mathcal{A}\mathcal{A}}^{-1}\|_2 \le \frac{1}{1-\delta_s}.
    \end{equation}
\end{proposition}
Substituting this into our explicit spectral bound (\cref{bound:dimension-free-bound}) yields a uniform Euclidean Hoffman bound over $s$-sparse active sets that depends on the RIP constant and the interaction term:
\begin{equation}
    H^{2,2}_{\mathcal{S}_{\mathcal{A}}}(\Gamma_{\text{aff}, \mathcal{A}}) \le \sqrt{ \frac{1}{(1-\delta_s)^2} + \frac{1}{\lambda_n^2} \left( 1 + \|G_{\mathcal{A}^c \mathcal{A}} G_{\mathcal{A}\mathcal{A}}^{-1}\|_2^2 \right) }.
\end{equation}

Beyond bounding the primal active curvature, a second mechanism of interest is whether the RIP can control the active-inactive interaction block. We now demonstrate that the RIP successfully constrains this interaction when evaluated as a bilinear form over the restricted cone, preventing arbitrary leakage into the inactive space along structured directions.

To establish this control, we rely on the following fundamental restricted orthogonality result, which bounds the inner product between disjoint sparse supports.
\begin{lemma}[Restricted orthogonality; \citet{candes2008rip}]
    \label{lem:candes-interaction}
    For a matrix satisfying the RIP of order $s + s'$ with constant $\delta_{s + s'}$, we have
    \begin{equation}
    \left| \left\langle \frac{1}{\sqrt{n}} A x, \frac{1}{\sqrt{n}} A x' \right\rangle \right| \le \delta_{s+s'} \|x\|_2 \|x'\|_2
    \end{equation}
    for all vectors $x, x'$ supported on disjoint subsets $T, T' \subseteq [d]$ with $|T| \le s$ and $|T'| \le s'$.
\end{lemma}
While restricted orthogonality bounds the interaction between the active set and any single $s$-sparse subset of inactive variables, evaluating stability over the full unrestricted ambient space requires bounding the entire matrix $G_{\mathcal{A}^c \mathcal{A}}$. If one trivially partitions the inactive set into $\lceil (d-s)/s \rceil$ disjoint blocks and applies the triangle inequality, the resulting operator norm scales linearly as $\mathcal{O}(d/s) \delta_{2s}$. This unrestricted bound causes the theoretical Hoffman constant to arbitrarily blow up in high dimensions; a more sophisticated argument is required.

\begin{lemma}[The RIP bounds the bilinear interaction term on the restricted cone]
    \label{lem:rip-interaction-term-bound}
    Suppose the design matrix $A$ satisfies the RIP of order $2s$ with constant $\delta_{2s}$. Let $v \in \mathbb{R}^{|\mathcal{A}|}$ be the restriction of a candidate vector to the active set ($|\mathcal{A}| \le s$), and $u \in \mathbb{R}^{d-|\mathcal{A}|}$ be the restriction to the inactive set.
    
    Assume the composite vector $(v, u) \in \mathcal{C}_{\alpha}(\mathcal{A})$, meaning $\|u\|_1 \le \alpha \|v\|_1$, for any $\alpha \in \mathbb{R}_+$. Then the bilinear interaction term is uniformly bounded independent of the ambient dimension $d$:
    \begin{equation}
        |u^\top G_{\mathcal{A}^c \mathcal{A}} v| \le \delta_{2s} \alpha (1 + \sqrt{s}) \|v\|_2^2.
    \end{equation}
\end{lemma}
\begin{proof}
    The proof relies on a shelling decomposition~\citep{candes2005decoding}. Note that the inactive vector $u \in \mathbb{R}^{d-|\mathcal{A}|}$ is not assumed to be sparse. To apply the RIP to this dense vector, we partition the support of $u$ into disjoint sets $T_1, T_2, \dots$, each of size $s$, ordered by descending magnitude: $T_1$ contains the indices of the $s$ largest absolute entries of $u$, $T_2$ the next $s$ largest, and so on. By the triangle inequality, we evaluate the entire dense tail as an infinite sum of sparse blocks:
    \begin{equation}
        |u^\top G_{\mathcal{A}^c \mathcal{A}} v| = \left| \sum_{j \ge 1} u_{T_j}^\top G_{T_j \mathcal{A}} v \right| \le \sum_{j \ge 1} | u_{T_j}^\top G_{T_j \mathcal{A}} v |.
    \end{equation}

    For any index $j$, the sets $T_j$ and $\mathcal{A}$ are disjoint by construction, and $|T_j \cup \mathcal{A}| \le 2s$. Applying \cref{lem:candes-interaction} yields $| u_{T_j}^\top G_{T_j \mathcal{A}} v | \le \delta_{2s} \|u_{T_j}\|_2 \|v\|_2$. Substituting this into our sum gives:
    \begin{equation}
        \label{eq:shelling-sum}
        |u^\top G_{\mathcal{A}^c \mathcal{A}} v| \le \delta_{2s} \|v\|_2 \sum_{j \ge 1} \|u_{T_j}\|_2.
    \end{equation}

    To bound the sum $\sum \|u_{T_j}\|_2$ for the tail ($j \ge 2$), we note that because the coefficients are sorted, the magnitude of every entry in block $T_j$ is bounded by the average magnitude of the entries in the preceding block $T_{j-1}$. Stechkin's lemma yields $\|u_{T_j}\|_2 \le \frac{1}{\sqrt{s}} \|u_{T_{j-1}}\|_1$. Summing this inequality over all $j \ge 2$:
    \begin{equation}
        \sum_{j \ge 2} \|u_{T_j}\|_2 \le \frac{1}{\sqrt{s}} \sum_{j \ge 2} \|u_{T_{j-1}}\|_1 \le \frac{1}{\sqrt{s}} \|u\|_1.
    \end{equation}
    Invoking the cone condition $\|u\|_1 \le \alpha \|v\|_1$ and the norm equivalence $\|v\|_1 \le \sqrt{s} \|v\|_2$, the tail sum is bounded by:
    \begin{equation}
        \sum_{j \ge 2} \|u_{T_j}\|_2 \le \frac{\alpha \|v\|_1}{\sqrt{s}} \le \frac{\alpha \sqrt{s} \|v\|_2}{\sqrt{s}} = \alpha \|v\|_2.
    \end{equation}

    For the head block ($j=1$), because the $\ell_2$ norm is bounded by the $\ell_1$ norm and $T_1$ is a subset of the inactive variables, we have $\|u_{T_1}\|_2 \le \|u_{T_1}\|_1 \le \|u\|_1 \le \alpha \sqrt{s} \|v\|_2$. Combining the head and tail bounds into \eqref{eq:shelling-sum} yields the final result:
    \begin{equation}
        |u^\top G_{\mathcal{A}^c \mathcal{A}} v| \le \delta_{2s} \|v\|_2 \left( \alpha \sqrt{s} \|v\|_2 + \alpha \|v\|_2 \right) = \delta_{2s} \alpha (1+\sqrt{s}) \|v\|_2^2.
    \end{equation}
\end{proof}

As mentioned, the RE condition bounds the active curvature $\|G_{\mathcal{A}\mathcal{A}}^{-1}\|_2$ but leaves the active-inactive interaction term $\|G_{\mathcal{A}^c\mathcal{A}}G_{\mathcal{A}\mathcal{A}}^{-1}\|_2$ unconstrained; some control over this interaction is still needed to control the Hoffman constant. For the RIP, bounding the bilinear form over the cone demonstrates that the active-inactive feature interactions remains locally well-behaved along structured directions. 

\subsection{Translating Incoherence Conditions to Bound Active-Inactive Interaction}

However, a major theoretical hurdle remains: translating these localized properties into an unconditional $\ell_2$ operator norm bound for the KKT matrix inverse inevitably introduces a $\sqrt{d}$ dimensional penalty. One might naturally hope that the MIC, which directly controls $\|G_{\mathcal{A}^c\mathcal{A}}G_{\mathcal{A}\mathcal{A}}^{-1}\|_\infty$, could fill this gap. We now show that this hope is \textit{misplaced}: translating the MIC into a spectral bound inevitably introduces a $\sqrt{d}$ dimensional penalty, and neither MIC nor the usual entrywise incoherence scaling avoids this
dimensional dependence in the bound.

The obstruction is a norm conversion problem. The MIC is an $\ell_\infty$ condition, whereas our Hoffman bounds require control of the $\ell_2$ operator norm. For an $(d-s) \times s$ matrix $M$, the sharpest conversion available is
\begin{equation}
    \label{eq:norm-conversion}
    \|M\|_2 \leq \sqrt{d-s}\,\|M\|_\infty.  
\end{equation}
If the MIC holds with margin $\eta \in (0,1]$, so that $\|G_{\mathcal{A}^c\mathcal{A}}G_{\mathcal{A}\mathcal{A}}^{-1}\|_\infty \leq 1 - \eta$, then~\eqref{eq:norm-conversion} yields
\begin{equation}
    \|G_{\mathcal{A}^c\mathcal{A}}G_{\mathcal{A}\mathcal{A}}^{-1}\|_2 
    \leq (1-\eta)\sqrt{d-s}.
\end{equation}
This bound grows without bound as $d \to \infty$, so MIC alone cannot yield a spectral Hoffman constant independent of ambient dimension via this route.

The population spiked model from \cref{subsubsec:spiked-model} provides an explicit example. For fixed support size $s$ and correlation $\rho \in (0,1)$,
\begin{equation}
    \left\|
    \Sigma_{\mathcal A^c\mathcal A}
    \Sigma_{\mathcal A\mathcal A}^{-1}
    \right\|_\infty
    =
    \frac{\rho s}{1+(s-1)\rho}
    =1-\eta,
    \qquad
    \eta=\frac{1-\rho}{1+(s-1)\rho}>0,
\end{equation}
so the MIC holds. On the other hand, \eqref{eq:spiked-population-interaction} and \eqref{eq:spiked-population-interaction-scaling} gives
\begin{equation}
    \left\|
    \Sigma_{\mathcal A^c\mathcal A}
    \Sigma_{\mathcal A\mathcal A}^{-1}
    \right\|_2
    =
    \frac{\rho\sqrt{s(d-s)}}{1+(s-1)\rho}
    =
    \Theta(\sqrt{d}).
\end{equation}
Thus, a fixed population MIC margin can coexist with
diverging spectral interaction, even though the population
curvature remains bounded below by $1-\rho$. This establishes the obstruction at the population level; transferring it to empirical high-dimensional designs requires additional probabilistic analysis.

One might attempt a more refined argument using the typical incoherence scaling. In standard sparse recovery analyses, pairwise incoherence conditions imply element-wise bounds of the form $\|G_{\mathcal{A}^c\mathcal{A}}G_{\mathcal{A}\mathcal{A}}^{-1}\|_{\max} \lesssim 1/s$, where $\|\cdot\|_{\max}$ denotes the maximum absolute entry. Bounding the spectral norm via the Frobenius norm then gives
\begin{equation}
    \|G_{\mathcal{A}^c\mathcal{A}}G_{\mathcal{A}\mathcal{A}}^{-1}\|_2 
    \leq \|G_{\mathcal{A}^c\mathcal{A}}G_{\mathcal{A}\mathcal{A}}^{-1}\|_F 
    \leq \sqrt{s(d-s)}\,\|G_{\mathcal{A}^c\mathcal{A}}G_{\mathcal{A}\mathcal{A}}^{-1}\|_{\max} 
    \lesssim \frac{\sqrt{s(d-s)}}{s} \sim \sqrt{\frac{d}{s}},
\end{equation}
which still diverges as $d \to \infty$ for any fixed sparsity $s$. Keeping this quantity bounded would require $s \gtrsim d$, contradicting the sparse regime $s \ll d$ that motivates the LASSO in the first place.

For the interaction term to remain bounded as $d \to \infty$ via element-wise control alone, one would need the substantially stronger scaling
\begin{equation}
    \|G_{\mathcal{A}^c\mathcal{A}}G_{\mathcal{A}\mathcal{A}}^{-1}\|_{\max} 
    \lesssim \frac{1}{\sqrt{sd}},
    \label{eq:required-scaling}
\end{equation}
since this would give $\|G_{\mathcal{A}^c\mathcal{A}}G_{\mathcal{A}\mathcal{A}}^{-1}\|_F \lesssim \sqrt{s(d-s)/sd} = \mathcal{O}(1)$. The scaling in~\eqref{eq:required-scaling} is far stronger than anything guaranteed by standard incoherence (e.g., IRC or MIC) analyses.

This establishes a limitation of the Euclidean optimization geometry and motivates a discussion of the Hoffman constant in other norms, along with what statistical conditions can be used to control it.

\section{Hoffman Bounds in Non-Euclidean Norms}
\label{sec:hoffman-other-norms}

Apart from the motivation from \cref{sec:restricted-statistical-conditions}, Hoffman constants in other norms control the convergence rate of other algorithms as well as more general $\ell_p$ stability under perturbation. For example, the $\ell_1$ and $\ell_\infty$ norms control convergence in the case of coordinate methods~\citep{nutini2015coordinate}. 

In this section, we derive explicit restricted Hoffman bounds for the Manhattan ($\ell_1$) and Chebyshev ($\ell_\infty$) geometries. By setting the domain and residual norms to $\|\cdot\|_V = \|\cdot\|_W = \|\cdot\|_p$ for $p \in \{1, \infty\}$, we maintain the separability requirement of \cref{prop:hoffman-separable-polyhedral}, ensuring that the polyhedral box and sign constraints do not worsen the conditioning. Thus, the geometry is entirely governed by the induced $\ell_p$ norm of the active KKT inverse $\|M_{\mathcal{A}}^{-1}\|_p$ and we may focus on the affine component.

\subsection{The Chebyshev Geometry (\texorpdfstring{$\ell_\infty$}{l-infinity})}

Evaluating the KKT system in the $\ell_\infty$ norm provides a worst-case sensitivity bound. The $\ell_\infty$ operator norm is defined as the maximum absolute row sum of the matrix. Geometrically, $H^{\infty, \infty}$ measures the maximum perturbation of any single primal parameter or dual subgradient resulting from simultaneous, worst-case perturbations across all constraints.

\begin{corollary}[The $\ell_\infty$ Hoffman constant]
    \label{cor:linf-affine-hoffman}
    The restricted $\ell_\infty$ Hoffman constant $H^{\infty,\infty}_{\mathcal{S}_{\mathcal{A}}}(\Gamma_{\text{aff}, \mathcal{A}})$ evaluates to:
    \begin{equation}
        H^{\infty,\infty}_{\mathcal{S}_{\mathcal{A}}}(\Gamma_{\text{aff}, \mathcal{A}}) = \max \left( \|G_{\mathcal{A}\mathcal{A}}^{-1}\|_\infty, \; \frac{1}{\lambda_n} \left( 1 + \|G_{\mathcal{A}^c\mathcal{A}} G_{\mathcal{A}\mathcal{A}}^{-1}\|_\infty \right) \right).
    \end{equation}
\end{corollary}
\begin{proof}
    The $\ell_\infty$ operator norm is the maximum absolute row sum. From the explicit block inverse \eqref{eq:inverse-formulation}, the maximum row sum of the active top block is $\|G_{\mathcal{A}\mathcal{A}}^{-1}\|_\infty$. For the inactive bottom block, the identity matrix $I$ adds exactly $1/\lambda_n$ to the absolute sum of every row, yielding a maximum row sum of $\frac{1}{\lambda_n}(1 + \|G_{\mathcal{A}^c\mathcal{A}} G_{\mathcal{A}\mathcal{A}}^{-1}\|_\infty)$. The result follows by taking the maximum of these two block suprema.
\end{proof}

\subsubsection{Mutual Incoherence and Pairwise Incoherence}

The algebraic structure of the inactive block reveals why the MIC is the natural statistical counterpart to the $\ell_\infty$ optimization geometry. If the design matrix satisfies the MIC with parameter $\eta \in (0, 1]$, the interaction term is bounded by $\|G_{\mathcal{A}^c\mathcal{A}} G_{\mathcal{A}\mathcal{A}}^{-1}\|_\infty \le 1 - \eta$. This controls the active-inactive feature interaction.

However, to obtain an $\ell_\infty$ bound explicitly independent of ambient dimension, we must also control the active inverse $\|G_{\mathcal{A}\mathcal{A}}^{-1}\|_\infty$. The MIC does not provide this. Instead, we require an additional assumption to enforce strict diagonal dominance on the active submatrix. In classical statistics, this is achieved via pairwise incoherence (PW)~\citep[Equation 7.12]{wainwright2019highdimensional}.
\begin{definition}[Pairwise incoherence; \citet{wainwright2019highdimensional}]
    A matrix $A$ satisfies PW with constant $\delta_{\text{PW}} \in (0, 1)$ if the coherence between any two columns of the normalized matrix is bounded:
    \begin{equation}
        \delta_{\text{PW}} := \max_{j,k \in [d], j\neq k} \left| \left\langle \frac{A_j}{\sqrt{n}}, \frac{A_k}{\sqrt{n}} \right\rangle \right|.
    \end{equation}
\end{definition}
Geometrically, this means the Gram matrix is a minor perturbation of the identity matrix. Note that pairwise incoherence is exactly the RIP of order 2, i.e., $\delta_2 = \delta_{PW}$; further the following general relationship holds for $2 \le s \le d$~\citep[Equation 7.15]{wainwright2019highdimensional}:
\begin{equation}
    \delta_{\text{PW}} \le \delta_s \le s \delta_{\text{PW}},
\end{equation}
and can be improved to $\delta_s = \sqrt{s}\delta_{\text{PW}}$ in some cases. If the pairwise incoherence constant satisfies $3s \delta_{PW} < 1$, it implies the RNP for $s$-sparse active sets~\citep[Proposition 7.9]{wainwright2019highdimensional}.

Under PW, we have the following $\ell_{\infty}$ Hoffman bound.

\begin{bound}[An $\ell_\infty$ Hoffman bound via incoherence]
    \label{bound:infinity-hoffman-bound-pw}
    Suppose the normalized design matrix $A$ satisfies the MIC with parameter $\eta \in (0, 1]$ and satisfies pairwise incoherence with constant $(s-1)\delta_{\text{PW}} < 1$. Then for any $s$-sparse active set $\mathcal{A}$, the $\ell_\infty$ Hoffman constant is uniformly bounded by:
    \begin{equation}
    H^{\infty, \infty}_{\mathcal{S}_{\mathcal{A}}}(\Gamma_{\text{aff}, \mathcal{A}}) \le \max \left( \frac{1}{1 - (s-1) \delta_{\text{PW}}}, \quad \frac{2 - \eta}{\lambda_n} \right).
    \end{equation}
\end{bound}
\begin{proof}
    We bound the two terms in the $\max$ operator from the preceding corollary separately. First, the MIC explicitly enforces $\| G_{\mathcal A^c\mathcal A}G_{\mathcal A\mathcal A}^{-1} \|_\infty \le 1 -\eta$. Substituting this directly into the inactive row sum immediately yields the upper bound $(2 - \eta)/\lambda_n$.

    Next, to bound $\|G_{\mathcal{A}\mathcal{A}}^{-1}\|_\infty$, we rely on the PW constant $\delta_{PW}$. Because the columns of the design matrix are normalized, the diagonal elements of the Gram matrix are exactly $1$. Thus, the active submatrix can be decomposed as $G_{\mathcal{A}\mathcal{A}} = I + E$, where $E$ contains the off-diagonal correlations. The maximum absolute row sum of $E$ is bounded by the pairwise coherence and the sparsity dimension $s$:
    \begin{equation}
        \|E\|_\infty = \max_{i \in \mathcal{A}} \sum_{j \in \mathcal{A}, j \neq i} |G_{ij}| \le (s-1) \delta_{PW}.
    \end{equation}
    Under the assumption that $(s-1)\delta_{PW} < 1$, the matrix $G_{\mathcal{A}\mathcal{A}}$ is diagonally dominant. This allows us to expand its inverse as a convergent Neumann series $G_{\mathcal{A}\mathcal{A}}^{-1} = \sum_{k=0}^\infty (-1)^k E^k$. Taking the induced $\ell_\infty$ norm of the series and applying the triangle inequality and submultiplicativity yields:
    \begin{equation}
        \|G_{\mathcal{A}\mathcal{A}}^{-1}\|_\infty \le \sum_{k=0}^\infty \|E\|_\infty^k = \frac{1}{1 - \|E\|_\infty} \le \frac{1}{1 - (s-1) \delta_{PW}}.
    \end{equation}
    Taking the maximum of these two bounded components completes the proof.
\end{proof}

The MIC assumption is a weak, standard condition in the literature for many recovery applications. Furthermore, the condition that $(s-1)\delta_{PW} < 1$ has strong precedent. In particular, the literature on basis pursuit establishes that the condition $(2s-1)\delta_{PW} < 1$ is sufficient for exact recovery in noiseless settings~\citep[Equation 1.12]{foucart2013compressivesensing}.

\subsubsection{Restricted Eigenvalue Conditions and Sample Complexity}

However, the PW condition comes at a heavy sample complexity cost. For standard Gaussian and sub-Gaussian ensembles to satisfy $3s\delta_{PW} < 1$ with high probability, a sufficient sample-size scaling is quadratic in sparsity: $n \gtrsim s^2 \log d$~\citep[Exercise 7.7]{wainwright2019highdimensional}. We examine two alternative approaches using $\ell_2$-RE and $\ell_\infty$-RE conditions and discuss the dimensional dependence and sufficient sample-size requirements of the
resulting bounds.

The first approach is to use the $\ell_{2}$-RE condition.

\begin{bound}[An $\ell_\infty$ Hoffman bound via the $\ell_2$-RE condition]
    Suppose the normalized design matrix $A$ satisfies the MIC with parameter $\eta \in (0, 1]$ and the $\ell_2$-RE condition with parameter $\gamma > 0$ over the active set $\mathcal{A}$. Then the restricted $\ell_\infty$ Hoffman constant is bounded by:
    \begin{equation}
        H^{\infty,\infty}_{\mathcal{S}_{\mathcal{A}}}(\Gamma_{\text{aff}, \mathcal{A}}) \le \max\left(\frac{\sqrt{s}}{\gamma}, \quad \frac{2-\eta}{\lambda_n}\right).
    \end{equation}
\end{bound}
\begin{proof}
    By the $\ell_2$-RE condition, $\lambda_{\min}(G_{\mathcal{A}\mathcal{A}}) \ge \gamma$, yielding the spectral bound $\|G_{\mathcal{A}\mathcal{A}}^{-1}\|_2 \le 1/\gamma$. To convert this into an $\ell_\infty$ bound, we apply standard matrix norm equivalence. For an $s \times s$ active matrix, this introduces a penalty factor of $\sqrt{s}$, giving $\|G_{\mathcal{A}\mathcal{A}}^{-1}\|_\infty \le \sqrt{s}\|G_{\mathcal{A}\mathcal{A}}^{-1}\|_2 \le \sqrt{s}/\gamma$. Substituting this active bound and the MIC interaction bound ($\|G_{\mathcal{A}^c\mathcal{A}} G_{\mathcal{A}\mathcal{A}}^{-1}\|_\infty \le 1 -\eta$) into the exact $\ell_\infty$ Hoffman decomposition yields the stated result.
\end{proof}

This approach allows us to drop the pairwise assumption, but the bound deteriorates as sparsity $s$ increases. One sufficient way to avoid the $\sqrt{s}$ factor introduced
by this norm conversion is to impose quantitative diagonal
dominance, as provided by pairwise incoherence with
$(s-1)\delta_{PW}$ bounded away from one.

Instead of bounding the full operator norm of the inverse, a second alternative is to define an $\ell_\infty$ Hoffman constant restricted to the error cone $\mathcal{C}_3(\mathcal{A})$, mirroring the cone-conditional Hoffman regime for the $\ell_2$ norm. To control the maximum element-wise prediction error $\|\hat{\beta} - \beta^*\|_\infty$ without a $\sqrt{s}$ penalty, a stronger geometric condition is required.

\begin{definition}[An $\ell_\infty$-RE condition; \citet{wainwright2019highdimensional}]
    In the case of Gaussian ensembles, the sample covariance matrix $\hat{\Sigma} = \frac{1}{n} A^\top A$ satisfies the $\ell_\infty$-curvature condition restricted to the cone if there exists a parameter $\gamma_\infty > 0$ such that:
    \begin{equation}
        \|\hat{\Sigma}\Delta\|_\infty \ge \gamma_\infty\|\Delta\|_\infty \quad \text{for all } \Delta \in \mathcal{C}_3(\mathcal{A}).
    \end{equation}
\end{definition}
By using this $\ell_\infty$-curvature condition, the stability of the active subspace in the $\ell_\infty$ norm can be directly bounded proportionally to $1/\gamma_\infty$. This condition avoids the explicit $\sqrt{s}$ factor in the preceding norm conversion. The cited sufficient conditions for Gaussian designs use sample sizes of order $n \gtrsim s^2\log d$~\citep[Exercises 7.13 and 7.14]{wainwright2019highdimensional}. This sufficient scaling does not, by itself, establish that quadratic dependence on sparsity is necessary.

\subsection{The Manhattan Geometry (\texorpdfstring{$\ell_1$}{l-1})}

While the $\ell_\infty$ norm controls element-wise error, the $\ell_1$ norm governs the absolute sum of the optimization error. The Hoffman constant in the $\ell_1$ norm corresponds to the maximum absolute column sum of the active inverse.

\begin{bound}[The $\ell_1$ Hoffman constant]
    \label{bound:l1-affine-hoffman}
    The restricted $\ell_1$ Hoffman constant $H^{1,1}_{\mathcal{S}_{\mathcal{A}}}(\Gamma_{\text{aff}, \mathcal{A}})$ is bounded by:
    \begin{equation}
        H^{1,1}_{\mathcal{S}_{\mathcal{A}}}(\Gamma_{\text{aff}, \mathcal{A}}) \le \max\left( \|G_{\mathcal{A}\mathcal{A}}^{-1}\|_1 + \frac{1}{\lambda_n} \|G_{\mathcal{A}^c\mathcal{A}} G_{\mathcal{A}\mathcal{A}}^{-1}\|_1, \quad \frac{1}{\lambda_n} \right).
    \end{equation}
\end{bound}
\begin{proof}
    The $\ell_1$ operator norm is the maximum absolute column sum of $M_{\mathcal A}^{-1}$. From the explicit block inverse \eqref{eq:inverse-formulation}, we partition the columns into the active block and the inactive block.
    
    For the inactive columns (the right half of the block matrix), the top block is zero, and the bottom block is $\frac{1}{\lambda_n} I$. The absolute sum of any inactive column is exactly $\frac{1}{\lambda_n}$.
    
    For the active columns $k \in \mathcal{A}$ (the left half), the absolute sum of the $k$-th column is the sum of its active primal block and its inactive dual block:
    \begin{equation}
        \|(G_{\mathcal{A}\mathcal{A}}^{-1})_{:,k}\|_1 + \frac{1}{\lambda_n} \|(G_{\mathcal{A}^c\mathcal{A}} G_{\mathcal{A}\mathcal{A}}^{-1})_{:,k}\|_1.
    \end{equation}
    Taking the maximum over all columns $k \in \mathcal{A}$, and applying the subadditivity of the maximum operator ($\max_k (a_k + b_k) \le \max_k a_k + \max_k b_k$), we obtain the upper bound for the active block's columns. Taking the maximum of the active and inactive column bounds yields the stated result.
\end{proof}

\subsubsection{Active Block \texorpdfstring{$\ell_1$}{l-1}-Stability via \texorpdfstring{$\ell_2$}{l-2}-RE}

The Euclidean RE condition also yields an upper bound on the
$\ell_1$ operator norm of the active inverse, with a factor
depending on the support size.

\begin{lemma}[Active block $\ell_1$-stability via the $\ell_2$-RE condition]
\label{lem:l1-stability-active}
Suppose the design matrix satisfies the $\ell_2$-RE condition over a nonempty active set $\mathcal{A}$ with curvature parameter $\gamma>0$ and cone parameter $\alpha\geq 0$. Writing $s=|\mathcal{A}|$, we have
\begin{equation}
    \|G_{\mathcal{A}\mathcal{A}}^{-1}\|_1
    \leq \frac{\sqrt{s}}{\gamma}.
\end{equation}
\end{lemma}
\begin{proof}
For any $v\in\mathbb{R}^{s}$, its extension by zeros outside
$\mathcal{A}$ lies in $\mathcal{C}_{\alpha}(\mathcal{A})$.
The RE condition therefore gives
\begin{equation*}
    v^\top G_{\mathcal{A}\mathcal{A}}v
    \geq \gamma\|v\|_2^2.
\end{equation*}
Consequently, $G_{\mathcal{A}\mathcal{A}}$ is positive definite and
\begin{equation*}
     \|G_{\mathcal{A}\mathcal{A}}^{-1}\|_2
    \leq \frac{1}{\gamma}.
\end{equation*}
For any $u\in\mathbb{R}^{s}$,
\begin{equation}
    \|G_{\mathcal{A}\mathcal{A}}^{-1}u\|_1
    \leq \sqrt{s}\,
        \|G_{\mathcal{A}\mathcal{A}}^{-1}u\|_2
    \leq \frac{\sqrt{s}}{\gamma}\|u\|_2
    \leq \frac{\sqrt{s}}{\gamma}\|u\|_1.
\end{equation}
Taking the supremum over nonzero $u$ proves the claim.
\end{proof}

Applying \cref{lem:l1-stability-active} to our $\ell_1$ Hoffman bound reveals two potential sources of dimensional dependence. First, the active-inverse bound retains a factor of $\sqrt{s}$, which is relatively mild. Second, bounding the interaction block in the $\ell_1$ norm ($\|G_{\mathcal A^c\mathcal A}G_{\mathcal A\mathcal A}^{-1}\|_1$) imposes a different statistical requirement than the MIC. Rather than bounding the total correlation of an inactive feature across all active features (row sum), the $\ell_1$ norm bounds the total correlation of a single active feature across \textit{all} inactive features simultaneously (column sum). Because the interaction block is a $(d-s) \times s$ matrix, summing down the columns accumulates the correlations over all $d-s$ inactive variables. These column sums can grow with the number of inactive features (and thus implicitly with the ambient dimension) when their contributions do not decay sufficiently.

\section{Applications to Other Sparse Polyhedral Regularizers}
\label{sec:application-other}

So far the focus has been on the LASSO. However, the prototypical decomposition described in \cref{sec:explicit-hc} can apply to a variety of sparse regularization problems. We now demonstrate how certain algorithm designs like the addition of an $\ell_2$ regularization penalty in the elastic net, the addition of an adaptive weight scheme in the adaptive LASSO, or the combination of the two in the adaptive elastic net influence or control different aspects of the Hoffman bound~\citep{zou2005elasticnet, zou2006adaptivelasso, zou2009adaptiveelasticnet}. This is significant as one can see these architectures as being motivated by the optimization geometry of the space, which can lead to further algorithmic improvements in the future.

\subsection{Elastic Net as a Spectral Smoother}

The elastic net regularizer introduces an $\ell_2$ regularization term to the LASSO.

\begin{problem}[The elastic net estimator; \citet{zou2005elasticnet}]
    \begin{equation}
    \hat{\beta} = \uargmin_{\beta \in \mathbb{R}^d} \frac{1}{2n} \|A\beta - y\|_2^2 + \lambda_{n,1} \|\beta\|_1 + \frac{\lambda_{n, 2}}{2} \|\beta\|_2^{2},
    \end{equation}
\end{problem}

Classically, the elastic net can be formulated as an augmented LASSO problem. We define the augmented design matrix $\tilde{A} \in \mathbb{R}^{(n+d) \times d}$ and target $\tilde{y} \in \mathbb{R}^{n+d}$ as:
\begin{equation}
    \tilde{A} = \begin{pmatrix} A \\ \sqrt{n \lambda_{n,2}} I_d \end{pmatrix}, \quad \tilde{y} = \begin{pmatrix} y \\ 0_d \end{pmatrix}.
\end{equation}
Solving the standard LASSO objective $\frac{1}{2n}\|\tilde{A}\beta - \tilde{y}\|_2^2 + \lambda_{n,1}\|\beta\|_1$ on this augmented dataset recovers the elastic net objective.

Because the $\ell_2$ penalty is smooth, it is absorbed into the gradient of the differentiable part of the objective:
\begin{equation}
    \nabla f(\beta) = \frac{1}{n} A^\top (A\beta - y) + \lambda_{n,2}\beta
\end{equation}
For a fixed active set $\mathcal{A}$, setting the gradient plus the $\ell_1$ subdifferential to zero yields the modified KKT equality:
\begin{equation}
    (G + \lambda_{n,2} I) \beta - u + \lambda_{n,1} g = 0.
\end{equation}
By partitioning this system into active  and inactive blocks, the local affine KKT system takes the form:
\begin{equation}
    \underbrace{\begin{pmatrix}
    G_{\mathcal{A}\mathcal{A}} + \lambda_{n,2} I & 0 \\
    G_{\mathcal{A}^c\mathcal{A}} & \lambda_{n,1} I
    \end{pmatrix}}_{:= M_{\mathcal{A}}^{\text{EN}}}
    \begin{pmatrix}
    \beta_{\mathcal{A}} \\
    g_{\mathcal{A}^c}
    \end{pmatrix}
    =
    \begin{pmatrix}
    u_{\mathcal{A}} - \lambda_{n,1} g_{\mathcal{A}} \\
    u_{\mathcal{A}^c}
    \end{pmatrix},
\end{equation}
which is the same as the LASSO with a perturbation to the top-left active block.

The $\ell_2$ penalty inflates the diagonal of the active Gram submatrix by $\lambda_{n,2}$. Because the normalized Gram matrix is symmetric positive semi-definite, an application of Weyl's inequality demonstrates that the minimum eigenvalue of the modified active block is bounded away from zero: 
\begin{equation}
    \label{eq:elastic-net-eigenvalue-chain}
    \lambda_{\min}(G_{\mathcal{A}\mathcal{A}} + \lambda_{n,2} I) \ge \lambda_{\min}(G_{\mathcal{A}\mathcal{A}}) + \lambda_{n,2} \ge \lambda_{n,2}.
\end{equation}
This artificial active curvature means that the restricted KKT matrix is invertible for \textit{any} active set, removing the need for the general position assumption or some RNP-like condition. Substituting this modified block inverse into our framework yields a bound on the Hoffman constant.
\begin{bound}[Spectral Hoffman bound for the elastic net]
    \label{bound:elastic-net}
    For any active set $\mathcal{A}$, including nonsparse active sets, the restricted $\ell_2$ Hoffman constant for the elastic net KKT system is bounded by:
    \begin{equation}
        H^{2, 2}_{\mathcal{S}_{\mathcal{A}}}(\Gamma_{\text{aff}, \mathcal{A}}^{\text{EN}}) \le \sqrt{ \frac{1}{\lambda_{n,2}^2} + \frac{1}{\lambda_{n,1}^2} \left( 1 + \frac{\|G_{\mathcal A^c\mathcal A}\|_2^2}{\lambda_{n,2}^2} \right) }.
    \end{equation}
\end{bound}
\begin{proof}
    Applying the dimension-free spectral bound from \cref{bound:dimension-free-bound} to the elastic net KKT matrix $M_{\mathcal{A}}^{\text{EN}}$ yields:
    \begin{equation}
        H^{2, 2}_{\mathcal{S}_{\mathcal{A}}}(\Gamma_{\text{aff}, \mathcal{A}}^{\text{EN}}) \le \sqrt{ \|(G_{\mathcal{A}\mathcal{A}} + \lambda_{n,2} I)^{-1}\|_2^2 + \frac{1}{\lambda_{n,1}^2} \left( 1 + \|G_{\mathcal A^c\mathcal A} (G_{\mathcal{A}\mathcal{A}} + \lambda_{n,2} I)^{-1}\|_2^2 \right) }.
    \end{equation}
    Because $\lambda_{\min}(G_{\mathcal{A}\mathcal{A}} + \lambda_{n,2} I) \ge \lambda_{n,2}$, the spectral norm of the inverse is bounded by $\|(G_{\mathcal{A}\mathcal{A}} + \lambda_{n,2} I)^{-1}\|_2 \le 1/\lambda_{n,2}$. Applying matrix submultiplicativity to the interaction term and substituting these components yields the stated bound.
\end{proof}
Of course, under some further RE condition, the bound for sparse active sets can be improved by using the lower bound on the eigenvalue of the active Gram matrix as in \eqref{eq:elastic-net-eigenvalue-chain}. 

\subsubsection{A Geometric View on Hyperparameter Selection}
\label{subsubsec:elastic-net-hyperparam-selection}

Based on the optimization-geometric perspective developed so far, we present another hyperparameter selection scheme.

One approach to selecting optimal hyperparameters is to minimize the variational bound on the Hoffman constant for a fixed regularization budget $\Lambda_n = \lambda_{n,1} + \lambda_{n,2}$. Since minimizing \cref{bound:elastic-net} is equivalent to minimizing its square, we can formulate this as:
\begin{problem}[Hyperparameter selection for the elastic net]
    \label{prob:direct-hyperparam-elastic-net}
    \begin{align}
    \min_{\lambda_{n,1},\lambda_{n,2}} \quad & \frac{1}{\lambda_{n,2}^2} + \frac{1}{\lambda_{n,1}^2} \left( 1 + \frac{\|G_{\mathcal{A}^c\mathcal{A}}\|_2^2}{\lambda_{n,2}^2} \right) \nonumber \\
    \text{subject to} \quad & \lambda_{n,1} + \lambda_{n,2} = \Lambda_n, \quad \lambda_{n,1}, \lambda_{n,2} > 0.
    \end{align}
\end{problem}
We instead consider a maximum-envelope surrogate that balances the active-curvature and inactive-interaction contributions. This alternative criterion admits an explicit characterization of its minimizer and provides an interpretable hyperparameter selection rule. The optimality statements below concern this envelope rather than the original objective in \cref{prob:direct-hyperparam-elastic-net}. Writing $U$ for the square root of that objective, the envelope $E$ defined below satisfies $(1/\sqrt{2})U \leq E \leq \sqrt{2}U$. Thus, the envelope approximates the original Hoffman upper bound within a uniform multiplicative factor.

We define the surrogate envelope $\mathcal{E}(\lambda_{n,1}, \lambda_{n,2})$ for the worst-case spectral bound in the following problem:
\begin{equation}
    \mathcal{E}(\lambda_{n,1}, \lambda_{n,2}) := \max \left( \frac{1}{\lambda_{n,2}}, \quad \frac{1}{\lambda_{n,1}} \left( 1 + \frac{\|G_{\mathcal{A}^c \mathcal{A}}\|_2}{\lambda_{n,2}} \right) \right),
\end{equation}
which can be used as the surrogate objective to minimize in \cref{prob:direct-hyperparam-elastic-net}.

\begin{lemma}[Naive optimal geometric stability of the relaxed envelope]
    \label{lem:naive-optimal-elasticnet}
    Consider a fixed total regularization budget $\Lambda_n = \lambda_{n,1} + \lambda_{n,2}$. In the regime of weak active-inactive interaction ($\|G_{\mathcal{A}^c \mathcal{A}}\|_2 \le \Lambda_n/3$), the maximum-envelope surrogate $\mathcal{E}(\lambda_{n,1}, \lambda_{n,2})$ is minimized exactly when the hyperparameters satisfy the balance condition:
    \begin{equation}
        \lambda_{n,1}^* - \lambda_{n,2}^* = \|G_{\mathcal{A}^c \mathcal{A}}\|_2.
    \end{equation}
    Furthermore, for large regularization budgets where $\Lambda_n \gg \|G_{\mathcal{A}^c \mathcal{A}}\|_2$, this geometry simplifies to the equipartition rule $\lambda_{n,1}^* \approx \lambda_{n,2}^* \approx \Lambda_n/2$.
\end{lemma}
\begin{proof}
    See \cref{app:proof-naive-optimal-elasticnet}.
\end{proof}

If we apply this same surrogate minimization to the tighter bound that explicitly includes the native active curvature $\lambda_{\min}(G_{\mathcal{A}\mathcal{A}})$ from \eqref{eq:elastic-net-eigenvalue-chain}, the modified envelope becomes:
\begin{equation}
    \mathcal{E}_{\text{tight}}(\lambda_{n,1}, \lambda_{n,2}) := \max \left( \frac{1}{\lambda_{n,2} + \lambda_{\min}(G_{\mathcal{A}\mathcal{A}})}, \quad \frac{1}{\lambda_{n,1}} \left( 1 + \frac{\|G_{\mathcal{A}^c \mathcal{A}}\|_2}{\lambda_{n,2} + \lambda_{\min}(G_{\mathcal{A}\mathcal{A}})} \right) \right).
\end{equation}
Minimizing this tighter surrogate shifts the optimal balance intersection to:
\begin{equation}
    \lambda_{n,1}^* - \lambda_{n,2}^* =  \lambda_{\min}(G_{\mathcal{A}\mathcal{A}}) + \|G_{\mathcal{A}^c \mathcal{A}}\|_2.
\end{equation}
This equation demonstrates that the intrinsic curvature of the active data $\lambda_{\min}(G_{\mathcal{A}\mathcal{A}})$ subsidizes the geometric requirement for the artificial $\ell_2$ penalty from $\lambda_{n,2}$. If the active features possess curvature, less $\ell_2$ regularization is required to optimize the Hoffman constant, sparing the estimator from unnecessary shrinkage bias. The elastic net must thus inject exactly enough artificial curvature to compensate for the data's specific eigenvalue deficit relative to its correlation leakage. We can bound the approximation error between the surrogate bound with and without considering the intrinsic curvature. Notably, this error decays quadratically with the regularization budget $\Lambda_n$.

\begin{lemma}[Naive approximation gap of the relaxed envelope]
    \label{lem:naive-envelope-gap}
    Let $\mathcal{E}(\lambda_{n,1}, \lambda_{n,2})$ be the worst-case surrogate envelope and $\mathcal{E}_{\text{tight}}(\lambda_{n,1}, \lambda_{n,2})$ be the tight surrogate that includes the native active curvature $\lambda_{\min}(G_{\mathcal{A}\mathcal{A}})$. For any fixed hyperparameters $\lambda_{n,1}, \lambda_{n,2} > 0$, the pointwise approximation gap is bounded by:
    \begin{equation}
        0 \le \mathcal{E}(\lambda_{n,1}, \lambda_{n,2}) - \mathcal{E}_{\text{tight}}(\lambda_{n,1}, \lambda_{n,2}) \le \frac{\lambda_{\min}(G_{\mathcal{A}\mathcal{A}})}{\lambda_{n,2}(\lambda_{n,2} + \lambda_{\min}(G_{\mathcal{A}\mathcal{A}}))} \max \left( 1, \frac{\|G_{\mathcal{A}^c \mathcal{A}}\|_2}{\lambda_{n,1}} \right).
    \end{equation}
    Furthermore, if $\|G_{\mathcal{A}^c \mathcal{A}}\|_2 \le \Lambda_n/3$ and $\Lambda_n > \|G_{\mathcal{A}^c \mathcal{A}}\|_2 + \lambda_{\min}(G_{\mathcal{A} \mathcal{A}})$, the exact difference between their minimized optimal values evaluates to:
    \begin{equation}
        \mathcal{E}^* - \mathcal{E}_{\text{tight}}^* = \frac{2\lambda_{\min}(G_{\mathcal{A}\mathcal{A}})}{(\Lambda_n - \|G_{\mathcal{A}^c \mathcal{A}}\|_2)(\Lambda_n - \|G_{\mathcal{A}^c \mathcal{A}}\|_2 + \lambda_{\min}(G_{\mathcal{A}\mathcal{A}}))}.
    \end{equation}
\end{lemma}
\begin{proof}
    See \cref{app:proof-naive-envelope-gap}.
\end{proof}

For the envelope in the weak-interaction regime, the result that maximum convergence speed occurs when $\lambda_{n,1} \approx \lambda_{n,2}$ is a strong, counter-intuitive theoretical claim, as practitioners often treat $\lambda_{n,2}$ as a small nuisance parameter to make the Gram matrix $G$ invertible. In the case of low interaction, optimal geometric stability is actually achieved at the balance of the two.

However, while \cref{lem:naive-optimal-elasticnet} establishes an unconstrained optimum, it ignores the \textit{statistical validity} of the resulting estimator---it is in this sense that we call these naive results. In practice, the $\ell_1$ penalty cannot be chosen arbitrarily and must be sufficiently large relative to the projected-noise scale; exact support recovery additionally requires active sign consistency and strict dual feasibility. As established in \citet[Theorem 7.21]{wainwright2019highdimensional}, for the KKT system to avoid false positive selections, the $\ell_1$ parameter must dominate the projected noise. Let $\lambda_{\text{stat}}$ denote this strict statistical lower bound, such that we require $\lambda_{n,1} \ge \lambda_{\text{stat}} \gtrsim \|\frac{1}{n}A^\top\epsilon\|_\infty$, for example.~\footnote{Although the value $\epsilon$ is not available in practice, 
surrogates using the high-dimensional central limit theorem or bootstrap are available~\citep{chernozhukov2023high}.}

We now formulate the statistically constrained geometric optimization problem, which is of the same form as \cref{prob:direct-hyperparam-elastic-net} subject to the baseline recovery guarantee:
\begin{problem}[Hyperparameter selection for the elastic net under statistical validity constraints]
    \label{prob:direct-hyperparam-elastic-net-stat-val}
    \begin{align}
    \min_{\lambda_{n,1},\lambda_{n,2}} \quad & \frac{1}{\lambda_{n,2}^2} + \frac{1}{\lambda_{n,1}^2} \left( 1 + \frac{\|G_{\mathcal{A}^c\mathcal{A}}\|_2^2}{\lambda_{n,2}^2} \right) \nonumber \\
    \text{subject to} \quad &\lambda_{n,1}+\lambda_{n,2}=\Lambda_n,\quad\lambda_{n,1}\ge\lambda_{\text{stat}}, \quad \lambda_{n,1}, \lambda_{n,2} > 0.
    \end{align}
\end{problem}
Again, maximum-envelope formulations can be used as surrogates.

\begin{theorem}[Statistically constrained geometric optimum of the relaxed envelope]
    \label{thm:statistically-constrained-optimum}
    Consider a fixed total regularization budget $\Lambda_n=\lambda_{n,1}+\lambda_{n,2}$ and a statistical lower bound $\lambda_{\text{stat}}$ with $0 < \lambda_{\text{stat}} < \Lambda_n$. The surrogate envelope formulation of \cref{prob:direct-hyperparam-elastic-net-stat-val} admits a closed-form solution. The unconstrained geometric optimum $\lambda_{n,1}^{\text{unc}}$ exhibits a phase transition dependent on the spectral leakage $\|G_{\mathcal{A}^c\mathcal{A}}\|_2$:
    \begin{enumerate}
    \item \textbf{Weak interaction regime} ($\|G_{\mathcal{A}^c\mathcal{A}}\|_2\le\Lambda_n/3$):
    \begin{equation}
        \lambda_{n,1}^{\text{unc}}=\frac{\Lambda_n+\|G_{\mathcal{A}^c\mathcal{A}}\|_2}{2}.
    \end{equation}
    \item \textbf{Strong interaction regime} ($\|G_{\mathcal{A}^c\mathcal{A}}\|_2>\Lambda_n/3$):
    \begin{equation}
        \lambda_{n,1}^{\text{unc}}=\Lambda_n+\|G_{\mathcal{A}^c\mathcal{A}}\|_2-\sqrt{\|G_{\mathcal{A}^c\mathcal{A}}\|_2(\Lambda_n+\|G_{\mathcal{A}^c\mathcal{A}}\|_2)}.
    \end{equation}
    \end{enumerate}
    The hyperparameters that minimize the relaxed envelope are thus
    \begin{equation}
        \lambda_{n,1}^*=\max\left(\lambda_{n,1}^{\text{unc}},\lambda_{\text{stat}}\right), \quad \lambda_{n,2}^*=\Lambda_n-\lambda_{n,1}^*
    \end{equation}
\end{theorem}
\begin{proof}
    See \cref{app:proof-statistically-constrained-optimum}.
\end{proof}

Recall the augmented system $\tilde{A}=\begin{pmatrix}A\\\sqrt{n\lambda_{n,2}}I_d\end{pmatrix}$ and $\tilde{y}=\begin{pmatrix}y\\0_d\end{pmatrix}$. The true data-generating process is $y=A\beta^*+\epsilon$. By expressing the augmented response as $\tilde{y}=\tilde{A}\beta^*+\tilde{\epsilon}$, we can solve for the augmented noise vector:
\begin{equation}
    \tilde{\epsilon}=\tilde{y}-\tilde{A}\beta^*=\begin{pmatrix}A\beta^*+\epsilon\\0_d\end{pmatrix}-\begin{pmatrix}A\beta^*\\\sqrt{n\lambda_{n,2}}\beta^*\end{pmatrix}=\begin{pmatrix}\epsilon\\-\sqrt{n\lambda_{n,2}}\beta^*\end{pmatrix}.
\end{equation}
The strict dual feasibility condition (preventing false positives) requires the $\ell_1$ penalty to suppress the projection of this noise onto the feature space. Computing this projection yields:
\begin{equation}
    \frac{1}{n}\tilde{A}^\top\tilde{\epsilon}=\frac{1}{n}\begin{pmatrix}A^\top&\sqrt{n\lambda_{n,2}}I_d\end{pmatrix}\begin{pmatrix}\epsilon\\-\sqrt{n\lambda_{n,2}}\beta^*\end{pmatrix}=\frac{1}{n}A^\top\epsilon-\lambda_{n,2}\beta^*.
\end{equation}
We contrast this with the standard LASSO formulation (when $\lambda_{n,2}=0$). In the regular LASSO, the effective noise projection reduces to the baseline statistical error $\frac{1}{n}A^\top\epsilon$. Consequently, the statistical threshold $\lambda_{\text{stat}}$ depends only on the variance of the noising process through $\epsilon$. 

On the other hand, the elastic net induces a deterministic shrinkage bias on the active features, namely $-\lambda_{n,2}\beta_{\mathcal{A}}^*$, which acts as an additional source of effective noise. We now make the corresponding inactive KKT calculation precise. Let
\begin{equation}
    q:=\frac{1}{n}A^\top\epsilon,
    \qquad
    H(t):=G_{\mathcal A\mathcal A}+tI,
    \qquad
    s_{\mathcal A}
    :=
    \operatorname{sign}(\widehat\beta_{\mathcal A}).
\end{equation}
For the candidate active set $\mathcal A$, the active KKT equation gives
\begin{equation}
    \widehat\beta_{\mathcal A}
    =
    \beta_{\mathcal A}^*
    +
    H(\lambda_{n,2})^{-1}
    \left(
        q_{\mathcal A}
        -
        \lambda_{n,1}s_{\mathcal A}
        -
        \lambda_{n,2}\beta_{\mathcal A}^*
    \right).
\end{equation}
Define the unnormalised inactive KKT residual by
\begin{equation}
    r_{\mathcal A^c}
    :=
    u_{\mathcal A^c}
    -
    G_{\mathcal A^c\mathcal A}
    \widehat\beta_{\mathcal A}.
\end{equation}
Substitution of the active solution yields
\begin{equation}
    \label{eq:elastic-net-inactive-residual}
    r_{\mathcal A^c}
    =
    q_{\mathcal A^c}
    -
    G_{\mathcal A^c\mathcal A}
    H(\lambda_{n,2})^{-1}q_{\mathcal A}
    +
    \lambda_{n,1}
    G_{\mathcal A^c\mathcal A}
    H(\lambda_{n,2})^{-1}s_{\mathcal A}
    \nonumber
    +
    \lambda_{n,2}
    G_{\mathcal A^c\mathcal A}
    H(\lambda_{n,2})^{-1}\beta_{\mathcal A}^*.
\end{equation}
Strict dual feasibility requires $\|r_{\mathcal A^c}\|_\infty<\lambda_{n,1}$. Together with active sign consistency, this condition provides the usual primal-dual certificate for support recovery.

To obtain a tractable sufficient condition, suppose for this statistical analysis that $G_{\mathcal A\mathcal A}\succ0$, and define
\begin{align}
    \Upsilon
    &:=
    \sup_{0\leq t\leq\Lambda_n}
    \left\|
        G_{\mathcal A^c\mathcal A}
        H(t)^{-1}s_{\mathcal A}
    \right\|_\infty,\\
    \Xi_0
    &:=
    \sup_{0\leq t\leq\Lambda_n}
    \left\|
        q_{\mathcal A^c}
        -
        G_{\mathcal A^c\mathcal A}
        H(t)^{-1}q_{\mathcal A}
    \right\|_\infty.
\end{align}
Assume that $\Upsilon<1$. The shrinkage term can be bounded uniformly by observing that
\begin{equation}
    H(t)^{-1}
    =
    G_{\mathcal A\mathcal A}^{-1}
    G_{\mathcal A\mathcal A}H(t)^{-1}
\end{equation}
and
\begin{equation}
    \left\|
        G_{\mathcal A\mathcal A}H(t)^{-1}
    \right\|_2
    \leq 1.
\end{equation}
Therefore,
\begin{equation}
    \left\|
        G_{\mathcal A^c\mathcal A}
        H(t)^{-1}\beta_{\mathcal A}^*
    \right\|_\infty
    \leq
    \left\|
        G_{\mathcal A^c\mathcal A}
        G_{\mathcal A\mathcal A}^{-1}
    \right\|_{2\to\infty}
    \|\beta_{\mathcal A}^*\|_2
    =:\Omega_0,
\end{equation}
where $\|\cdot\|_{2\to\infty}$ denotes the induced operator norm from $\ell_2$ to $\ell_\infty$.

Applying these bounds to
\eqref{eq:elastic-net-inactive-residual} gives
\begin{equation}
    \|r_{\mathcal A^c}\|_\infty
    \leq
    \Xi_0
    +
    \Upsilon\lambda_{n,1}
    +
    \Omega_0\lambda_{n,2}.
\end{equation}
Hence, a sufficient condition for strict dual feasibility is
\begin{equation}
    (1-\Upsilon)\lambda_{n,1}
    >
    \Xi_0+\Omega_0\lambda_{n,2}
    \quad\Longrightarrow\quad
    \|r_{\mathcal A^c}\|_\infty<\lambda_{n,1}.
\end{equation}
Defining
\begin{equation}
    \label{eq:elastic-net-coupled-constants}
    \Xi
    :=
    \frac{\Xi_0}{1-\Upsilon},
    \qquad
    \Omega
    :=
    \frac{\Omega_0}{1-\Upsilon},
\end{equation}
we obtain the conservative linear coupling condition
\begin{equation}
    \label{eq:elastic-net-linear-coupling}
    \lambda_{n,1}
    >
    \Xi+\Omega\lambda_{n,2}.
\end{equation}
Thus, allocating additional budget to $\lambda_{n,2}$ improves active curvature but can simultaneously increase the sufficient $\ell_1$ penalty required for dual feasibility. This coupling effect transforms the original deterministic noise-suppression problem into a coupled problem where a higher $\lambda_{n,2}$ demands a higher $\lambda_{n,1}$ to suppress its own bias.

This motivates us to bound the approximation error between the coupled and uncoupled systems. The following result quantifies the cost of imposing the closed relaxation $\lambda_{n,1}\geq\Xi+\Omega\lambda_{n,2}$.

\begin{lemma}[Approximation error from coupling constraint]
    \label{lem:uncoupled-gap}
    Suppose that $\|G_{\mathcal A^c\mathcal A}\|_2 \le \Lambda_n / 3$ and that $(\Lambda_n + \|G_{\mathcal A^c\mathcal A}\|_2)/2 < \Xi < \Lambda_n$.
   
    Let $\mathcal E_{\mathrm{uncoupled}}^*$ denote the minimum of the relaxed envelope under
    \begin{equation}
        \lambda_{n,1}+\lambda_{n,2}=\Lambda_n,
        \qquad
        \lambda_{n,1}\geq\Xi,
    \end{equation}
    and let $\mathcal E_{\mathrm{coupled}}^*$ denote its minimum under the conservative coupled constraint
    \begin{equation}
        \lambda_{n,1}+\lambda_{n,2}=\Lambda_n,
        \qquad
        \lambda_{n,1}\geq\Xi+\Omega\lambda_{n,2}.
    \end{equation}
    Then
    \begin{equation}
        \mathcal E_{\mathrm{uncoupled}}^*
        = 
        \frac{1}{\Lambda_n-\Xi},
        \qquad
        \mathcal E_{\mathrm{coupled}}^*
        =
        \frac{1+\Omega}{\Lambda_n-\Xi}.
    \end{equation}
    Consequently,
    \begin{equation}
        \mathcal E_{\mathrm{coupled}}^*
        -
        \mathcal E_{\mathrm{uncoupled}}^*
        =
        \frac{\Omega}{\Lambda_n-\Xi},
        \quad 
        \text{and, }
        \quad
        \frac{
            \mathcal E_{\mathrm{coupled}}^*
        }{
            \mathcal E_{\mathrm{uncoupled}}^*
        }
        =
        1+\Omega.
    \end{equation}
\end{lemma}

The lemma shows that the conservative coupling produces an absolute envelope gap of order $(\Lambda_n-\Xi)^{-1}$. Thus, for fixed $\Omega$,
the absolute gap decreases as the available budget above the baseline statistical floor increases. The relative difference factor, however,
remains $1+\Omega$, and the corresponding hyperparameter allocations need not converge.

More concretely, \cref{fig:en-phase-transition} provides support recovery and empirical Hoffman heatmaps over varying regularization parameters (normalized by the noise floor $\Xi$) for a fixed elastic net problem. The theoretical curves are overlaid. A few observations are notable: the naive and constrained closed-form curves from \cref{lem:naive-optimal-elasticnet} and \cref{thm:statistically-constrained-optimum} differ only for very small values of $\lambda_{n,1}$ (where $\lambda_{n,1} / \Xi < 1$). Both underestimate the $\lambda_{n,1}$ regularization required when $\lambda_{n,1} / \Xi < 10$ due to the coupling described in \eqref{eq:elastic-net-linear-coupling}; beyond this point, the statistical recovery constraints are no longer the limiting factor and the theoretical paths become visually close, consistent with the decay of the absolute envelope gap in \cref{lem:uncoupled-gap}. Finally, we observe that when the interaction regime boundary $\Lambda_n = 3 \|G_{\mathcal{A}^c\mathcal{A}}\|_2$ intersects the paths, the split shifts in the fashion described by the phase transition in \cref{thm:statistically-constrained-optimum}.

\begin{figure}[ht]
    \centering
    \includegraphics[width=0.9\textwidth]{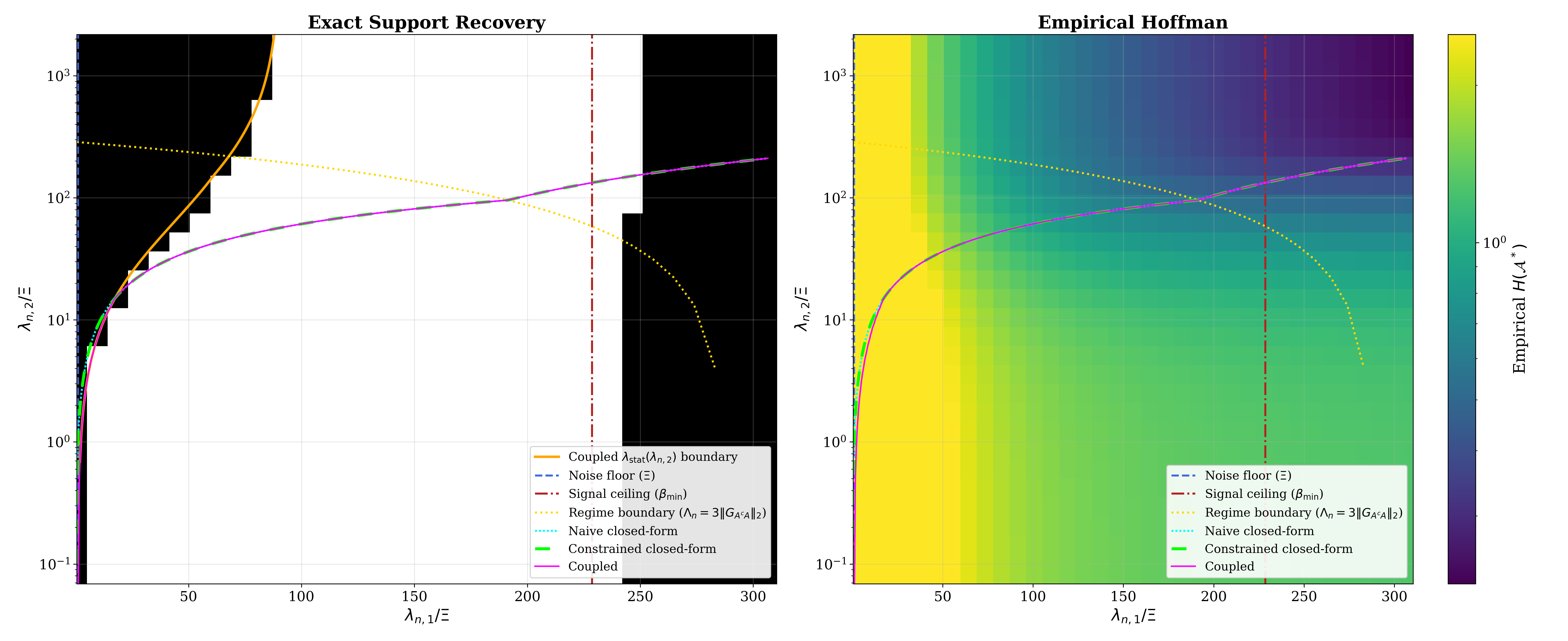}
    \caption{\textbf{Elastic net phase transition map on a Gaussian design ($n=300, d=500, s=5$).} (Left) Exact support recovery (white) for varying regularization parameters against theoretical boundaries. The coupled $\lambda_{\text{stat}}(\lambda_{n, 2})$ boundary correctly identifies the rightward shift in the noise floor caused by internal shrinkage bias. (Right) Empirical affine Hoffman constant of the true support for varying regularization parameters. Note the $y$-axis is log-scaled in both plots.}
    \label{fig:en-phase-transition}
\end{figure}

While it is theoretically possible to find the optimal regularization parameters in an oracle sense, both the noise parameter $\epsilon$ and the final active set are unknown a priori in practice. However, as mentioned in the footnote, these may be estimated. For example, a shorter pilot run of LASSO may be used to estimate the active set $\mathcal{A}$. In the case of a Gaussian design, if the noise variance $\sigma^2$ is known, concentration inequalities may be used to estimate the noise floor $\Xi$ directly. Or, more generally, the bootstrap or the high-dimensional central limit theorem are available~\citep{chernozhukov2023high}.

\subsection{Adaptive LASSO as an Interaction Dampener}

While the elastic net stabilizes the active curvature, its geometry does not intrinsically resolve the active-inactive feature interaction penalty. To address this, the adaptive LASSO modifies the regularization penalty by applying a unique feature-specific weight $w_j \ge 0$ to each coefficient.

\begin{problem}[The adaptive LASSO estimator; \citet{zou2006adaptivelasso}]
    \begin{equation}
    \hat{\beta} = \uargmin_{\beta \in \mathbb{R}^d} \frac{1}{2n} \|A\beta - y\|_2^2 + \lambda_{n} \sum_{j=1}^d w_j |\beta_j|.
    \end{equation}
\end{problem}

In practice, the weights are typically chosen as $w_j = 1/|\hat{\beta}^{\text{init}}_j|^\zeta$ for some $\zeta > 0$, where $\hat{\beta}^{\text{init}}$ is an initial consistent estimator (such as the OLS or standard LASSO estimator). If a feature is a true noise variable, its initial estimate tends to 0 ($\hat{\beta}^{\text{init}}_j \to 0$) causing its penalty weight to diverge ($w_j \to \infty$).

Geometrically, this feature-specific weighting alters the separable polyhedral constraints $\Gamma_{\text{con}}$ to nonuniform, rescaled box and sign constraints. That is, the new KKT condition is given by:
\begin{equation}
    \frac{1}{n} A^\top (A\beta - y) + \lambda_n 
    W g = 0,
\end{equation}
where $W = \diag(w)$, and $w_j g_j = w_j\sign(\beta_j)$ if $\beta_j \neq 0$ or $w_j g_j \in [-w_j, w_j]$ if $\beta_j = 0$. Note that, since \cref{prop:hoffman-separable-polyhedral} applies to weighted $\ell_p$ norms, the Hoffman bound on $\Gamma_{\text{con}}$ still applies.

To construct the modified KKT system, we define the active and inactive diagonal weight matrices $W_{\mathcal{A}} = \operatorname{diag}(w_j)_{j \in \mathcal{A}}$ and $W_{\mathcal{A}^c} = \diag(w_j)_{j \in \mathcal{A}^c}$ respectively. The affine KKT system takes the form:
\begin{equation}
    \underbrace{\begin{pmatrix}
    G_{\mathcal{A}\mathcal{A}} & 0 \\
    G_{\mathcal{A}^c\mathcal{A}} & \lambda_n W_{\mathcal{A}^c}
    \end{pmatrix}}_{:= M_{\mathcal{A}}^{\text{Ada}}}
    \begin{pmatrix}
    \beta_{\mathcal{A}} \\
    g_{\mathcal{A}^c}
    \end{pmatrix}
    =
    \begin{pmatrix}
    u_{\mathcal{A}} - \lambda_n W_{\mathcal{A}} \sign(\beta_{\mathcal{A}}) \\
    u_{\mathcal{A}^c}
    \end{pmatrix}.
\end{equation}
Assuming some condition to make $G_{\mathcal{A}\mathcal{A}}$ invertible, and assuming $w_j > 0$ for all inactive features, the block inverse is given by:
\begin{equation}
    (M_{\mathcal{A}}^{\text{Ada}})^{-1} = 
    \begin{pmatrix}
    G_{\mathcal{A}\mathcal{A}}^{-1} & 0 \\[4pt]
    -\frac{1}{\lambda_n} W_{\mathcal{A}^c}^{-1} G_{\mathcal{A}^c\mathcal{A}} G_{\mathcal{A}\mathcal{A}}^{-1} & \frac{1}{\lambda_n} W_{\mathcal{A}^c}^{-1}
    \end{pmatrix}.
\end{equation}
Observe that the active curvature block $G_{\mathcal{A}\mathcal{A}}^{-1}$ is completely unchanged from the standard LASSO. However, the interaction block is now modulated by the inverse weight matrix $W_{\mathcal{A}^c}^{-1} = \operatorname{diag}(1/w_j)_{j \in \mathcal{A}^c}$.

\begin{theorem}[Spectral Hoffman bound for the adaptive LASSO]
    The restricted $\ell_2$ Hoffman constant for the adaptive LASSO KKT system is bounded by:
    \begin{equation}
        H^{2, 2}_{\mathcal{S}_{\mathcal{A}}}(\Gamma_{\text{aff}, \mathcal{A}}^{\text{Ada}}) \le \sqrt{ \|G_{\mathcal{A}\mathcal{A}}^{-1}\|_2^2 + \frac{1}{\lambda_n^2} \max_{j \in \mathcal{A}^c} \left( \frac{1}{w_j^2} \right) + \frac{1}{\lambda_n^2} \|W_{\mathcal{A}^c}^{-1} G_{\mathcal A^c\mathcal A} G_{\mathcal{A}\mathcal{A}}^{-1}\|_2^2 }.
    \end{equation}
\end{theorem}
\begin{proof}
    Applying the variational expansion from \cref{bound:dimension-free-bound} to $(M_{\mathcal{A}}^{\text{Ada}})^{-1}$, we substitute the identity matrix with $W_{\mathcal{A}^c}^{-1}$. The squared spectral norm of the diagonal weight block is the maximum squared diagonal entry, yielding $\max_{j \in \mathcal{A}^c} (1/w_j^2)$. The rest of the bound follows identically by submultiplicativity.
\end{proof}
This result provides a geometric interpretation of adaptive reweighting. The inverse weight matrix $W_{\mathcal A^c}^{-1}$ rescales the rows of the active-inactive interaction block, so larger inactive weights attenuate their contributions to the Hoffman bound. However, divergence of the weights alone does not guarantee that these contributions vanish: their behavior also depends on the regularization parameter, the interaction geometry, and any growth in dimension. Specifically, writing $w_{\min,n}:=\min_{j\in\mathcal A^c}w_j$, we have
\begin{equation}
    \frac{1}{\lambda_n}
    \|
    W_{\mathcal A^c}^{-1}
    G_{\mathcal A^c\mathcal A}G_{\mathcal A\mathcal A}^{-1}
    \|_2
    \leq
    \frac{
    \|G_{\mathcal A^c\mathcal A}G_{\mathcal A\mathcal A}^{-1}\|_2
    }{
    \lambda_n w_{\min,n}
    }.
\end{equation}
Thus, adaptive reweighting provides an alternative mechanism for controlling the interaction contribution without directly imposing incoherence conditions on the unweighted design, provided the inactive weights grow sufficiently fast relative to the interaction norm and the inverse regularization parameter.

\subsubsection{The Adaptive Elastic Net}

Having established that the elastic net stabilizes active curvature and the adaptive LASSO dampens active-inactive interaction, it is natural to consider their synthesis. Proposed by \citet{zou2009adaptiveelasticnet}, the adaptive elastic net combines a squared $\ell_2$ penalty with a weighted $\ell_1$ penalty, achieving the oracle property even in highly collinear scenarios where the standard adaptive LASSO fails.

\begin{problem}[The adaptive elastic net estimator; \citet{zou2009adaptiveelasticnet}]
    \begin{equation}
    \hat{\beta} = \uargmin_{\beta \in \mathbb{R}^d} \frac{1}{2n} \|A\beta - y\|_2^2 + \lambda_{n,1} \sum_{j=1}^d w_j |\beta_j| + \frac{\lambda_{n, 2}}{2} \|\beta\|_2^2.
    \end{equation}
\end{problem}

We can immediately construct the affine KKT system for this hybrid regularizer by combining the augmented active block of the elastic net with the weighted inactive bounds of the adaptive LASSO. Let $W_{\mathcal{A}^c} = \operatorname{diag}(w_j)_{j \in \mathcal{A}^c}$. The KKT block matrix is given by:
\begin{equation}
    M_{\mathcal{A}}^{\text{AEN}} = 
    \begin{pmatrix}
    G_{\mathcal{A}\mathcal{A}} + \lambda_{n,2} I & 0 \\
    G_{\mathcal{A}^c\mathcal{A}} & \lambda_{n,1} W_{\mathcal{A}^c}
    \end{pmatrix}.
\end{equation}
Because the active block is positive definite due to $\lambda_{n,2}$, and assuming $w_j > 0$ for all inactive features, the block inverse exists for any active set:
\begin{equation}
    (M_{\mathcal{A}}^{\text{AEN}})^{-1} = 
    \begin{pmatrix}
    (G_{\mathcal{A}\mathcal{A}} + \lambda_{n,2} I)^{-1} & 0 \\[4pt]
    -\frac{1}{\lambda_{n,1}} W_{\mathcal{A}^c}^{-1} G_{\mathcal{A}^c\mathcal{A}} (G_{\mathcal{A}\mathcal{A}} + \lambda_{n,2} I)^{-1} & \frac{1}{\lambda_{n,1}} W_{\mathcal{A}^c}^{-1}
    \end{pmatrix}.
\end{equation}

\begin{bound}[Euclidean Hoffman bound for the adaptive elastic net]
    For any active set $\mathcal{A}$, including nonsparse active sets, the restricted $\ell_2$ Hoffman constant for the adaptive elastic net is bounded by:
    \begin{equation}
        H^{2, 2}_{\mathcal{S}_{\mathcal{A}}}(\Gamma_{\text{aff}, \mathcal{A}}^{\text{AEN}}) \le \sqrt{ \frac{1}{\lambda_{n,2}^2} + \frac{1}{\lambda_{n,1}^2} \max_{j \in \mathcal{A}^c} \left( \frac{1}{w_j^2} \right) + \frac{\|W_{\mathcal{A}^c}^{-1} G_{\mathcal A^c\mathcal A}\|_2^2}{\lambda_{n,1}^2 \lambda_{n,2}^2} }.
    \end{equation}
\end{bound}
\begin{proof}
    Applying the variational expansion from \cref{bound:dimension-free-bound} to $(M_{\mathcal{A}}^{\text{AEN}})^{-1}$ and using submultiplicativity on the interaction term yields:
    \begin{equation}
        \|W_{\mathcal{A}^c}^{-1} G_{\mathcal A^c\mathcal A} (G_{\mathcal{A}\mathcal{A}} + \lambda_{n,2} I)^{-1}\|_2 \le \|W_{\mathcal{A}^c}^{-1} G_{\mathcal A^c\mathcal A}\|_2 \|(G_{\mathcal{A}\mathcal{A}} + \lambda_{n,2} I)^{-1}\|_2 \le \frac{\|W_{\mathcal{A}^c}^{-1} G_{\mathcal A^c\mathcal A}\|_2}{\lambda_{n,2}}.
    \end{equation}
    Squaring this bounded term and substituting the bounds for the diagonal blocks yields the stated result.
\end{proof}
As in \cref{bound:elastic-net}, this result can be tightened when some bound on $\lambda_{\min}(G_{\mathcal{A}\mathcal{A}})$ is provided. From a geometric perspective, it is now easy to see how adaptive elastic net creates a well-conditioned polyhedral landscape by controlling both geometric sources of instability (active set curvature and active-inactive feature interaction) which otherwise lead to loose Hoffman-type bounds.

\section{Discussion}
\label{sec:discussion}

In this paper, we derived a decomposition for bounding the Hoffman constant of general sparse polyhedral regularizers. In doing so, we isolated the contributions of the affine, polyhedral, and interaction components. When the same absolute norm is used on the domain and residual space, the polyhedral component reduces exactly to unity. Furthermore, under mild uniqueness conditions---such as general position for sparse active sets or the RNP---the geometric interaction penalty collapses, allowing the linear regularity constant to be evaluated exactly.

By discussing the properties of the LASSO estimator, we demonstrated that the affine component permits many bounding techniques which depend on the choice of norm; for example, in the $\ell_2$ norm, techniques for bounding the minimum positive singular value become pertinent. We provide a deterministic, dimension-free bound using matrix perturbation techniques. These bounds further emphasize that the affine component depends on curvature in the active set and feature interaction between the active and inactive sets.  

In order to bound these two sources of instability in the affine component, we considered conditions from high-dimensional statistics that hold with high probability of random design matrices. The example of the spiked model in \cref{subsubsec:spiked-model} demonstrated that the RE condition was not sufficient to control the Hoffman bound. Indeed, in the $\ell_2$ norm, the RE and RIP conditions are sufficient to only control the active curvature over some active set or all sparse active sets respectively. Further incoherence conditions like the MIC or the IRC are required to bound the active-inactive feature interaction, yet scale poorly when translated to the $\ell_2$ setting. This naturally motivated a discussion for Hoffman bounds in other norms ($\ell_1$ and $\ell_\infty$) which other statistical conditions could control.

Finally, we concluded by extending the analysis to the elastic net and adaptive LASSO. Elastic net provides a lower bound on the active curvature but leaves the active-inactive feature interaction unaffected; adaptive LASSO, on the other hand, bounds the active-inactive feature interaction but leaves the active curvature term unaffected; it is shown that adaptive elastic net controls both these terms by combining their effects. Hence, providing an explicit, interpretable decomposition of the Hoffman bound permits an immediate geometric justification for these methods, and provides a framework for designing novel regularization schemes that simultaneously optimize statistical recovery, algorithmic convergence rates, and geometric stability. Further, in the case of elastic net, we used the variational Hoffman bound as a method of determining the optimal hyperparameters---namely, the $\ell_1$ and $\ell_2$ penalties that minimize the upper bound on the Hoffman constant---subject to some statistical recovery condition, a method which can be used in a wide range of settings.

We conclude with directions for future work:
\begin{itemize}
    \item \textbf{Relaxing uniqueness and sparsity assumptions.} Our analysis has focused on the sparse regime where, under mild assumptions, the KKT block matrix is invertible and the bounded regularity constant can be evaluated exactly. For nonsparse active sets during early optimization iterates, the affine subspace relies on the pseudoinverse $M_{\mathcal{A}}^\dagger$ in the spectral case, and $\Gamma_{\text{aff}}$ is no longer a singleton. Bounding the intersection angle $\tau(\Gamma_{\text{aff}}, \Gamma_{\text{con}})$ in this regime remains an open problem. A promising direction is the use of more granular nullspace conditions, such as the stable or robust nullspace properties (SNP/RoNP)~\citep{foucart2013compressivesensing}. While the standard RNP is binary and fails nonsparse active sets, the SNP and RoNP provide descriptive geometric bounds on the $\ell_2$ and $\ell_1$ mass of the nullspace vectors, which could potentially be mapped directly to the linear regularity constant.

    Moreover, in settings involving discrete or categorical features, the design matrix may fail to be in general position with nonnegligible probability, leading to noninvertible KKT systems.

    \item \textbf{Extensions to other polyhedral regularizers.} The results provided in \cref{sec:explicit-hc} apply to a more general class of polyhedral (separable) regularizers. Because this geometric decomposition relies fundamentally on the intersection of affine subspaces with separable polyhedra, it provides a natural pathway to analyze other piecewise-linear or $\ell_\infty$-bounded estimators, such as sparse classification via the $\ell_1$ support vector machines, robust regression via least absolute deviations, and structured sparsity via the $\ell_1/\ell_\infty$ group LASSO. 
    
    Further, extending this geometric decomposition to nonseparable regularizers, such as the fused LASSO or generalized total variation, introduces a nontrivial polyhedral geometry on $\Gamma_{\text{con}}$. Because these regularizers impose incidence matrices and pairwise parameter constraints, the linear regularity constant $\tau(\Gamma_{\text{aff}}, \Gamma_{\text{con}})$ will no longer cleanly collapse, requiring new techniques to measure the angle of intersection between the affine system and the coupled constraint faces.

    \item \textbf{Probabilistic Hoffman bounds.} While we have given deterministic upper bounds, the construction of probabilistic upper bounds on the Hoffman constant is the next step. These partially follow immediately from the literature on probabilistic bounds for singular values~\citep{wainwright2019highdimensional, vershynin2018highdimensional, rudelson2010singularvalues}. These bounds can either be asymptotic or nonasymptotic. However, these high-probability bounds tend to be dependent on the design matrix and its distribution.

    \item \textbf{Geometric-driven hyperparameter tuning.} We introduced a closed-form geometric criterion for selecting the elastic net hyperparameters to tighten the variational Hoffman upper bound. Apart from using this technique for other regularizers, future work could embed this theoretical boundary directly into automated hyperparameter search algorithms or homotopy path solvers, dynamically adjusting the regularization parameters to achieve a tighter convergence rate without sacrificing statistical validity. For example, in a Bayesian setting, given some prior on the distribution of the noise---or some posterior distribution based on empirical observations---the distribution of optimal hyperparameters can be computed.

    \item \textbf{Deterministic and totally unimodular designs.} While our statistical analysis heavily relies on the high-probability restricted conditions of random design matrices, the proposed geometric decomposition applies equally well to deterministic design matrices endowed with specific, highly structured properties---which often appear in linear programming and combinatorial optimization problems~\citep{guler1995approximations}. An example is the class of totally unimodular (TUM) matrices, which are important in linear programming. Because TUM matrices possess structure which allows their global Hoffman constants to be analyzed~\citep{hinder2024tum}, applying our localized decomposition to these deterministic LP settings could yield tight, explicit convergence rates without the need for probabilistic, distribution-dependent assumptions.
\end{itemize}

\section*{Software Availability} Code to reproduce all figures and numerical experiments in this paper is provided in the \texttt{experiments/hoffman/} directory of the following \href{https://github.com/farischaudhry/explicit-hoffman-constants}{GitHub repository}.

\section*{Acknowledgments}  A.M.~is supported by the EPSRC AI Hub on Mathematical Foundations of Intelligence: An ``Erlangen Programme'' for AI [EP/Y028872/1]. K.Y.~is supported by JSPS KAKENHI (24K15120, 24H00247, 26K02869, 26K02871).

\bibliographystyle{plainnat}
\bibliography{bibliography} 

\clearpage
\appendix

\section{Derivation of the Full LASSO Polyhedron}
\label{app:lasso-polyhedral}

To apply the restricted Hoffman framework developed in \cref{sec:explicit-hc}, the underlying solution mapping of the optimization problem must be polyhedral. We construct the constraint matrix of the LASSO KKT system to demonstrate its polyhedral geometry.

\subsection{Optimality Conditions}

Recall the LASSO objective $F(\beta) = f(\beta) + h(\beta)$, where $f(\beta) := \frac{1}{2n}\|A\beta - y\|_2^2$ and $h(\beta) := \lambda_n \|\beta\|_1$. First, note that both $f(\beta)$ and $h(\beta)$ are convex, and thus $F(\beta)$ is convex. For a convex function, the condition that the zero vector is in the subdifferential is both necessary and sufficient for global optimality. That is, $\hat{\beta}$ is a global minimizer of the LASSO estimator if and only if:
\begin{equation}
    0 \in \partial F(\hat{\beta}) = \nabla f(\hat{\beta}) + \partial h(\hat{\beta}).
\end{equation}
Now we compute the gradient of the smooth part and characterize the subgradient of the nonsmooth part:
\begin{equation}
    \nabla f(\beta) = \frac{1}{n} A^{\top} (A\beta - y) = \frac{1}{n}A^{\top} A \beta - \frac{1}{n}A^{\top} y,
\end{equation}
and $\partial h(\beta) = \lambda_n \partial\|\beta\|_1$. The optimality condition requires that there exists a specific subgradient vector $g \in \partial\|\beta\|_1$ such that the components of $g$ satisfy:
\begin{equation}
    \begin{cases} 
    g_j =\sign(\beta_j) & \text{if } \beta_j \neq 0, \\
    g_j \in  [-1, 1] & \text{if } \beta_j = 0.
    \end{cases}
\end{equation}

Substituting this specific vector $g$ into the optimality condition yields
\begin{equation}
    \frac{1}{n}A^{\top} A \beta - \frac{1}{n}A^{\top} y + \lambda_n g = 0.
\end{equation}
Let $G := \frac{1}{n}A^{\top} A \in \mathbb{R}^{d \times d}$ be the normalized Gram matrix and $u := \frac{1}{n}A^{\top} y \in \mathbb{R}^{d}$. Hence the KKT equality is given in its fixed form:
\begin{equation}
    \label{eq:kkt-equality-lasso}
    G\beta - u + \lambda_n g = 0.
\end{equation}

\subsection{The Local KKT System for a Fixed Active Set}

For a fixed active set $\mathcal{A}$, the unknowns are:
\begin{equation}
   \begin{pmatrix}
\beta_{\mathcal{A}} \\
g_{\mathcal{A}^c}
\end{pmatrix}, 
\end{equation}
since all other $\beta_{\mathcal{A}^c} = 0$ and $g_{\mathcal{A}} = \sign(\beta_\mathcal{A}) \in \{-1, 1\}$ are known.

Consider now that the KKT system in \eqref{eq:kkt-equality-lasso} can be partitioned as 
\begin{equation}
    \begin{pmatrix}
G_{\mathcal{A}\mathcal{A}} & G_{\mathcal{A} \mathcal{A}^c} \\
G_{\mathcal{A}^c \mathcal{A}} & G_{\mathcal{A}^c \mathcal{A}^c} 
\end{pmatrix}
\begin{pmatrix}
\beta_{\mathcal{A}} \\
\beta_{\mathcal{A}^c}
\end{pmatrix}
+ \lambda_n 
\begin{pmatrix}
g_{\mathcal{A}} \\
g_{\mathcal{A}^c}
\end{pmatrix}
=
\begin{pmatrix}
u_\mathcal{A} \\
u_{\mathcal{A}^c}
\end{pmatrix},
\end{equation}
which, since $\beta_{\mathcal{A}^c} = 0$ and the active signs $g_{\mathcal{A}}$ are fixed, can be further reduced to:
\begin{equation}
    \begin{pmatrix}
G_{\mathcal{A}\mathcal{A}} & 0 \\
G_{\mathcal{A}^c \mathcal{A}} & \lambda_n I
\end{pmatrix}
\begin{pmatrix}
\beta_{\mathcal{A}} \\
g_{\mathcal{A}^c}
\end{pmatrix}
=
\begin{pmatrix}
 u_{\mathcal{A}} -\lambda_n g_{\mathcal{A}} \\
u_{\mathcal{A}^c}
\end{pmatrix}.
\end{equation}
The first $2 \times 2$ block matrix is exactly the matrix denoted as $M_\mathcal{A}$ in \cref{sec:application-lasso}, and the right-hand side is denoted as $b_{\mathcal{A}}$. Notice this forms a linear system of the form $M_{\mathcal{A}} z = b_{\mathcal{A}}$, where $z := (\beta_{\mathcal{A}}^\top, g_{\mathcal{A}^c}^\top)^{\top}$.

The final full polyhedron of the system can now be written by separating the equality into two inequalities and appending the bounds on the variables:
\begin{equation}
\label{eq:full-kkt-inequality-polyhedron}
\begin{pmatrix}
G_{\mathcal{A}\mathcal{A}} & 0 \\
-G_{\mathcal{A}\mathcal{A}} & 0 \\
G_{\mathcal{A}^c \mathcal{A}} & \lambda_n I \\
-G_{\mathcal{A}^c \mathcal{A}} & -\lambda_n I \\
0 & I \\ 
0 & -I \\
-\operatorname{diag}(g_{\mathcal{A}}) & 0
\end{pmatrix}
\begin{pmatrix}
\beta_{\mathcal{A}} \\
g_{\mathcal{A}^c}
\end{pmatrix}
\leq 
\begin{pmatrix}
u_{\mathcal{A}} -\lambda_n g_{\mathcal{A}} \\
-u_{\mathcal{A}} + \lambda_n g_{\mathcal{A}} \\
u_{\mathcal{A}^c} \\ 
- u_{\mathcal{A}^c} \\
\mathbf{1}_{d-|\mathcal{A}|} \\ 
\mathbf{1}_{d-|\mathcal{A}|} \\
0_{\mathcal{A}}
\end{pmatrix}.
\end{equation}
We refer to this as the $\text{KKT}_{\mathcal{A}}$ system. The first four rows are inequalities defining the affine linear system $M_{\mathcal{A}}z = b_{\mathcal{A}}$. The $(0, I)$ and $(0, -I)$ rows bound the box constraints $g_j \in [-1, 1]$ for the inactive variables, and the diagonal row $-\operatorname{diag}(g_{\mathcal A})$ enforces the sign-consistency condition:
\begin{equation}
-\operatorname{diag}(g_{\mathcal A})\beta_{\mathcal A}\le 0
\quad\Longleftrightarrow\quad
g_j\beta_j\ge 0\quad\forall j\in\mathcal A,
\end{equation}
which (together with the active-set assumption $\beta_j\neq0$ for $j\in\mathcal A$) requires $g_j=\sign(\beta_j)$. Thus the block is both sufficient and necessary to impose sign agreement. As the exact intersection of these polyhedral box and sign inequalities and the affine space, the solution space of the restricted LASSO is a polyhedral set.

If the active set is empty ($\mathcal{A} = \emptyset$), which occurs for sufficiently a large regularization penalty $\lambda_n$, the unknowns consist entirely of the inactive dual sign variables $g_{\mathcal{A}^c}$. The block matrix $M_{\mathcal{A}}$ reduces exactly to its lower-right block, $\lambda_n I_{d \times d}$. Consequently, the local Euclidean Hoffman constant is simply $H^{2,2}(\emptyset) = 1/\lambda_n$. This confirms the intuition that when no features are selected, geometric stability is dictated by the regularization penalty.

\section{Tightness and Scaling Behavior of Hoffman Decompositions}
\label{app:empirical-tightness-scaling}

In this appendix, we numerically test the validity and tightness of the derived bounds for the local affine Hoffman constant across three geometries: Euclidean ($\ell_2$, \cref{bound:dimension-free-bound}), Manhattan ($\ell_1$, \cref{bound:l1-affine-hoffman}), and Chebyshev ($\ell_\infty$, \cref{cor:linf-affine-hoffman}).

\subsection{General Numerical Tightness}

To measure the exact Hoffman constant, we instantiate the KKT system $M_{\mathcal{A}}$ for a randomly sampled design $A$ and support $\mathcal{A}$. We compute the ground truth value as the induced operator norm of the matrix inverse, $\|M_{\mathcal{A}}^{-1}\|_p$, using standard numerical linear algebra routines. This value is compared against our explicit theoretical upper bounds in the form of a tightness ratio (theoretical bound / exact value) across 30 Monte Carlo trials for varying dimensions and sparsity levels. A ratio of $1.0$ indicates that the theoretical bound is an exact characterization of the true Hoffman constant. 

As shown in \cref{tab:bound-tightness}, the bounds are remarkably tight across all configurations. Of particular note is the $\ell_\infty$ bound, which remains tight up to numerical precision even for highly correlated designs. For orthogonal and moderately correlated design, both the $\ell_2$ and $\ell_1$ bounds are tight. For the highly correlated design, the $\ell_2$ bound is slightly loose (e.g., around $6\%$ on average for a certain configuration). It is known from the result of \eqref{eq:l2-hoffman-tight-lower-upper-bound} that the $\ell_2$ bound will not be loose by more than approximately $73.2\%$.

However, in this case, the components of the bounds such as $\|G_{\mathcal{A}\mathcal{A}}^{-1}\|_2$ and $\|G_{\mathcal{A}^c\mathcal{A}} G_{\mathcal{A}\mathcal{A}}^{-1}\|_2$ were computed exactly via standard numerical linear algebra regimes. In practice, some approximation error may be introduced when bounding these quantities (e.g., by an RE condition estimate, a submultiplicativity bound on the interaction, or bounding the spectral norm by the Frobenius norm).

\begin{table}[ht]
\centering
\caption{\textbf{Empirical tightness of affine Hoffman bounds in different norms.} The table reports the ratio of the theoretical upper bound to the exact empirical value across 30 Monte Carlo trials per configuration. In each Monte Carlo trial, a design and $s$-sparse support are sampled. Values are reported as median [min, max].}
\label{tab:bound-tightness}
\scriptsize
\begin{tabular}{lrrrccc}
\toprule
\textbf{Design} & \textbf{$d$} & \textbf{$n$} & \textbf{$s$} & \textbf{$H^{2,2}_{\mathcal{S}_{\mathcal{A}}}$ (Spectral)} & \textbf{$H^{\infty, \infty}_{\mathcal{S}_{\mathcal{A}}}$ (Chebyshev)} & \textbf{$H^{1,1}_{\mathcal{S}_{\mathcal{A}}}$ (Manhattan)} \\ 
\midrule
\textit{Gaussian} & 500 & 100 & 5  & 1.000 [1.000, 1.000] & 1.000 [1.000, 1.000] & 1.000 [1.000, 1.001] \\
                  & 500 & 200 & 20 & 1.001 [1.000, 1.001] & 1.000 [1.000, 1.000] & 1.001 [1.000, 1.003] \\
                  & 1000 & 100 & 10 & 1.000 [1.000, 1.000] & 1.000 [1.000, 1.000] & 1.000 [1.000, 1.001] \\
                  & 2000 & 400 & 40 & 1.000 [1.000, 1.000] & 1.000 [1.000, 1.000] & 1.000 [1.000, 1.001] \\
\midrule
\textit{Spiked ($\rho=0.5$)} & 500 & 100 & 10 & 1.001 [1.001, 1.002] & 1.000 [1.000, 1.000] & 1.001 [1.000, 1.004] \\
                    & 1000 & 200 & 20 & 1.001 [1.001, 1.002] & 1.000 [1.000, 1.000] & 1.001 [1.000, 1.002] \\
                    & 2000 & 400 & 40 & 1.001 [1.001, 1.002] & 1.000 [1.000, 1.000] & 1.001 [1.000, 1.003] \\
\midrule
\textit{Spiked ($\rho=0.8$)} & 500 & 100 & 5  & 1.003 [1.001, 1.004] & 1.000 [1.000, 1.000] & 1.001 [1.000, 1.004] \\
                    & 500 & 200 & 40 & 1.059 [1.036, 1.086] & 1.000 [1.000, 1.000] & 1.010 [1.000, 1.026] \\
                    & 1000 & 100 & 20 & 1.012 [1.008, 1.023] & 1.000 [1.000, 1.000] & 1.004 [1.000, 1.008] \\
                    & 1000 & 400 & 40 & 1.017 [1.013, 1.026] & 1.000 [1.000, 1.000] & 1.005 [1.000, 1.011] \\
                    & 2000 & 400 & 40 & 1.008 [1.007, 1.012] & 1.000 [1.000, 1.000] & 1.002 [1.000, 1.005] \\
\bottomrule
\end{tabular}
\end{table}

\subsection{Scaling Behavior with Ambient Dimension}

A second consideration is how the exact value of and bound on the Hoffman constant scales as the ambient dimension is increased by introducing unrelated features. We validate this behavior by fixing the sample size $n=200$ and sparsity $s=10$, while increasing the ambient dimension $d$ from 400 to 2000 by appending features to the design matrix while maintaining the original support.

As illustrated in \cref{fig:scaling-decomposition}, the exact values of $H^{2,2}$ and $H^{1,1}$ grow steadily as the number of inactive features increases. This growth occurs despite the underlying signal and its local curvature remaining unchanged. In contrast, the $H^{\infty, \infty}$ constant remains relatively flat. This suggests coordinate-wise solvers, which rely on the $\ell_\infty$ geometry for their convergence rates, should work well in datasets which a high feature count.

\begin{figure}[ht]
    \centering
    \includegraphics[width=\textwidth]{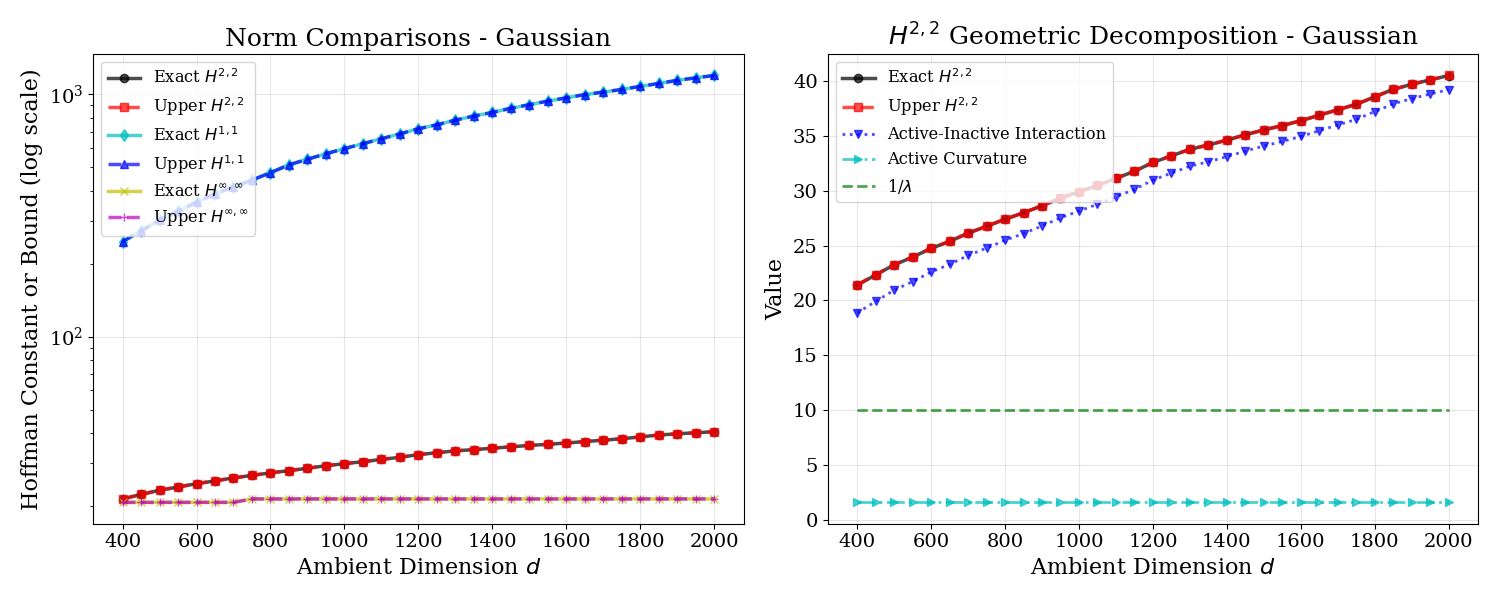}
    \caption{\textbf{Scaling of affine Hoffman constants and geometric decomposition with ambient dimension $d$ for a Gaussian design ($n=200, s=10$).} (Left) Comparison of exact constants and theoretical upper bounds for $H^{2,2}$, $H^{1,1}$, and $H^{\infty,\infty}$ as unrelated ambient features are added. Both the Euclidean and Manhattan measures scale with $d$, while the Chebyshev constant remains relatively dimension-independent. (Right) Decomposition of the $H^{2,2}$ bound components. The growth in the Hoffman constant is driven entirely by the active-inactive interaction norm $\|G_{\mathcal{A}^c\mathcal{A}}\|_2$, while the active curvature remains close to constant.}
    \label{fig:scaling-decomposition}
\end{figure}

Furthermore, the decomposition into the max bound in \eqref{eq:l2-hoffman-tight-lower-upper-bound} allows isolating the cause of this instability. The right plot of \cref{fig:scaling-decomposition} shows the constituent components of the affine Hoffman bound in the spectral case: the active curvature penalty $\|G_{\mathcal{AA}}^{-1}\|_2$, the interaction norm $\|G_{\mathcal{A}^c\mathcal{A}} G_{\mathcal{AA}}^{-1}\|_2$, and the inverse of the regularization parameter $\lambda_n$. The total growth in the Hoffman constant tracks to some multiple of the interaction norm.

\section{Proofs for Elastic Net Hyperparameter Selection}

This appendix provides the proofs for the geometric optimization of the elastic net hyperparameters introduced in \cref{subsubsec:elastic-net-hyperparam-selection}.

\subsection{Proof of \cref{lem:naive-optimal-elasticnet}}
\label{app:proof-naive-optimal-elasticnet}

\begin{proof}
    We aim to minimize the surrogate envelope $\mathcal{E}$ subject to the regularization budget constraint $\lambda_{n,1} + \lambda_{n,2} = \Lambda_n$. We parameterize the terms inside the envelope by $\lambda_{n,2}$:
    \begin{equation}
        T_1(\lambda_{n,2}) = \frac{1}{\lambda_{n,2}}, \quad T_2(\lambda_{n,2}) = \frac{1}{\Lambda_n - \lambda_{n,2}} \left( 1 + \frac{\|G_{\mathcal{A}^c \mathcal{A}}\|_2}{\lambda_{n,2}} \right).
    \end{equation}
    Observing the behavior of the terms on the valid interval $\lambda_{n,2} \in (0, \Lambda_n)$: $T_1(\lambda_{n,2})$ is strictly decreasing, and $T_2(\lambda_{n,2})$ is a convex U-shaped function diverging at both boundaries. 
        
    The envelope $\mathcal{E} = \max(T_1, T_2)$ is minimized either at the intersection of the two curves ($T_1 = T_2$) or at the local vertex of $T_2$. As established later in the general phase transition of \cref{thm:statistically-constrained-optimum}, the condition $\|G_{\mathcal{A}^c \mathcal{A}}\|_2 \le \Lambda_n/3$ guarantees that the vertex of $T_2$ occurs at a $\lambda_{n,2}$ value greater than the intersection point. Because $T_1$ is strictly decreasing, the global minimum of the upper envelope $\mathcal{E}$ is at the intersection of the two curves.
    
    Equating $T_1 = T_2$:
    \begin{equation}
        \frac{1}{\lambda_{n,2}} = \frac{1}{\lambda_{n,1}} \left( 1 + \frac{\|G_{\mathcal{A}^c \mathcal{A}}\|_2}{\lambda_{n,2}} \right) \implies \frac{\lambda_{n,1}}{\lambda_{n,2}} = \frac{\lambda_{n,2} + \|G_{\mathcal{A}^c \mathcal{A}}\|_2}{\lambda_{n,2}} \implies \lambda_{n,1} = \lambda_{n,2} + \|G_{\mathcal{A}^c \mathcal{A}}\|_2.
    \end{equation}
    Solving the linear system defined by this balance condition $\lambda_{n,1} - \lambda_{n,2} = \|G_{\mathcal{A}^c \mathcal{A}}\|_2$ and the budget constraint $\lambda_{n,1} + \lambda_{n,2} = \Lambda_n$ yields the explicit optimal parameters:
    \begin{equation}
        \lambda_{n,1}^* = \frac{\Lambda_n + \|G_{\mathcal{A}^c \mathcal{A}}\|_2}{2}, \quad \lambda_{n,2}^* = \frac{\Lambda_n - \|G_{\mathcal{A}^c \mathcal{A}}\|_2}{2}.
    \end{equation}
\end{proof}

\subsection{Proof of \cref{lem:naive-envelope-gap}}
\label{app:proof-naive-envelope-gap}

\begin{proof}
    We denote the components of the relaxed envelope as $T_1, T_2$ and the tight envelope as $T_1^{\text{tight}}, T_2^{\text{tight}}$. The pointwise differences are:
    \begin{align}
        T_1 - T_1^{\text{tight}} &= \frac{1}{\lambda_{n,2}} - \frac{1}{\lambda_{n,2} + \lambda_{\min}(G_{\mathcal{A}\mathcal{A}})} = \frac{\lambda_{\min}(G_{\mathcal{A}\mathcal{A}})}{\lambda_{n,2}(\lambda_{n,2} + \lambda_{\min}(G_{\mathcal{A}\mathcal{A}}))} \nonumber \\
        T_2 - T_2^{\text{tight}} &= \frac{1}{\lambda_{n,1}}\left(1 + \frac{\|G_{\mathcal{A}^c \mathcal{A}}\|_2}{\lambda_{n,2}}\right) - \frac{1}{\lambda_{n,1}}\left(1 + \frac{\|G_{\mathcal{A}^c \mathcal{A}}\|_2}{\lambda_{n,2} + \lambda_{\min}(G_{\mathcal{A}\mathcal{A}})}\right) \nonumber \\
        &= \frac{\|G_{\mathcal{A}^c \mathcal{A}}\|_2}{\lambda_{n,1}} \left( \frac{1}{\lambda_{n,2}} - \frac{1}{\lambda_{n,2} + \lambda_{\min}(G_{\mathcal{A}\mathcal{A}})} \right) = \frac{\|G_{\mathcal{A}^c \mathcal{A}}\|_2}{\lambda_{n,1}} \frac{\lambda_{\min}(G_{\mathcal{A}\mathcal{A}})}{\lambda_{n,2}(\lambda_{n,2} + \lambda_{\min}(G_{\mathcal{A}\mathcal{A}}))}.
    \end{align}
    By the nonexpansive property of the maximum operator, $\mathcal{E} - \mathcal{E}_{\text{tight}} \le \max(T_1 - T_1^{\text{tight}}, \quad T_2 - T_2^{\text{tight}})$. Factoring out the common fractional term yields the pointwise bound.
    
    For the optimal gap, recall from \cref{lem:naive-optimal-elasticnet} that in the weak interaction regime, the relaxed optimum occurs at $\lambda_{n,2}^* = (\Lambda_n - \|G_{\mathcal{A}^c \mathcal{A}}\|_2)/2$, yielding the minimum envelope value:
    \begin{equation}
        \mathcal{E}^* = \frac{1}{\lambda_{n,2}^*} = \frac{2}{\Lambda_n - \|G_{\mathcal{A}^c \mathcal{A}}\|_2}.
    \end{equation} 
    Similarly, the intersection of the tight surrogate occurs at $\lambda_{n,2}^{\text{tight}, *} = (\Lambda_n - \|G_{\mathcal{A}^c \mathcal{A}}\|_2 - \lambda_{\min}(G_{\mathcal{A}\mathcal{A}}))/2$. The minimum value of the tight envelope is therefore:
    \begin{equation}
        \mathcal{E}_{\text{tight}}^* = \frac{1}{\lambda_{n,2}^{\text{tight}, *} + \lambda_{\min}(G_{\mathcal{A}\mathcal{A}})} = \frac{2}{\Lambda_n - \|G_{\mathcal{A}^c \mathcal{A}}\|_2 + \lambda_{\min}(G_{\mathcal{A}\mathcal{A}})}.
    \end{equation}
    Subtracting the two optimal values yields
    \begin{align}
        \mathcal{E}^* - \mathcal{E}_{\text{tight}}^* &= \frac{2}{\Lambda_n - \|G_{\mathcal{A}^c \mathcal{A}}\|_2} - \frac{2}{\Lambda_n - \|G_{\mathcal{A}^c \mathcal{A}}\|_2 + \lambda_{\min}(G_{\mathcal{A}\mathcal{A}})} \nonumber \\
        &= \frac{2(\Lambda_n - \|G_{\mathcal{A}^c \mathcal{A}}\|_2 + \lambda_{\min}(G_{\mathcal{A}\mathcal{A}})) - 2(\Lambda_n - \|G_{\mathcal{A}^c \mathcal{A}}\|_2)}{(\Lambda_n - \|G_{\mathcal{A}^c \mathcal{A}}\|_2)(\Lambda_n - \|G_{\mathcal{A}^c \mathcal{A}}\|_2 + \lambda_{\min}(G_{\mathcal{A}\mathcal{A}}))} \nonumber \\
        &= \frac{2\lambda_{\min}(G_{\mathcal{A}\mathcal{A}})}{(\Lambda_n - \|G_{\mathcal{A}^c \mathcal{A}}\|_2)(\Lambda_n - \|G_{\mathcal{A}^c \mathcal{A}}\|_2 + \lambda_{\min}(G_{\mathcal{A}\mathcal{A}}))}.
    \end{align}
\end{proof}

\subsection{Proof of \cref{thm:statistically-constrained-optimum}}
\label{app:proof-statistically-constrained-optimum}

\begin{proof}
    Substituting the budget constraint $\lambda_{n,2}=\Lambda_n-\lambda_{n,1}$, the surrogate envelope evaluates to the upper envelope of two functions parameterized by $\lambda_{n,1}$ on the interval $(0,\Lambda_n)$:
    \begin{equation}
        \mathcal{E}(\lambda_{n,1})=\max\left(\underbrace{\frac{1}{\Lambda_n-\lambda_{n,1}}}_{T_1(\lambda_{n,1})},\quad\underbrace{\frac{\Lambda_n-\lambda_{n,1}+\|G_{\mathcal{A}^c\mathcal{A}}\|_2}{\lambda_{n,1}(\Lambda_n-\lambda_{n,1})}}_{T_2(\lambda_{n,1})}\right).
    \end{equation}
    We identify the critical points of this envelope. First, the intersection $T_1(\lambda_{n,1})=T_2(\lambda_{n,1})$ occurs at:
    \begin{equation}
        \frac{1}{\Lambda_n-\lambda_{n,1}}=\frac{\Lambda_n-\lambda_{n,1}+\|G_{\mathcal{A}^c\mathcal{A}}\|_2}{\lambda_{n,1}(\Lambda_n-\lambda_{n,1})}\implies\lambda_{\text{int}}=\frac{\Lambda_n+\|G_{\mathcal{A}^c\mathcal{A}}\|_2}{2}.
    \end{equation}
    Second, the global minimum of the convex curve $T_2(\lambda_{n,1})$ is found by setting its derivative to zero:
    \begin{equation}
        T_2'(\lambda_{n,1})=\frac{-1(\lambda_{n,1}(\Lambda_n-\lambda_{n,1})) - (\Lambda_n-\lambda_{n,1}+\|G_{\mathcal{A}^c\mathcal{A}}\|_2)(\Lambda_n-2\lambda_{n,1})}{(\lambda_{n,1}(\Lambda_n-\lambda_{n,1}))^2}=0.
    \end{equation}
    Setting the numerator to zero and collecting terms yields the quadratic equation:
    \begin{equation}
        \lambda_{n,1}^2-2(\Lambda_n+\|G_{\mathcal{A}^c\mathcal{A}}\|_2)\lambda_{n,1}+\Lambda_n(\Lambda_n+\|G_{\mathcal{A}^c\mathcal{A}}\|_2)=0.
    \end{equation}
    Taking the root that lies in the valid interval $(0,\Lambda_n)$ yields the vertex:
    \begin{equation}
        \lambda_{\text{vert}}=\Lambda_n+\|G_{\mathcal{A}^c\mathcal{A}}\|_2-\sqrt{\|G_{\mathcal{A}^c\mathcal{A}}\|_2(\Lambda_n+\|G_{\mathcal{A}^c\mathcal{A}}\|_2)}.
    \end{equation}
    
    The unconstrained minimum $\lambda_{n,1}^{\text{unc}}$ depends on whether the vertex $\lambda_{\text{vert}}$ occurs before or after the intersection $\lambda_{\text{int}}$. Setting $\lambda_{\text{int}}\le\lambda_{\text{vert}}$ simplifies algebraically to $\|G_{\mathcal{A}^c\mathcal{A}}\|_2\le\Lambda_n/3$. Thus, if $\|G_{\mathcal{A}^c\mathcal{A}}\|_2\le\Lambda_n/3$, the intersection curve dictates the minimum. If $\|G_{\mathcal{A}^c\mathcal{A}}\|_2>\Lambda_n/3$, the envelope is dominated by $T_2$, and the vertex dictates the minimum. This establishes the phase transition for the exact geometric optimum $\lambda_{n,1}^{\text{unc}}$. 
        
    Finally, because the upper envelope $\mathcal{E}(\lambda_{n,1})$ is a convex function that is strictly decreasing for $\lambda_{n,1}<\lambda_{n,1}^{\text{unc}}$ and strictly increasing for all $\lambda_{n,1}>\lambda_{n,1}^{\text{unc}}$, enforcing the feasible region $\lambda_{n,1}\ge\lambda_{\text{stat}}$ requires the optimum to be $\lambda_{n,1}^{\text{unc}}$ if the constraint is slack ($\lambda_{\text{stat}}\le\lambda_{n,1}^{\text{unc}}$), and on the boundary $\lambda_{\text{stat}}$ if the constraint binds ($\lambda_{\text{stat}}>\lambda_{n,1}^{\text{unc}}$). Therefore, the exact constrained minimum is $\lambda_{n,1}^*=\max(\lambda_{n,1}^{\text{unc}},\lambda_{\text{stat}})$.
\end{proof}

\subsection{Proof of \cref{lem:uncoupled-gap}}
\label{app:proof-uncoupled-gap}

\begin{proof}
    In the weak-interaction regime, the unconstrained minimizer of the relaxed envelope is
    \begin{equation}
        \lambda_{n,1}^{\mathrm{unc}}
        =
        \frac{\Lambda_n+\|G_{\mathcal A^c\mathcal A}\|_2}{2}.
    \end{equation}
    Since $\Xi>\lambda_{n,1}^{\mathrm{unc}}$, the uncoupled constraint is binding. Therefore,
    \begin{equation}
        \lambda_{n,1}^{\mathrm{uncoupled}}
        =
        \Xi,
        \qquad
        \lambda_{n,2}^{\mathrm{uncoupled}}
        =
        \Lambda_n-\Xi.
    \end{equation}
    Moreover,
    \begin{equation}
        \Xi
        >
        \frac{\Lambda_n+\|G_{\mathcal A^c\mathcal A}\|_2}{2}
        \quad\Longleftrightarrow\quad
        \Xi
        >
        \Lambda_n-\Xi+\|G_{\mathcal A^c\mathcal A}\|_2.
    \end{equation}
    Hence the first branch of the relaxed envelope dominates at the uncoupled optimum, and
    \begin{equation}
        \mathcal E_{\mathrm{uncoupled}}^*
        =
        \frac{1}{\Lambda_n-\Xi}.
    \end{equation}

    Under the coupled constraint, the budget identity gives
    \begin{equation}
        \Lambda_n-\lambda_{n,2}
        \geq
        \Xi+\Omega\lambda_{n,2}.
    \end{equation}
    Equivalently,
    \begin{equation}
        \lambda_{n,2}
        \leq
        \frac{\Lambda_n-\Xi}{1+\Omega}.
    \end{equation}
    Since the relaxed envelope is strictly increasing in
    $\lambda_{n,1}$ to the right of its unconstrained minimizer, its
    coupled minimum occurs at the boundary. Thus,
    \begin{equation}
        \lambda_{n,2}^{\mathrm{coupled}}
        =
        \frac{\Lambda_n-\Xi}{1+\Omega},
        \qquad
        \lambda_{n,1}^{\mathrm{coupled}}
        =
        \frac{\Xi+\Omega\Lambda_n}{1+\Omega}.
    \end{equation}

    We also have
    \begin{equation}
        \lambda_{n,1}^{\mathrm{coupled}}
        \geq
        \Xi
        >
        \Lambda_n-\Xi+\|G_{\mathcal A^c\mathcal A}\|_2
        \geq
        \lambda_{n,2}^{\mathrm{coupled}}+\|G_{\mathcal A^c\mathcal A}\|_2.
    \end{equation}
    It follows that
    \begin{equation}
        \mathcal E_{\mathrm{coupled}}^*
        =
        \frac{1}{\lambda_{n,2}^{\mathrm{coupled}}}
        =
        \frac{1+\Omega}{\Lambda_n-\Xi}.
    \end{equation}
    Subtracting the uncoupled value yields
    \begin{equation}
        \mathcal E_{\mathrm{coupled}}^*
        -
        \mathcal E_{\mathrm{uncoupled}}^*
        =
        \frac{\Omega}{\Lambda_n-\Xi},
    \end{equation}
    while division gives
    \begin{equation}
        \frac{
            \mathcal E_{\mathrm{coupled}}^*
        }{
            \mathcal E_{\mathrm{uncoupled}}^*
        }
        =
        1+\Omega.
    \end{equation}
\end{proof}

\end{document}